\documentclass[aop]{imsart}

\RequirePackage{amsthm,amsmath,amsfonts,amssymb}
\usepackage{dsfont,color,xfrac,enumitem,mathrsfs,bm}

\RequirePackage[numbers]{natbib}
\RequirePackage[colorlinks,citecolor=blue,urlcolor=blue]{hyperref}
\RequirePackage{graphicx}
\usepackage[textsize=tiny]{todonotes}

\startlocaldefs
\numberwithin{equation}{section}

\newtheorem{theorem}{Theorem}[section]

\newtheorem{thm}{Theorem}[section]
\newtheorem{lem}[thm]{Lemma}
\newtheorem{cor}[thm]{Corollary}
\newtheorem{prop}[thm]{Proposition}
\theoremstyle{remark}
\newtheorem{definition}[theorem]{Definition}

\newtheorem{defn}[thm]{Definition}
\newtheorem{rem}[thm]{Remark}
\newtheorem{obs}[thm]{Observation}

\newtheorem{ass}{Assumption}
\renewcommand\theass{\Roman{ass}}
\newtheorem{question}{Question}

\newtheorem{innercustomass}{Assumption}
\newenvironment{customass}[1]
 {\renewcommand\theinnercustomass{#1}\innercustomass}
 {\endinnercustomass}

\renewcommand{\le}{\leqslant} 
\renewcommand{\ge}{\geqslant} 
\renewcommand{\leq}{\leqslant}

\newcommand{\ra}{\rangle}
\newcommand{\la}{\langle}

\newcommand{\half}{\ensuremath{\sfrac12}} 
\newcommand{\ind}{\mathds{1}}
\newcommand{\eps}{\varepsilon}

\newcommand{\norm}[1]{\left\Vert#1\right\Vert}
\newcommand{\abs}[1]{\left\vert#1\right\vert}

\newcommand{\ie}{\emph{i.e.,}}

\def\qed{\hfill $\blacksquare$} 
\let\ga=\alpha  \let\gc=\gamma \let\gd=\delta 
     \let\gl=\lambda   \let\go=\omega   \let\gs=\sigma  
  \let\gz=\zeta
\let\gC=\Gamma \let\gD=\Delta  \let\gL=\Lambda 
   \let\gS=\Sigma \let\gU=\Upsilon 
 
\newcommand{\cA}{\mathcal{A}}\newcommand{\cC}{\mathcal{C}}
\newcommand{\cD}{\mathcal{D}}
\newcommand{\cG}{\mathcal{G}}\newcommand{\cH}{\mathcal{H}}
\newcommand{\cK}{\mathcal{K}}
\newcommand{\cM}{\mathcal{M}}
\newcommand{\cP}{\mathcal{P}}
\newcommand{\cS}{\mathcal{S}}\newcommand{\cT}{\mathcal{T}}\newcommand{\cU}{\mathcal{U}}
\newcommand{\cV}{\mathcal{V}}
 
\newcommand{\mvzero}{\boldsymbol{0}}

\newcommand{\mvR}{\boldsymbol{R}}
\newcommand{\mvU}{\boldsymbol{U}}
\newcommand{\mvW}{\boldsymbol{W}}\newcommand{\mvX}{\boldsymbol{X}}
\newcommand{\mvY}{\boldsymbol{Y}}\newcommand{\mvZ}{\boldsymbol{Z}}
\newcommand{\mvb}{\boldsymbol{b}}

\newcommand{\mvs}{\boldsymbol{s}}

\newcommand{\mvw}{\boldsymbol{w}}\newcommand{\mvx}{\boldsymbol{x}}

\newcommand{\mveps}{\boldsymbol{\eps}}
\newcommand{\mvgth}{\boldsymbol{\theta}}

\newcommand{\mvgo}{\boldsymbol{\omega}}

\newcommand{\dN}{\mathds{N}}

\newcommand{\dR}{\mathds{R}}

\newcommand{\dZ}{\mathds{Z}} 

\newcommand{\sS}{\mathscr{S}}

\DeclareMathOperator{\E}{\mathds{E}}
\DeclareMathOperator{\pr}{\mathds{P}}

\DeclareMathOperator{\var}{Var}
\DeclareMathOperator{\cov}{Cov}

\DeclareMathOperator{\hess}{Hess}
\DeclareMathOperator{\tr}{Tr} 
 
\DeclareMathOperator{\md}{mod} 
\DeclareMathOperator{\hs}{H.S.}
\DeclareMathOperator{\op}{op}

\newcommand{\DW}{\Delta\!\mvW}
\newcommand{\DY}{\Delta\!Y}
\providecommand{\1}{\mathds{1}} 
\providecommand{\dwas}{d_\mathcal{W}} 
\usepackage{comment} 
\usepackage{mathtools} 

\def\eqd{\,{\buildrel d \over =}\,}

\newcommand\Alpha{\mathrm{A}}
\newcommand\gA{\mathrm{A}}
\newcommand\ext{\mathrm{ext}}
\newcommand{\bgo}{\overline{\omega}}
\endlocaldefs

\begin{document}

\begin{frontmatter}
	\title{Stein's method for Conditional Central Limit Theorem}
	\runtitle{Stein's method for CCLT}

	\begin{aug}
		\author[A]{\fnms{Partha S.~} \snm{Dey}\ead[label=e1,mark]{psdey@illinois.edu}}
		\and
		\author[A]{\fnms{Grigory} \snm{Terlov}\ead[label=e2,mark]{gterlov2@illinois.edu}}
		\address[A]{Department of Mathematics,
			University of Illinois at Urbana-Champaign\\
			\printead{e1,e2}
		}
		\date{\today}

	\end{aug}

	\begin{abstract}
		In the seventies, Charles Stein revolutionized the way of proving the Central Limit Theorem by introducing a method that utilizes a characterization equation for Gaussian distribution. In the last fifty years, much research has been done to adapt and strengthen this method to a variety of different settings and other limiting distributions. However, it has not been yet extended to study conditional convergences. In this article, we develop a novel approach using Stein's method for exchangeable pairs to find a rate of convergence in the Conditional Central Limit Theorem of the form $(X_n\mid Y_n=k)$, where $(X_n, Y_n)$ are asymptotically jointly Gaussian, and extend this result to a multivariate version. We apply our general result to several concrete examples, including pattern count in a random binary sequence and subgraph count in Erd\H{o}s-R\'enyi random graph.
	\end{abstract}

	\begin{keyword}[class=MSC2020]
		\kwd[Primary ]{60F05}
		\kwd[, ]{60G50}
		\kwd[, ]{60B10}
		\kwd[; Secondary ]{05C80}
		\kwd[, ]{62E17}
	\end{keyword}

	\begin{keyword}
		\kwd{Stein's method}
		\kwd{Central Limit Theorem}
		\kwd{Rate of convergence}
		\kwd{Conditional Law}
		\kwd{Multivariate normal approximation}
	\end{keyword}

\end{frontmatter}
\setcounter{tocdepth}{2}
\tableofcontents

\section{Introduction}
One of the oldest and most outstanding results in Probability theory is the Central Limit Theorem (CLT), which in classical form states that the sum of i.i.d.~finite variance random variables, appropriately centered and scaled, converges to the standard normal distribution asymptotically. There are several generalizations and variations, such as Multivariate CLT, Martingale CLT, Local Limit Theorem (LLT), and Conditional CLT (CCLT), among others; however, the essence is the same as in the simplest classical CLT. Though the classical proof of CLT uses characteristic functions and their characterization of distributions, over the last century, various methods, such as the moment method and Lindeberg's technique, have been developed to prove CLT even in highly dependent structures. State of the art in establishing CLT and getting a convergence rate is Stein's method, which we will discuss in further detail in Section~\ref{ssec:Stein}. Researchers have applied Stein's method to prove and find a rate of convergence in Multivariate CLT~\cite{Barbour90, ChatterjeeMeckes08, BarbourRollin19, GoldsteinReinert05, GoldsteinRinott96, RinottRotar96, Rollin13}, Martingale CLT~\cite{Rollin18}, Local Limit theorem~\cite{BarbourRollin19, Rollin05}, and in other non-Gaussian limit theorems (see~\cite{Arras19, ChatterjeeFulmanRollin11, ChatterjeeShao11, Chen75, Meckes06} among many others). Stein's method has also been applied to prove concentration inequalities~\cite{Chatterjee07, ChatterjeeDey10, Goldstein11, Goldstein14}, moderate deviation results~\cite{ChenFang13}, and strong coupling~\cite{Chatterjee12}. For a more detailed overview of the topic, we refer to the books~\cite{ChenGoldstein11, DiaconisHolmes98} and the survey~\cite{Ross11}. However, to our knowledge, very little is known about proving and getting a convergence rate in CCLT of the form $X$ given $Y=y$ (see~\cite{Bulinski16, Dedecker02, Holst79, YuanWei14}), especially in structures with dependence.

The main focus of this article is to develop Stein's method for exchangeable pairs to prove and get an explicit rate of convergence in Conditional Central Limit Theorems of the form $X_n$ given $Y_n=k$. Our focus will be on the case where the random vector $(X_n, Y_n)$ converges to a multivariate normal distribution as $n$ tends to infinity, and $Y_n$ is a discrete random variable. We believe that this is the first application of Stein's method in proving CCLT and one of the first works that study the explicit rate of convergence in CCLT. In the rest of this section, we briefly discuss literature on CCLT and Stein's method. We present the main idea behind the exchangeable pair approach on which our result is based. We state our main result in full generality in Section~\ref{ssec:mainres} and extend it to the multivariate setting in Section~\ref{ssec:multstein}.

\subsection{Conditional Central Limit Theorem}

It is common in Probability theory and Statistics to study conditional convergence of random variables. In general, the study of conditional probabilities played an important role in shaping the field as we know it and related to notions such as Markov chains and martingales. Even though Conditional Central Limit Theorem (CCLT) has been studied in the last forty years, it has been mostly looked at on particular examples or under specific assumptions. The strongest result is known for the sum of a stationary sequence $\gS_{i=1}^n(X\circ T^i), $ ${n\ge1}$ conditioned on a non-decreasing filtration $\cM_i:= T^{-i}(\cM_0)$ with a bi-measurable probability measure-preserving map $T$. The necessary and sufficient conditions for such CCLT were obtained by Dedecker and Merlev\`ede in~\cite{Dedecker02}. In~\cite[Theorem 4]{Bolthausen80}, Bolthausen proved CCLT with explicit rate of convergence in Kolmogorov--Smirnov distance for a positively recurrent Markov chain with finite absolute third moment conditioned on the time of the $n^\textrm{th}$ return to $0$. This result was recently extended by Guo and Peterson~\cite[Theorem 4.2]{Peterson19} to a CCLT for sums of i.i.d.~sequence of random vectors $(X_i, Y_i)$ with $\E(|X_1|+|Y_1|)^{2+\delta}<\infty$ conditioned on $\sum Y_i=k$. Another general CCLT was proved by Holst in~\cite{Holst79}. In that work, he also considered an i.i.d.~sequence of random vectors $(X_i, Y_i)$. Assuming that $\sum Y_i$ is a sufficient statistic and a certain multivariate CLT holds, he derived CCLT for
$
	(\sum {X}_i\mid \sum {Y}_i=k)
$,
after appropriate centering and scaling. In all of the relevant results from~\cite{Bolthausen80, Peterson19, Holst79}, the
authors used variations of the method of characteristic functions and thus crucially relied on the independence among the random variables.
Other settings in the literature include CCLT for conditionally independent random variables, which reduces to the classical setting see~\cite{Bulinski16, YuanWei14}, among others.

All of the above results use independence in one way or the other, which could lead one to believe that asymptotic independence and joint convergence to a multivariate normal are sufficient for CCLT to hold. However, the following artificial, but still insightful, example shows that it is not the case. Let $X_n$ and $Y_n$ be centered Binomial$(n, 1/2)$ independent random variables. Define
\begin{align*}
	\bigl(\widehat{X}_n, \widehat{Y}_n\bigr)=
	\begin{cases}
		(X_n, Y_n) & \text{w.~p.} \quad 1-\ga_n \\
		(U, \gd_0) & \text{w.~p.} \quad \ga_n,
	\end{cases}
\end{align*}
where $ n^{-\half}\ll \ga_n\ll 1$, $\gd_0$ is Dirac measure at $0$, and $U$ has arbitrary distribution. By Local Limit Theorem we have $\pr(Y_n=0)\approx n^{-\half}\ll \ga_n$ and $(\widehat{X}_n\mid \widehat{Y}_n=0)\Rightarrow U$, even though $\widehat{X}_n$ and $\widehat{Y}_n$ are asymptotically independent and $(\widehat{X}_n, \widehat{Y}_n)$ converges to the two dimensional standard normal distribution. Therefore a more careful characterisation of dependency is needed to have even heuristic understanding of when to expect CCLT to hold in a general setting.

\subsection{Stein's Method}\label{ssec:Stein}

Over the last few decades, Stein's methods have become one of the essential tools to prove and get a rate of convergence in Central Limit Theorems for sums of dependent random variables. It was first introduced by Charles Stein in 1972~\cite{Stein72}, who combined Gaussian Integration by parts or ``Stein characterizing equation for standard normal distribution'' with certain ``noise robustness'' property, which is now called the exchangeable pair approach. This method can now be applied using a variety of approaches, namely exchangeable pairs, dependency graphs or local dependencies~\cite{ChenShao04, Rinott94}, size-bias~\cite{GoldsteinRinott96} and zero-bias couplings~\cite{GoldsteinReinert96}, Stein coupling~\cite{ChenRollin10}, and through Malliavin calculus~\cite{NourdinPeccati10} among others. The main underlying idea in Stein's methods for CLT is as follows:

\emph{A random variable $W$ is close to the standard normal distribution with respect to an appropriate metric if $\sup_{f\in\cD}\E(f'(W)-Wf(W))$ is small for an adequately chosen class of functions $\cD$ depending on the metric.}

Our work is built on the exchangeable pair approach, which we state here and refer to~\cite{ChenGoldstein11, DiaconisHolmes98, Ross11} for further details.

\begin{definition}
	Two random variables $W$ and $W'$ are said to be exchangeable if $(W, W')\eqd(W', W)$.
\end{definition}
While most works on Stein's method using exchangeable pairs require exchangeability, this assumption can usually be relaxed to the requirement of $W$ and $W'$ to be equidistributed as pointed out in~\cite{Rollin06}. This remark applies to the majority of our work as well. The only place where we use exchangeability to the fullest extent is in the last step in the proof of Theorem~\ref{th: imp symcase} and the analogous place in the multivariate result in Theorem~\ref{th: multi case}.

The method of exchangeable pairs is usually applicable in systems where small perturbations do not change the distribution significantly. Classical results due to Stein~\cite{Stein72, Stein86} can be stated in the following way.
\begin{thm}\label{th:st exch}
	Let $(W, W')$ be an exchangeable pair of random variables defined on the same probability space. Suppose $\E W=0$, $\E W^2=1$, $\E \abs{W}^3<\infty $, and $\gD W :=W'-W$ almost surely satisfies
	\begin{align*}
		\E (\gD W\mid W)=-\gl(W+R_1) \text{ and } \E \left(\gD W^2\mid W\right)=2\gl(1+R_2)
	\end{align*}
	for some constant $\gl\in(0, 1)$ and random variables $R_i=R_i(W)$ for $i=1, 2$. Then
	\begin{align*}
		\dwas(W, Z) \le \E|R_1|+\sqrt{\frac{2}{\pi}}\E |R_2|+\frac{1}{3\gl}\E|\gD W|^3,
	\end{align*}
	where $Z$ is a standard normal random variable and $\dwas(W, Z)$ denotes the Wasserstein distance.
\end{thm}
The outline of the proof is usually of the following form. One first finds an antisymmetric function, then using Taylor expansion and the properties of the conditional expectation one derives a bound on the Stein operator. In particular, given a bounded twice differentiable function $f$, one can find a function $F$ with $F'=f$. Then $\E(F(W')-F(W))=0$ and expanding the expression inside of the parenthesis around $W$, one gets
\begin{align*}
	F(W')-F(W) \approx \gD W\cdot f(W) +\frac12 (\gD W)^{2}\cdot f'(W) + \frac16(\gD W)^3\cdot f''(W). \end{align*}
Using the tower property of the conditional expectation together with the assumptions on $\E(\gD W\mid W)$ and $\E(\gD W^2\mid W)$ one can derive the bound on $\abs{\E (W f(W)-f'(W))}.$

To derive the rate of convergence one aims to bound $\sup_{h\in\cC}|\E h(W)-\E h(Z)|$ for the corresponding class of functions $\cC$ and standard normal random variable $Z$. Given a function $h$, let $f$ be the ``bounded" solution to the following differential equation
\begin{align}\label{eq:stein de}
	f'(w) - w f(w)=h(w)-\E h(Z).
\end{align}
The core idea behind Stein's method is to work with the expectation of the left hand side in order to derive the desired bound on the distance. By inverting Ornstein-Uhlenbeck operator one can see that $f$ has one more derivative than $h$ and that $\norm{f}_{\infty} \le \norm{h'}_{\infty}, $ $\norm{f'}_{\infty} \le\sqrt{2/\pi} \norm{h'}_{\infty}, $ and $\norm{f''}_{\infty} \le 2 \norm{h'}_{\infty}$.

\subsection{Multivariate Stein's method via exchangeable pairs}\label{ssec:MultiStein}

The definition of exchangeability can be extended to a multidimensional case in a natural way; we say that a pair of random vectors $(\mvW, \mvW')$ is exchangeable if $(\mvW, \mvW')\eqd(\mvW', \mvW)$. Multivariate versions of Stein's method first appeared in~\cite{Barbour90, Gotze91}. It was extended to exchangeable pairs more than a decade later in~\cite{ChatterjeeMeckes08} under the assumptions
\begin{align*}
	\E( \DW\mid \mvW)=-\gl\left( \mvW+\mvR\right) \text{ and } \E\left( \DW\DW^T\mid \mvW\right)=2\gl(I+\gC),
\end{align*}
where $\DW:=\mvW'-\mvW$, $\gl\in(0, 1)$, $I$ is the identity matrix, and $\gC=\gC(\mvW)$ is a random matrix.
It was later extended in~\cite{ReinertRollin09} to the case of a general covariance matrix $\gS$.

For the multivariate case one can derive an analogous relation to~\eqref{eq:stein de}. Let $h$ be a $1$-Lipschitz function on $\dR^d$, \ie
\[
	\sup_{\mvw_1, \mvw_2\in\dR^d}\frac{\abs{h(\mvw_1)-h(\mvw_2)}}{\abs{\mvw_1-\mvw_2}}\le 1.
\]
Here, $\abs{\cdot}$ denotes the Euclidean norm.
If $\gS$ is a $d\times d$ symmetric and positive definite matrix, there is a ``bounded'' solution $f$ to the equation
\begin{align}\label{eq:stein de 2}
	\sS f(\mvw) =h(\mvw)-\E h(\Sigma^{\half}\mvZ),
\end{align}
where the operator
\begin{align}\label{def:mstein}
	\sS f(\mvx) := \la \gS, \hess f(\mvx)\ra_{\hs} - \la \mvx, \nabla f(\mvx)\ra
\end{align}
defined for two times continuously differentiable function $f$ characterizes $d$-dimensional normal distribution with mean zero and variance-covariance matrix $\gS$.

Both~\cite{ChatterjeeMeckes08, ReinertRollin09} bound the rate of convergence to the appropriate multivariate Gaussian vector in terms of smooth function metrics (also see~\cite{Meckes09} for a unified approach). A decade later~\cite{Raic19} presented a smoothing scheme that allows one to go from a smooth test function to Lipschitz ones. These approaches of~\cite{ChatterjeeMeckes08, ReinertRollin09, Raic19} were recently combined and extended in~\cite{FangKoike22}, under additional assumption on the finiteness of the fourth moment, to apply Stein's method for exchangeable pairs under Wasserstein distance. The result from~\cite{FangKoike22} can be stated in the following form.

\begin{thm}[{\cite[Theorem 2.1]{FangKoike22}}]\label{thm:FangKoike22}
	Let $(\mvW, \mvW')$ be an exchangeable pair of $d$-dimensional random vectors satisfying the linearity assumption
	\[
		\E(\DW\mid \mvW)=-\gL(\mvW+\mvR),
	\]
	where $\DW:=\mvW'-\mvW$. Assume that $\E\abs{\mvW}^4<\infty$. Let $\gS$ be a $d\times d$ positive definite symmetric matrix such that
	\begin{align}\label{eq:secondmoment}
		\E(\DW\DW^T\mid \mvW)=2\gL(\gS+\gC)
	\end{align}
	for some symmetric matrix $\gC$.
	Then
	\begin{align*}
		\dwas(\mvW, \gS^{\half}\mvZ) &
		=\sup_{h \text{ $1$-Lipschitz}}\abs{\E h(\mvW)-\E h(\gS^{1/2}\mvZ)}                                                                                                                                                \\
		                             & \leq \E\abs{\mvR}+\norm{\gS^{-\half}}_{\op}\E\norm{\gC}_{\hs}                                                                                                                       \\
		                             & \quad+\norm{\gS^{-\half}}_{\op}^{3/2}\left(\frac{\pi}{8}\right)^{\frac14}\max\left(\E\abs{\mvW}^2, \tr(\gS)\right)^{\frac14}\sqrt{\E\left(\abs{\gL^{-1}\DW}\cdot\abs{\DW}^3\right)}
	\end{align*}
	where $Z$ is a $d$-dimensional standard normal vector, $\abs{\cdot}$ denotes the Euclidean $\ell_{2}$ norm, $\norm{\cdot}_{\op}$ denotes the operator norm and $\norm{\cdot}_{\hs}$ denotes the Hilbert-Schmidt norm.

	In particular, if $\E\mvW\mvW^{T}=\gS=I_d$ the error bound reduces to
	\begin{align}
		\dwas(\mvW, \mvZ) & \leq \E\abs{\mvR}+\E\norm{\gC}_{\hs}+\left(\frac{\pi}{8}\right)^{\frac14}d^{\frac14}\sqrt{\E\left(\abs{\gL^{-1}\DW}\cdot\abs{\DW}^3\right)}.
	\end{align}
\end{thm}

\subsection{Local Limit Theorem}\label{ssec:LLT}

While Central Limit Theorem provides the global limiting behavior for the distribution of the scaled sum of random variables, the Local Limit Theorem provides the behavior for the probabilities of the scaled sum of random variables to be equal to a particular number. The classical case for independent lattice distributed random variables is well understood (see~\cite{Durrett10, Gnedenko54}, among others) and can be stated in the following way.

\begin{thm}\label{thm:cllt}
	Let $X_1, X_2, \ldots$ be i.i.d.~random variables with mean zero, variance $\sigma^2$ and having a common lattice distribution with span $1$, \ie~$\pr(X_i\in \gz+\dZ)=1$ for some $\gz\in \dR$. Let $W_n= n^{-\half}\sum_{i=1}^nX_i$ and $\mathcal{L}_n:=\left\{n^{-1/2}(n\gz+\dZ)\right\}$, then as $n\to \infty$, we have
	\begin{align*}
		\sup_{x\in\mathcal{L}_n}\left|\sqrt{n}\pr\left(W_n=x\right)-\varphi_{\gs^2}(x)\right|\to 0,
	\end{align*}
	where is the density of the normal distribution with mean $0$ and variance $\gs^2$ given by
	\[
		\varphi_{\gs^2}(x)=(2\pi \sigma^2)^{-\half}e^{-{x^2}/{2\sigma^2}}.
	\]
\end{thm}
The analogous result for integer-valued dependent random variables was introduced in~\cite{McDonald79} under certain `smoothness' conditions on the distribution. This approach was extended and connected to Stein's method in~\cite{BarbourRollin19, Rollin05, RollinRoss15}.
The main result of~\cite{BarbourRollin19} is for general Stein coupling, which combines several approaches of Stein's method. In this article, we focus on the exchangeable pair approach with an additional assumption that the change in the variable of interest takes values in $\{-1, 0, 1\}$. Hence we present in Theorem~\ref{th:steinllt} a simplified version of the main result in~\cite[Theorem 2.1]{BarbourRollin19} in combination with~\cite[Remark 2.2]{BarbourRollin19}. We refer to~\cite[Theorem 2.1]{BarbourRollin19} for the general statement.

\begin{thm}[{\cite[Theorem 2.1 and Remark 2.2]{BarbourRollin19}}]\label{th:steinllt}
	Let $Y$ be a mean zero random variable with variance $\gs^2$ and taking values in $\{\gz\}+\dZ$ for some $\gz\in [0, 1)$. Assume that $(Y', Y)$ is an exchangeable pair satisfying the linearity condition with some $\gl\in(0, 1)$,
	\begin{equation}\label{eq:lin}
		\E(\DY \mid Y)=-\gl(Y+R_1),
	\end{equation}
	and $\DY \in \{-1, 0, 1\}$ almost surely. Then
	\begin{align}\label{eq:steinllt}
		\sup_{k\in\{\gz\}+\dZ}\abs{\gs\pr(Y=k)-\gs\varphi_{\gs^2}(k)} 
		& \leq\frac{C}{\gs}+\frac{\sqrt{\E R_1^2}}{\gs}\left(2+\frac{1}{\sqrt{2e}}+\gs\sup_{k}\pr(Y=k)\right) \\
		& +\frac{\E R_2}{\gs^2\sqrt{2e}}+\frac{\E (R_2\abs{Y})}{\gs^3}+\frac{2+\sup_{k}\E (R_2\1_{Y=k})}{\gs},
	\end{align}
	where $C$ is some universal positive constant and $R_2:=\frac{1}{\gl}\abs{\pr(\DY =1\mid Y)-\pr(\DY =1)}$.
\end{thm}
Notice that the original result in~\cite{BarbourRollin19} is concerned with bounding the local distance to a translated Poisson random variable. However, ~\cite[Lemma 1.1]{BarbourRollin19} yields that on the set of integers translated Poisson probability is within $\frac{C}{\gs^2}$ of the discretized normal probability with the same mean and variance. Hence the same holds after shifting the lattice by $-\E Y$; this additional step in the approximation is accounted for in the first error term in~\eqref{eq:steinllt} with a universal constant $C$.

\subsection{Strategy}

The main strategy behind our result is similar to the one outlined under Theorem~\ref{th:st exch}, however we need to work with a bivariate function that is anti-symmetric in the first coordinate and symmetric in the second one. We need continuous approximation in terms of $(w', w)$, but discrete approximation in terms of $(y', y)$. More precisely, we will consider functions of the following form
\begin{align}
	(F(w')-F(w))\cdot G(y', y)
\end{align}
where $G$ is a symmetric function of two variables. Note that, similar to the classical exchangeable pair approach, with $f=F'$ we have
\begin{align*} F(w')-F(w) \approx \gD w\cdot f(w) + \frac12\gD w^2\cdot f'(w) + \frac16 \gD w^{3}\cdot f''(w). \end{align*}
We can choose $G$ depending on the behavior of the exchangeable pair. We consider two cases:\\

\noindent\textbf{Case I.} $\pr(\gD W\neq 0, \DY =0)>0$. In this case we can take
\begin{align*}
	G(y', y):= \frac12(\1_{y=k} + \1_{y'=k}) \cdot \1_{y'-y=0} = \1_{y'=y=k}.
\end{align*}

Here we need
\begin{align*}
	\E(\gD W\cdot \1_{Y'=Y=k}\mid W, Y)     & = -\gl \cdot (W + \text{``error''})\cdot \1_{Y=k} \text{ and } \\
	\E(\gD W^{2}\cdot \1_{Y'=Y=k}\mid W, Y) & = 2\gl \cdot (\gs_W^2 + \text{``error''})\cdot \1_{Y=k}
\end{align*}
with some $\gl\in(0, 1)$ for the classical idea to work.

\begin{rem}[Classical techniques]\label{rem:stand}
	If the model falls into \text{Case I}, then conditioning on this event, one can derive CCLT with the explicit rate of convergence using standard techniques as in Theorem~\ref{th:st exch}. We apply this method in detail to various examples in Section~\ref{ssec:applications of classic}. In models where $\pr(X'\neq X\mid Y'=Y)=0$ (Case II described below) one can usually consider $\bigl(\widehat{X}', \widehat{Y}'\bigr)$ that is a result of the same Markov chain started at $(X, Y)$ after two steps it then would fall into \text{Case I}. However, the computations get significantly more complicated and often are not feasible in applications.
\end{rem}

\noindent\textbf{Case II.} $\pr(\gD W\neq 0, \DY =0)=0$. In this case we can take
\begin{align*}
	G(y', y):= g(y'-k)\cdot \1_{y'-y=1} + g(y-k)\cdot \1_{y'-y=-1}
\end{align*}
for some function $g$. Details are given in Section~\ref{sec:proofs}.
To simplify notations we define for $\ell\ge 0$,
\begin{align}\label{eq:Mell}
	M_{\ell, \pm}(W, Y):=\E((\gD W)^\ell\cdot \1_{\DY = \pm1} \mid W, Y).
\end{align}

Using Taylor series expansion for the following mean zero random variable
\begin{align}\label{eq:theta}
	\begin{split}
		& \Theta_f(W, Y) \\
		& \quad := \E\left( (F(W')-F(W)) \cdot (g(Y'-k)\cdot \1_{\DY =1} +g(Y-k)\cdot \1_{\DY =-1})\mid W, Y\right),
	\end{split}
\end{align}
we get
\begin{align*}
	 & \E(\ f(W) (M_{1+} (W, Y)\cdot g(Y-k+1) +M_{1-}(W, Y)\cdot g(Y-k))\ )                                  \\
	 & \qquad\qquad+ \frac12 \E(\ f'(W) (M_{2+}(W, Y)\cdot g(Y-k+1)+M_{2-}(W, Y)\cdot g(Y-k)) \ ) \approx 0.
\end{align*}
If we have
\begin{align}\label{eq:M1M2}
	M_{1, \pm}(W, Y)\approx -\frac12\gl W \quad \text{ and }\quad M_{2, \pm}(W, Y)\approx \gl \gs_W^2 ,
\end{align}
then to get the Stein characterizing equation for $(W\mid Y=k)$,
\begin{align*}
	\abs{\E\left((Wf(W)-\gs_W^2 f'(W)) \cdot \1_{Y=k}\right)} \ll \pr(Y=k)
\end{align*}
we
need a function $g$ such that $g(y+1) + g(y)$ is a constant multiple of $\ind_{y=0}$. Such a function is given by
\begin{align*}
	g(y):=(-1)^{y}\cdot \1_{y\le 0}.
\end{align*}

In many examples we have $M_{1, \pm}(W, Y)\approx -a_{\pm}\cdot \gl W$ where $a_{+}\neq a_{-}$ are fixed positive constants, even though $M_{2, \pm}(W, Y)\approx \gl \gs_W^2 $. Thus we need a way to remove the asymmetry in the two conditional means $M_{1, \pm}$. Similarly, even though we can make $W, Y$ uncorrelated by subtracting an appropriate multiple of $Y$ from $W$, in some examples, $M_{1, \pm}$ involves non-trivial linear terms with $Y$. We will subtract an appropriate ``small'' random variable from $W$, which will not change variance behavior but introduce symmetry and remove $Y$ dependence from the conditional mean computation. The change of variable is explained in Proposition~\ref{prop: change} in Section~\ref{ssec:mainres} and applied in several concrete examples in Section~\ref{ssec:applications of main thm}. The above setup can be generalized to high dimensions as given in Section~\ref{ssec:multstein}.

\begin{rem}
	Note that, in general, the function
	\begin{align*}
		(f(W', Y')-f(W, Y'))\cdot G(y, y') + (f(W', Y)-f(W, Y))\cdot G(y', y)
	\end{align*}
	is anti-symmetric in $(W', Y'), (W, Y)$ for any $f$ and $G$. So one can use different $f$ and $G$ to get the Stein operator for $W$ multiplied by $\1_{Y=k}$.
\end{rem}

For such choices of function $g(y)$, one can get CCLT for $(W\mid Y=k)$. However, for an effective bound, one needs error terms with $\1_{Y< k}$ to be small in comparison with $\gs_Y^{-1}$ as one can see in Theorem~\ref{th: symcase}. To avoid this issue we first consider the same function as in \text{Case I}, \ie~$g(y)=\1_{y=0}$, then the argument outlined in \text{Case II} yields a CCLT for $(W\mid Y\in\{k-1, k\})$. Thus it remains to compare $h(W)\1_{Y=k-1}$ with $h(W)\1_{Y=k}$ for appropriate smooth function $h(w)$. Using the intuition that exchangeable pair is most applicable in models where small perturbation does not change the system too much, we upper bound the difference between these two quantities by error terms involving $\E\abs{W'-W}$. We present details of this argument in the proof of Theorem~\ref{th: imp symcase}.

We now present a toy example to illustrate the ideas mentioned above.

\subsection{Toy Example}

Let $(\xi_i, \go_i)_{i\ge 1}$ be independent random vectors with $\xi_i-\eps_i\go_i$ independent with $\go_i$ for all $i\ge 1$ for some sequence of real numbers $(\eps_i)_{i\ge 1}$. We assume that $\go_i$'s are i.i.d.~Bernoulli$(p)$ and $X_i:=\xi_i-\eps_i(\go_i-p)$'s are i.i.d.~with mean zero variance one. Define
\begin{align*}
	\bgo_i:=\go_i-p, \quad
	Y=\sum_{i=1}^n\bgo_i\quad\text{ and }\quad W:=\sum_{i=1}^{n}\xi_i = \sum_{i=1}^n X_i + \sum_{i=1}^n\eps_i\bgo_i.
\end{align*}

We will assume that $\sum_{i=1}^{n}\eps_i=0$ and $\eps_{\max}:=\max_{1\le i\le n}|\eps_i|\ll 1$. It is easy to check that the random vector $(W/\sqrt{n}, Y/\sqrt{np q})$ asymptotically converges to independent standard Gaussian rvs. We want to prove a CCLT for $(W\mid Y=k)$ when $|k|\ll n^{\half}$.

We consider the exchangeable pair created by independent re-sampling at a randomly chosen coordinate from $\{1, 2, \ldots, n\}$.
We have,
\begin{align*}
	M_{1, +}(W, Y)
	 & = \frac{1}{n}\sum_{i=1}^n \E( (\xi'_i-\xi_i)\cdot (1-\go_i)\go'_i \mid W, Y)                                                             \\
	 & = -\frac{p}{n}\sum_{i=1}^n\E( (X_i+\eps_i)(q-\bgo_i) \mid W, Y)                                                                          \\
	 & = -\frac{pq}{n}\left(W - \frac1q \sum_{i=1}^n \E((X_i+(1+q)\eps_i)\bgo_i \mid W, Y) \right)                                              \\
	 & = -\frac{pq}{n}\left(W - \frac1q \E\left(\frac{Y}{n} \cdot \sum_{i=1}^nX_i+(1+q) \sum_{i=1}^n\eps_i\bgo_i \ \bigl|\ W, Y\right) \right).
\end{align*}
Similarly, we get
\begin{align*}
	M_{1, -}(W, Y)
	 & = -\frac{pq}{n}\left(W + \frac1p \E\left(\frac{Y}{n} \cdot \sum_{i=1}^nX_i - (1+p) \sum_{i=1}^n\eps_i\bgo_i \ \bigl|\ W, Y\right) \right), \\
	M_{2, +}(W, Y)
	 & = \frac{pq}{n}\left(2(n-Y/q)+ \sum_{i=1}^n \E( (X_i^2-1+\eps_i^2 - 2\eps_i X_i )\cdot (1-\bgo_i/q) \mid W, Y) \right),                     \\
	\text{ and }
	M_{2, -}(W, Y)
	 & = \frac{pq}{n}\left(2(n+Y/p)+ \sum_{i=1}^n \E( (X_i^2-1+\eps_i^2 + 2\eps_i X_i )\cdot (1+\bgo_i/p) \mid W, Y) \right).
\end{align*}
One can easily verify that $\gs_W^2=\var(W)= n+o(n)$. Moreover, with $\gl=2pq/n$, we have
\[
	M_{1, \pm}(W, Y)\approx -\frac12\gl\cdot (W+o_p(\sqrt{n}))\quad \text{ and }\quad M_{2, \pm}(W, Y)\approx \gl\cdot (\gs_W^2+O_p(\sqrt{n})),
\]
as expected in~\eqref{eq:M1M2}.
This allows us to apply the idea in Case II to get an explicit rate of convergence for the CCLT.

We now write down all of the assumptions needed to state the main results.

\subsection{Assumptions}\label{ssec:ass}

It is natural to expect CCLT to hold for $(W\mid Y=k)$ under reasonable structural assumptions when $W$ and $Y$ are asymptotically jointly Gaussian. Since for jointly Gaussian random variables uncorrelated implies independent the first assumption that we impose on $(W, Y)$ is the following.
\begin{ass}\label{ass:uncorr and exch}
	Assume that
	\begin{enumerate}
		\item[I.1] $\mvW=(W_1, W_2, \dots, W_d)$, where $\mvW$ is a mean $0$ random vector with variance-covariance matrix $\gS$, which is invertible.  In the one dimensional case we will assume the variance to be $1$.
		\item[I.2] $Y$ is a mean $0$ random variable with variance $\sigma_Y^2$.
		\item[I.3] $\mvW$ and $Y$ are uncorrelated.
		\item[I.4] The random vectors $(\mvW, Y)$ and $(\mvW', Y')$ are exchangeable.
	\end{enumerate}
\end{ass}

In this paper, for simplicity, we focus on the case when $\DY :=Y'-Y\in\{-1, 0, 1\}$ and the exchangeable pair approach is applicable, which we state in the following assumptions.
Recall the definition of $M_{\ell, \pm}(\mvW, Y)$ from~\eqref{eq:Mell}.
\begin{ass}\label{ass:Y}
	$Y$ takes values in $\zeta+\dZ$ for some $\zeta\in [0, 1)$, $\DY \in\{-1, 0, 1\}$ almost surely, and
	\begin{align*}
		M_{0, \pm}(\mvW, Y)=Q+R_{0, \pm},
	\end{align*}
	where $Q=\gl\gs_Y^2$ for some $\gl\in(0, 1)$.
	Further, for any $k$ such that $\pr(Y=k)>0$ and $\pr(Y=k-1)>0$ we define
	\begin{equation}\label{eq:ratioassumption}
		r_{k}:= \frac{\pr(Y=k-1)}{\pr(Y=k)}\in(0, \infty).
	\end{equation}
\end{ass}

The condition on $M_{0, \pm}$ is natural because, in order to apply our method to $(\mvW\mid Y=k)$, we require the method of exchangeable pairs to be applicable to $Y$ on its own. In other words, we expect $\E(\gD Y\mid \mvW, Y)\approx -\gl Y$. Since we assume that $\gD Y\in\{-1, 0, 1\}$, if $\pr(\gD Y=1\mid \mvW, Y)$ is concentrated at $Q$, $\pr(\gD Y=-1\mid \mvW, Y)$ has to also concentrate at $Q$. This also implies that $Q\approx\gl\gs_Y^2$ as
\begin{align*}
	2\gl\gs_Y^2 & =\E(\E(\gD Y^2\mid \mvW, Y))                                       \\
	            & =\E(\pr(\gD Y=1\mid \mvW, Y)+\pr(\gD Y=-1\mid \mvW, Y))\approx 2Q.
\end{align*}
It is important to highlight that $M_{0, \pm}(\mvW, Y)=\pr(\DY =\pm1\mid \mvW, Y)$, but it is reasonable to expect that the extra conditioning on $\mvW$ does not affect this condition too much. In most examples that we consider, $Y$ is the sum of independent Bernoulli random variables and $M_{0, \pm}(\mvW, Y)\approx Q\mp\gl a_\pm Y$ for some $a_\pm$ such that $a_++a_-=1$.

\begin{obs}
	Assumption~\ref{ass:Y} implies that $Y$ satisfies the linearity condition $$\E(\DY \mid Y)=-\gl(Y+R), $$ with $R=-Y-\frac{1}{\gl}(R_{0, +}-R_{0, -})$.
	Moreover, if $R_{0, \pm}=\mp\gl a_\pm Y$, where $a_++a_-=1$ like in the case when $Y$ is the sum of i.i.d. Bernoulli$(p)$ random variables, $R$ simplifies to
	$$
		R=-Y-\frac{1}{\gl}(-\gl a_+ Y-\gl a_- Y)=-Y+(a_++a_-)Y=0.
	$$
\end{obs}

\begin{rem}[Ratio of probabilities and LLT]\label{rem:llt}
	Let $p_k(n):=\pr(Y=k)$ for $Y=Y_{n}$. Suppose the random variable $Y$ satisfies the LLT (see Theorems~\ref{thm:cllt} and~\ref{th:steinllt}) in the sense that as $n\to\infty$
	\begin{equation}\label{eq:lltrem}
		\sup_{k}\abs{\gs_Y\pr(Y=k)-\gs_Y\varphi_{\gs_Y^2}(k)}\to 0.
	\end{equation}
	Then for any $\abs{k}\ll~\gs_Y$ we have that
	$$\lim_{n\to\infty}\frac{p_{k-1}(n)}{p_k(n)}=1, $$
	see Lemma~\ref{lem:lltratio} for a proof.
	Thus this ratio is uniformly bounded from above and from below by positive constants that depend only on the law of $Y$ and the value of $k$. In particular, condition \eqref{eq:ratioassumption} of Assumption~\ref{ass:Y} is satisfied. Note that, for our results we only need $1+r_{k}$ to be bounded away from $0$ and $\infty$.
\end{rem}
The next assumption corresponds to the linearity condition in Theorem \ref{th:st exch}. However, since the change has been separated into two parts due to the change in $Y$, we state it as two separate equalities.
\begin{ass}\label{ass:M1}
	For a $d\times d$ invertible matrix $\Psi$ we have that
	\begin{align*}
		M_{1, \pm}(\mvW, Y)=-\gl\left(\frac12\Psi \mvW+\mvR_{1, \pm}\right).
	\end{align*}
	In one dimensional case this assumption takes form of
	\begin{align}\label{eq:as2b}
		M_{1, \pm} (W, Y)= - \gl \left( \frac{1}{2}\psi W+R_{1, \pm}\right),
	\end{align}
	for some number $\psi>0$.
\end{ass}

Notice that, if $Y'=Y$ implies $W'=W$, using the fact that $\DY \in\{-1, 0, 1\}$ a.s. and adding the equalities~\eqref{eq:as2b} together yields exactly the same linearity condition as in Theorem~\ref{th:st exch}, where $\psi$ accounts for the difference in $\gl$ for $W$ and $Y$. In general, the separation on the linearity condition based on the change to $Y$ need not be symmetric, in the sense that the factor of $1/2$ in front of $W$ would be replaced by $a_\pm$ with the property that $a_++a_-=1$. We tackle this difficulty by an appropriate change of variable that accounts for this asymmetry.

\begin{customass}
	{\theass a}\label{ass:M1g} Assume that for a $d\times d$ invertible matrix $\Psi$ we have that
	\begin{align*}
		M_{1, \pm}(\mvW, Y)=-\gl\left(\Psi_\pm \mvW+\mvb_{\pm} Y +\mvR_{1, \pm}\right),
	\end{align*}
	where $\Psi_{1, +}+\Psi_{1, -}=\Psi$, $\mvb_++\mvb_-={\bf 0}$ and $\mvb_+=\frac12\Psi\mvb_+$. In the one dimensional case this assumption becomes
	\begin{align*}
		M_{1, \pm}(W, Y) = - \gl \left( a_\pm \psi W+b_\pm Y +R_{1, \pm}\right).
	\end{align*}
	for some number $\psi>0$, where $a_++a_-=1$ and $b_++b_-=0$.
\end{customass}
The last assumption has to do with the behavior of the conditional second moment of the change in $W$.
\begin{ass}\label{ass:M2}
	With the same notations as above we assume that
	\begin{align}\label{eq:WWT}
		\E\left(\DW\DW^T\1_{\DY =\pm1}\mid \mvW, Y\right)=\gl\left(\Psi\gS+\gC_{2, \pm}\right),
	\end{align}
	for some random matrices $\gC_{2, \pm}=\gC_{2, \pm}(\mvW, Y)$. In the one dimensional case this assumption becomes
	\begin{align}\label{eq:M2}
		M_{2, \pm} (W, Y)= \gl \left(\psi +R_{2, \pm}\right).
	\end{align}
\end{ass}


\subsection{Notations}\label{ssec:notations}

Throughout this paper, we will use the following conventions:
\begin{itemize}
	\item Capital Roman letters and lower case Greek letters such as $\go$ and $\xi$ denote random variables.
	\item Capital Roman letters in bold font, as well as $\mvgo$, denote random vectors.
	\item Capital Greek letters denote Matrices.
	\item Lowercase Greek and Roman letters denote deterministic functions or numbers except for $\go$ and $\xi$.
	\item Unless explicitly needed to emphasize the dependence on $n$, we will omit the subscript $n$ and use $U, V, W, \ldots$ instead of $U_{n}, V_{n}, W_{n}, \ldots$.
	\item Expressions involving $\pm$ and $\mp$ should be read as two different expressions: taking all signs on the top and taking all signs on the bottom.
\end{itemize}
We will also use the following notations throughout the rest of the paper.
\begin{itemize}
	\item $Z$ always denotes a standard normal random variable
	\item $\mvZ=(Z_1, Z_2, \ldots, Z_d)$ denotes a $d$-dimensional standard Gaussian vector.
	\item $\overline{X}:=X-\E X$ denotes a centered version of a random variable $X$.
	\item $\sigma_W^{2}$ - variance of a random variable $W$,
	\item $W'$ - exchangeable copy of $W$,
	\item $\gD W:=W'-W$,
	\item $W_i$ represents its $i^{\textrm{th}}$ coordinate of a vector $\mvW$,
	\item $f_{i_1, i_2, \ldots, i_m}:=\frac{\partial^m f}{\partial w_{i_1}\partial w_{i_2}\cdots \partial w_{i_m}}$ for a function $f$,
	\item $\abs{f}_m:=\sup_{i_1, i_2, \ldots, i_m}\norm{\partial_{i_1, i_2, \ldots, i_m}f}_\infty$ if such quantities exist for a function $f$,
	\item $\cA:=\left\{f: \abs{f}_0\le1, \abs{f}_1\le\sqrt{2/\pi}, \abs{f}_2\le2\right\}$,
	\item $\dwas(W, Z):=\sup_{h: 1\textrm{-Lip.}}\abs{\E h(W) - \E h(Z)}$ - Wasserstein-$1$ distance,
	\item $f\approx g$ if $f=\Theta(g)$, $f\lesssim g$ if $f=O(g)$, and $f\ll g$ if $f=o(g)$,
	\item $N:= \binom{n}{2}$,
	\item $Q:=\gl\gs_Y^2$
	\item $p\in(0, 1)$ and $q:=1-p$,
	\item $\abs{\cdot}$ - the Euclidean $\ell_{2}$ norm,
	\item $\norm{\cdot}_{\op}$ - the operator norm of a square matrix,
	\item $\norm{\cdot}_{\hs}$ - Hilbert-Schmidt norm of a square matrix, {\ie} for a $d\times d$ matrix $\gC$,
	      $\norm{\gC}_{\hs}:=\sqrt{\tr(\gC^T\gC)}.$
	\item $\norm{\cdot}_{p-\hs}$ - is the $p^{\textrm{th}}$ norm of the Hilbert-Schmidt norm of a random matrix for $p\ge1$, {\ie} $\norm{A}_{p-\hs}=(\E\norm{A}_{\hs}^p)^{1/p}$.
\end{itemize}

To simplify notations we also define for $\ell\ge 1$ the $\ell^{\textrm{th}}$ conditional moment of the change in $W$ given the change in $Y$ as follows
\begin{align*}\label{eq:M}
	M_{\ell, \pm}(W, Y):=\E((\gD W)^\ell \1_{\{\DY = \pm1\}} \mid W, Y).
\end{align*}

\subsection{Road Map}\label{ssec:map}

This paper is organized as follows: after discussing preliminaries and heuristics, we provide the statements of the main results in univariate and multivariate settings in Section~\ref{ssec:mainres}~and~\ref{ssec:multstein}. We present applications of classical methods to CCLT in Section~\ref{ssec:applications of classic}, following that, we present applications of our main results in Sections~\ref{ssec:applications of main thm}~and~\ref{ssec:applications multistein} for univariate and multivariate cases, respectively. The remainder of this paper is dedicated to the proof of the main results. We conclude by discussing future work, open questions, and the difficulties of our approach. Below we provide an extended road map for Section~\ref{ssec:mainres} and Section~\ref{sec:appl}.

In Section~\ref{ssec:mainres} we first state our main result in the simplest case in Theorem~\ref{th: symcase}. Lemma~\ref{lem:consecutive k} functions as a stepping stone to the improved version of the result in Theorem~\ref{th: imp symcase}. After that we discuss the general case and the change of variables that reduces Assumption~\ref{ass:M1g} to Assumption~\ref{ass:M1} in Proposition~\ref{prop: change}. In Section~\ref{ssec:multstein} we present the analogous results to Theorem~\ref{th: imp symcase} and Proposition~\ref{prop: change} in multivariate setting in Theorem~\ref{th: multi case} and Proposition~\ref{prop: multichange}.

In applications we focus on counting examples in variety of models most notably subpattern counts in a binary sequence in Sections~\ref{ssec:pattern1/2}, ~\ref{ssec:even odd}, and~\ref{ssec:pattern}; as well as subgraph counts in a random graph in Sections~\ref{ssec:degsq}, ~\ref{ssec:wedgeedge}, and~\ref{ssec:triangle}, in Section~\ref{ssec: general subgraph} we build on that and present a CCLT for a general subgraph count joint with triangle and wedge counts given the number of edges in Theorem~\ref{th: gengraph}.

\section{Main results: Univariate Case}\label{ssec:mainres}
First, we state the simplest version of our main result.

\begin{thm}[Symmetric case]\label{th: symcase}
	Suppose $W$ and $Y$ are random variables satisfying Assumptions~\ref{ass:uncorr and exch},~\ref{ass:Y},~\ref{ass:M1} and~\ref{ass:M2}. Let $p_k:=\pr(Y=k)>0$. Then we have that
	\begin{align*}
		\dwas\bigl((W\mid Y=k), Z\bigr)
		 & \le \frac{2}{\psi}\left( A_k+\frac{1}{p_k} C\right)+ \sqrt{\frac{2}{\pi\psi^2}}\left( B_k+\frac{1}{p_k} D\right)+\frac{2}{3 \gl\psi p_k}E,
	\end{align*}
	where
	\begin{alignat*}{3}
		A_k & = \E\left(\abs{R_{1, -}} \, \big|\, Y=k\right), \qquad
		    &                                                        & B_k &  & = \E\left(\abs{R_{2, -}} \, \big|\, Y=k\right),   \\
		C   & =\E\abs{R_{1, +}-R_{1, -}}, \qquad
		    &                                                        & D   &  & =\E\abs{R_{2, +}-R_{2, -}},                       \\
		    & \qquad\qquad\quad\text{ and } \quad E \le \E|\gD W|^3. &     &  &                                                 &
	\end{alignat*}
\end{thm}

In the proof of Theorem~\ref{th: symcase} we consider $g(y):=(-1)^y\1_{y \le 0}$, so that
\[
	g(y+1-k) + g(y-k)=\1_{y= k}.
\]
This is the most straightforward function that satisfies all of the properties that we require.

\begin{rem}[Asymmetry in the error terms]
	One can notice that the bound in the conclusion of Theorem~\ref{th: symcase} is asymmetric with respect to the error terms $A_k$ and $B_k$, namely only $R_{1, -}$ and $R_{2, -}$ appear, while $R_{1, +}$ and $R_{2, +}$ do not. The reason for this is that we chose the indicator $\1_{y \le 0}$ inside of function $g(y)$. If instead one uses $\widehat{g}(y):=\left(-1\right)^{y}\cdot \ind_{y \ge 0}$ then the similar argument would give a bound with only $R_{1, +}$ and $R_{2, +}$ appearing in the first term. Moreover, one could also consider their average to get the following bound
	\begin{align*}
		\dwas\bigl((W\mid Y=k), Z\bigr) & \le \frac{1}{\psi}\left(A'_k+\frac{1}{p_k} C\right)+ \sqrt{\frac{1}{2\pi\psi^2}}\left( B'_k+\frac{1}{p_k} D\right)+\frac{1}{3 \gl\psi p_k}E,
	\end{align*}
	where
	\begin{align*}
		A_k' & =\E\left(\abs{R_{1, +}+R_{1, -}}\ \big|\ Y=k\right)\quad\text{and}\quad B_k' =\E\left(\abs{R_{2, +}+R_{2, -}}\ \big|\ Y=k\right),
	\end{align*}
	and the rest of the terms remain the same as in Theorem~\ref{th: symcase}.
\end{rem}

There are two aspects in which Theorem~\ref{th: symcase} needs improvement. First, it is often the case that the term $\frac{1}{p_k}\E \abs{R_{2, \pm}}$ might not go to zero even though $\E \abs{R_{2, \pm}}$ are small on their own. To improve this bound we consider a different function $g(y)=\1_{y= 0}$. This adaptation with a similar proof to the one of Theorem~\ref{th: symcase} allows us to derive the following CCLT.

\begin{lem}\label{lem:consecutive k}
	Suppose $W$ and $Y$ are random variables satisfying Assumptions~\ref{ass:uncorr and exch},~\ref{ass:Y},~\ref{ass:M1} and~\ref{ass:M2}. For any $k$ such that $p_k:=\pr(Y=~k)>0$ and $p_{k-1}:=\pr(Y=~k-1)>0$ we have
	\begin{align}\label{eq:cons}
		 & \dwas\bigl((W\mid Y\in \{k-1, k\}), Z\bigr)\le \frac{2}{\psi(1+r_k)}\widehat{A}_k+ \frac{\sqrt{2/\pi}}{\psi(1+r_k)}\widehat{B}_k+ \frac{2}{3 \gl\psi}\widehat{E}_k,
	\end{align}
	where $r_k=p_{k-1}/p_k$ and
	\begin{align*}
		\widehat{A}_k                   & =\E\left(\abs{R_{1, -}}\, \mid\, Y=k\right)+r_k\cdot\E\left(\abs{R_{1, +}} \mid Y=k-1\right),         \\
		\widehat{B}_k                   & =\E\left(\abs{R_{2, -}} \, \big|\, Y=k\right)+r_k\cdot\E\left(\abs{R_{2, +}} \, \big|\, Y=k-1\right), \\
		\text{ and }\quad \widehat{E}_k & = \E(|\gD W|^3\mid Y\in \{k-1, k\}).
	\end{align*}
\end{lem}
In most applications of our results we expect $r_k$ to be uniformly bounded away from $0$ and infinity, see Remark~\ref{rem:llt} and Lemma~\ref{lem:lltratio}.

As we mentioned in Section~\ref{ssec:Stein}, the exchangeable pair approach is most useful in models where a small perturbation does not change the observed quantity too much. Using this intuition it is natural that for a Lipschitz function $h(w)$ the difference between of $\E(h(W)\mid Y=k)$ and $\E(h(W)\mid Y=k-1)$ should be negligible. To make this heuristic rigorous, we use the fact that $M_{0, \pm}$ are concentrated at $Q$ (Assumption~\ref{ass:Y}) and exchangeability of $W$ and $W'$ (see Lemma~\ref{lem:stein dif} for the exact statement). Combining Lemma~\ref{lem:consecutive k} with the fact that $W$ conditioned on the event $\{Y=k-1\}$ is almost the same as $W$ conditioned on $\{Y=k\}$ yields the improved version of the univariate result.

\begin{thm}[Improved symmetric case]\label{th: imp symcase}
	Suppose $W$ and $Y$ are random variables satisfying Assumptions~\ref{ass:uncorr and exch},~\ref{ass:Y},~\ref{ass:M1} and~\ref{ass:M2}. Let $k$ be such that $\pr(Y=k)>0$ and $\pr(Y=k-1)>0$. Then
	\begin{align*}
		\dwas\bigl((W\mid Y=k), Z\bigr) & \le \frac{1}{\psi}\widehat{A}_k+\sqrt{\frac{1}{2\pi\psi^2}}\widehat{B}_k                                           \\
		                                & \qquad+(1+r_k)\left(\frac{1}{2Q} \widehat{C}_k+\frac{1}{2Q}\widehat{D}_k+\frac{2}{3 \gl\psi} \widehat{E}_k\right),
	\end{align*}
	where $r_k,$ $\widehat{A},$ $\widehat{B},$ and $\widehat{E}$ are as in Lemma~\ref{lem:consecutive k}, while
	\begin{align*}
		\widehat{C}_k             & =\E\left((\abs{W}+\E\abs{Z})(\abs{R_{0, +}}+\abs{R_{0, -}}) \, \big|\, Y\in\{k-1, k\}\right) \\
		\text{and } \widehat{D}_k & =\E\left(\abs{\gD W}\, \big|\, Y\in\{k-1, k\}\right).
	\end{align*}
\end{thm}

We now derive Theorem~\ref{th: imp symcase} using Lemmas~\ref{lem:consecutive k} and~\ref{lem:stein dif}. The proofs of these lemmas are presented in Section~\ref{sec:proofs}.

\begin{proof}
	By Lemma~\ref{lem:consecutive k} we have that for any $1$-Lipschitz function $h:\dR\to\dR$
	\begin{align}\label{eq:cons k}
		\abs{\E (h(W)-h(Z))\1_{Y\in\{k-1, k\}}}
		 & \le \frac{2}{\psi}\left( \E|R_{1, -}|\ind_{Y= k}+ \E|R_{1, +}|\ind_{Y= k-1}\right)                  \\
		 & +\sqrt{\frac{2}{\pi\psi^2}}\left( \E|R_{2, -}|\ind_{Y= k} + \E|R_{2, +}|\ind_{Y= k-1}\right) \notag \\
		 & +\frac{2}{3\gl \psi} \E\abs{\gD W}^3\ind_{Y\in\{ k-1, k\}}.\notag
	\end{align}
	On the other hand by Lemma~\ref{lem:stein dif} we have that for any such function $h$ and for any $k$ such that $\pr(Y=k)>0$ and $\pr(Y=k-1)>0$ we have that
	\begin{align}\label{eq:steindif}
		 & \abs{\E (h(W)-h(Z))\left(\ind_{Y=k}-\ind_{Y=k-1}\right) }                                                                                                           \\
		 & \leq \frac{1}{Q}\E\left(|\gD W|\cdot\1_{Y\in\{k-1, k\}}\right)+\frac{1}{Q}\E\left((\abs{W}+\abs{Z})(\abs{R_{0, +}}+\abs{R_{0, -}})\1_{Y\in\{k-1, k\}}\right).\notag
	\end{align}
	Adding the inequalities from Lemma~\ref{lem:consecutive k} to the one in Lemma~\ref{lem:stein dif} gives us that
	\begin{align*}
		2\cdot \abs{\E (h(W)-h(Z))\1_{Y=k}}
		 & \le \frac{2}{\psi}\left( \E|R_{1, -}|\ind_{Y= k}+ \E|R_{1, +}|\ind_{Y= k-1}\right)              \\
		 & +\sqrt{\frac{2}{\pi\psi^2}}\left( \E|R_{2, -}|\ind_{Y= k} + \E|R_{2, +}|\ind_{Y= k-1}\right)    \\
		 & +\frac{2}{3\gl \psi} \E\left(\abs{\gD W}^3\ind_{Y\in\{ k-1, k\}}\right)
		+\frac{1}{Q}\E\left(|\gD W|\cdot\1_{Y\in\{k-1, k\}}\right)                                         \\
		 & +\frac{1}{Q}\E\left((\abs{W}+\abs{Z})(\abs{R_{0, +}}+\abs{R_{0, -}})\1_{Y\in\{k-1, k\}}\right).
	\end{align*}
	Dividing both sides of the inequality by $2p_k$ and recalling that $r_k=\frac{p_{k-1}}{p_{k}}$ we derive the desired result.
\end{proof}

One can see that Theorem~\ref{th: imp symcase} partially generalizes~\cite[Theorem 4.2]{Peterson19}. While our result allows for the dependence among random variables, it also requires the third absolute moment to be finite (as in~\cite[Theorem 4]{Bolthausen80}) and limits the change in $Y$ to only $\{- 1, 0, 1\}$. One can adapt our approach to models with finite $(2+\delta)$ - th moment, although it might involve some technical computations. However, relaxing the assumption on the range of $\DY $ is of particular interest. In case when $\pr(\DY \in\{-1, 0, 1\})$ is sufficiently large, our techniques are still applicable but could yield a suboptimal rate of convergence. For further discussion see Section~\ref{ssec:rangeY}.

\begin{rem}[Explicit bound for the error terms with conditioning]\label{rem: eps}
	All of the error terms in Theorem~\ref{th: imp symcase} are of the form $\E(\abs{R}\mid Y=i)$ where $i\in\{k-1, k\}$. One can bound those terms using H\"older inequality in the following way
	\begin{align*}
		\E(\abs{R}\mid Y=i) & =\pr(Y=i)^{-1}\E(\abs{R}\1_{Y=i})\le\frac{\norm{R}_p}{\pr(Y=i)^{1/p}}
	\end{align*}
	for some $p>1$. Provided that the random variable $R$ has $p^{\textrm{th}}$ moment and is of order $n^{-\ga}$, \ie~ for some constant $C(p)$ depending on $p$ we have $\norm{n^\ga R}_p\le C(p)$, and if $\pr(Y=i)\approx~n^{-\beta/2}$, then
	\[
		\E(\abs{R}\mid Y=i)\le C(p) n^{-\ga+\frac{\beta}{2p}}.
	\]
	In our applications, we can take $p$ to be very large, so we will usually write \[\E(\abs{R}\mid Y=i)\lesssim n^{-\ga+\eps}\] for some small $\eps>0$.
\end{rem}

The second aspect in which both Theorems~\ref{th: symcase} and~\ref{th: imp symcase} need improvement is that one would want to extend it to the models with asymmetries concerning the change in $Y$ of the form
\begin{align*}
	M_{1, \pm}(W, Y) = - \gl \left( a_\pm \psi W+b_\pm Y +R_{1, \pm}\right),
\end{align*}
and where the term $b_\pm Y$ is not be negligible. In other words, we would like to apply this result to the models that satisfy Assumption~\ref{ass:M1g} instead of Assumption~\ref{ass:M1}. One can do it by applying a change of variable presented in the following proposition.

\begin{prop}[Univariate change of variable]\label{prop: change}
	Suppose $X$ and $Y$ are random variables satisfying Assumptions~\ref{ass:uncorr and exch},~\ref{ass:Y},~\ref{ass:M1g} and~\ref{ass:M2}. Assume that
	\[
		R_{0, \pm}=\mp\gl a_\pm Y
	\]
	and define the change of variable
	\begin{align}\label{eq:change of var}
		W^0:=X+\gl\, \psi \, \ga XY+\frac{\gl\, \theta}{2}\left(Y^2-\E Y^2\right)+\frac{\gl^2(\psi +1)\ga\theta}{3}(Y^3-\E Y^3),
	\end{align}
	where
	\[\ga=\frac{a_+-a_-}{2Q}\quad \textrm{and} \quad\theta=\frac{b_+}{Q}.\] Let $W={W^0}/{\gs_{W^0}}$. Then $\left(W, Y\right)$ satisfies Assumptions~\ref{ass:uncorr and exch},~\ref{ass:Y},~\ref{ass:M1} and~\ref{ass:M2}
	with error terms $\widetilde{R}_{1, \pm}$ and $\widetilde{R}_{2, \pm}$. Moreover, we have
	\begin{align*}
		 & \widetilde{R}_{1, \pm}=\frac{\gl\theta}{2}\left(1-\frac{\psi }{2}\right)\frac{\overline{Y^2}}{\gs_{W^0}}+\frac{1}{\gs_{W^0}}(\widetilde{\eps}_{0, \pm}+\widetilde{\eps}_{1, \pm}+\widetilde{\eps}_{2, \pm}+\widetilde{\eps}_{3, \pm}),
	\end{align*}
	where
	\begin{align*}
		\widetilde{\eps}_{0, \pm}                & :=\gl\psi^2\ga a_\pm X(Y\pm 1)-\frac{\gl\, \ga\, \psi }{2} XY+R_{1, \pm}\biggl(1\pm\gl\, \psi \, \ga+ \gl\, \psi \, \ga Y\biggr), \\
		\widetilde{\eps}_{1, \pm}                & :=\gl\left(\pm\gl b_\pm\psi\ga\mp\frac{a_\pm\theta}{2}+(\psi+1)\theta\ga\left(Q-\frac13\gl a_\pm\right)\right)Y,                  \\
		\widetilde{\eps}_{2, \pm}                & :=\mp\gl^2\theta(\psi +1)\ga a_\pm Y^2,                                                                                           \\
		\text{and}\quad\widetilde{\eps}_{3, \pm} & :=\frac13\gl\theta(\psi +1)\ga Q-\gl^2\theta(\psi +1)\ga a_\pm Y^3-\frac{\gl^2\psi (\psi +1)\ga\theta}{6}\left(Y^3-\E Y^3\right),
	\end{align*}
	and for $p\ge1$
	\begin{align*}
		\norm{\widetilde{R}_{2, \pm}-\frac{R_{2, \pm}}{\gs^2_{W^0}}}_{2p}
		 & \lesssim \gl^{\half}\psi \abs{\ga}\left(\frac{\gs_{X}}{\gs_{W^0}}\sqrt{\norm{\gl(1+R_{2, \pm}/\gs^2_{X})}_{2p}\norm{Y}_{2p}}+\frac{\norm{X}_p}{\gs_{W^0}}\right) \\
		 & \qquad+\gl^{\half}\abs{\theta}\cdot\frac{\norm{Y}_p}{\gs_{W^0}}
		+\gl^{3/2}(\psi +1)\abs{\ga\theta}\cdot\frac{ \norm{Y^2}_p}{\gs_{W^0}}.
	\end{align*}
\end{prop}

Assuming the particular form of $R_{0, \pm}$ one can see this change of variable as the correction for the asymmetry created by the change of $Y$ (\ie~the $XY$ term) and the approximation of the conditional mean up to the third order. Considering the examples where one can compute the exact conditional mean, we notice that the change of variable is indeed very close to the true value. We illustrate that in the Remarks~\ref{rem:cond mean pattern} and~\ref{rem:cond mean change}.

Now we explain the intuition behind each term in the change of variable.
One needs the term $\gl\, \psi\, \ga XY$ to account for the asymmetry between $a_+$ and $a_-$. Indeed, when $a_+=a_-$, we have that $\ga=0$ and this term is not present. In the computation of $M_{1, \pm}(W^0, Y)$ it produces the term $\mp\ga Q X$ that in combination with $a_\pm X$ gives us \begin{align*}
	(a_\pm\mp\ga Q)X=\left(a_\pm\mp\frac{a_+-a_-}{2}\right)X=\frac12X.
\end{align*}

The square term has the exact form to cancel out $b_\pm Y$ if it is non-negligible, while the cubic term is present to cancel everything created by the square term to match with the coefficient of $\E Y^2$.

\begin{rem}[Parameter $\psi $]
	In the majority of the examples for which we expect Theorem~\ref{th: imp symcase} to be used, one has $\psi =2$. When $\psi \neq 2$, we believe that one likely needs to use the multivariate version as in Theorem~\ref{th: multi case}, however, we presented the statement for general $\psi $, in case one has sufficient control on the quantity $\gl\theta\sigma_Y$.
\end{rem}

\section{Main results: Multivariate Case}\label{ssec:multstein}
In this section, we present the extension of our main result to the multivariate setting.
\begin{thm}\label{th: multi case}
	Let $\mvW, Y$ be, respectively, a $\dR^d$-valued random vector and a random variable satisfying Assumptions~\ref{ass:uncorr and exch},~\ref{ass:Y},~\ref{ass:M1}, and~\ref{ass:M2}. Let $k$ be such that $p_k :=\pr(Y=k)>0$ and $p_{k-1}:=\pr(Y=k-1)>0$. If $\E\abs{\mvW}^4<\infty$ then
	\begin{align*}
		 & \dwas((\mvW\mid Y=k), \gS^{\half}\mvZ)                                                                                                                     \\
		 & \qquad \le \frac{1}{1+r_k}\widehat{A}_k+\frac{1}{2(1+r_k)}\norm{\gS^{-\half}}_{\op}\widehat{B}_k+\frac{1+r_k}{2Q}\left(\widehat{C}_k+\widehat{D}_k \right) \\
		 & \qquad\qquad\qquad+ (1+r_k)\norm{\gS^{-\half}}_{\op}^{3/2}\sqrt{c_3\cdot \widehat{E}_k\cdot\widehat{F}_k},
	\end{align*}
	where $c_3=(2+8e^{-3/2})/\sqrt{2\pi}<2,$ $r_{k}=p_{k-1}/p_{k}$,
	\begin{align*}
		\widehat{A}_k                  & =\E\left(\abs{\Psi^{-1}\mvR_{1, -}}\mid Y=k\right)+r_k\E\left(\abs{\Psi^{-1}\mvR_{1, +}}\mid Y=k-1\right),             \\
		\widehat{B}_k                  & =\E\left(\norm{\Psi^{-1}\gC_{2, -}}_{\hs}\mid Y=k\right)+r_k\E\left(\norm{\Psi^{-1}\gC_{2, +}}_{\hs}\mid Y=k-1\right), \\
		\widehat{C}_k                  & =\E\left(\left(\abs{\mvW}+\sqrt{\tr(\gS)}\right)(\abs{R_{0, +}}+\abs{R_{0, -}}) \, \big|\, Y\in\{k-1, k\}\right)       \\
		\widehat{D}_k
		                               & =\E\left(\abs{\DW}\, \big|\, Y\in\{k-1, k\}\right)                                                                     \\
		\widehat{E}_k                  & ={\E\left(\abs{(\gl\Psi)^{-1}\DW}\cdot\abs{\DW}^3\, \big|\, Y\in\{k-1, k\}\right)},                                    \\
		\text{ and }\quad	\widehat{F}_k & =\sqrt{\tr(\gS) + \E \left(\abs{\mvW}^{2} \, \big|\, Y\in\{k-1, k\}\right)}.
	\end{align*}
	In particular, if $\cov(\mvW)=\gS=I_d$ then the error bound reduces to
	\begin{align*}
		\dwas((\mvW\mid Y=k), \mvZ) & \leq\frac{1}{1+r_k}\left(\widehat{A}_k+\frac{1}{2}\widehat{B}_k\right)+\frac{1+r_k}{2Q}\left(\widehat{C}_k+\widehat{D}_k\right) \\
		                            & \quad+ (1+r_k)\cdot (d+\E(\abs{\mvW}^{2}\mid Y\in \{k,k-1\}))^{\frac14}\cdot \sqrt{\widehat{E}_k}.
	\end{align*}
	On the other hand if only $\E\abs{\mvW}^3<\infty$ then 	\begin{align*}
		\dwas((\mvW\mid Y=k), \gS^{\half}\mvZ) & \le \frac{1}{1+r_k}\left(\widehat{A}_k+\frac{1}{2}\norm{\gS^{-\half}}_{\op}\widehat{B}_k\right) \\
		                                       & \quad+\frac{1+r_k}{2Q}\left(\widehat{C}_k+\widehat{D}_k\right)
		+c_2(1+r_k)\cdot\norm{\gS^{-1/2}}_{\op}^{2}\widehat{E}_k'\cdot\widehat{F}_k',
	\end{align*}
	where $c_2=4/\sqrt{2\pi e}<1$, $\widehat{A}_k, \widehat{B}_k, \widehat{C}_k, \widehat{D}_k$ are as above and
	\begin{align*}
		\widehat{E}_k'
		                                & =\E\left(\abs{(\gl\Psi)^{-1}\DW}\abs{\DW}^{2}\, \big|\, Y\in\{k-1, k\}\right),                        \\
		\text{and}\qquad \widehat{F}_k' & =1+\abs{\log\left(c_2 \cdot\norm{\gS^{-1/2}}_{\op}^{2}\cdot {\widehat{E}'_k}/{\widehat{F}_k}\right)}.
	\end{align*}
\end{thm}

\begin{rem}
	The case when $\Sigma$ is the identity matrix, the proof of Theorem~\ref{th: imp symcase} can be easily adapted to get a bound on $\left|\E\left(\gD f(\mvW)-\mvW^T\nabla f(\mvW)\right)\1_{Y=k}\right|$. This case, in a way, corresponds to the setting of~\cite{ChatterjeeMeckes08}. However if $\Sigma$ is not an identity matrix one has to change the function to an expression that already resembles Taylor approximation and derive an upper bound to $\left|\E\left(\nabla^T\Sigma\nabla f(\mvW)-\mvW^T\nabla f(\mvW)\right)\1_{Y=k}\right|$. This generalization is done in the spirit of~\cite{FangKoike22, Raic19, ReinertRollin09}.
	Moreover, in many applications $\norm{\Psi^{-1}}_{\op}$ is bounded by a constant
	and thus $\abs{\Psi^{-1}\DW}$ can be bounded by some constant multiple of $\abs{\DW}$.
\end{rem}
The proof of Theorem~\ref{th: multi case} is analogous to the univariate case. We first establish a CCLT for $(\mvW\mid Y\in\{k-1, k\})$ that we state in Lemma~\ref{lem: multi conseq k} and then combine it with the quantitative bound on the difference of $(\mvW\mid Y=k)$ and $(\mvW\mid Y=k-1)$ as in Lemma~\ref{lem:stein dif}. To derive the bound in terms of Wasserstein distance, we follow the smoothing technique as in \cite{FangKoike22, Raic19}. We present the proof of Theorem~\ref{th: multi case} in Section~\ref{ssec:multistein proof}.


Similar to the univariate case, we introduce a change of variable that allows one to pass from Assumption~\ref{ass:M1g} to Assumption~\ref{ass:M1}.

\begin{prop}[Multivariate change of variable]\label{prop: multichange}
	Let $\mvX, Y$ be, respectively, a $\dR^d$-valued random vector and a random variable satisfying Assumptions~\ref{ass:uncorr and exch}, ~\ref{ass:Y}, ~\ref{ass:M1g}, and~\ref{ass:M2}. Define the following change of variable
	\begin{align}\label{eq:multichange}
		\mvW^0:=\mvX+\gl\gA \mvX Y+\frac{\gl\mvgth}{2}\left(Y^2-\E Y^2\right)+\frac{\gl^2(\gA+\ga)\mvgth}{3}(Y^3-\E Y^3),
	\end{align}
	where
	\[\gA=\frac{\Psi_+-\Psi_-}{2Q}, \quad\ga=\frac{a_+-a_-}{2Q}, \quad \textrm{and}\quad\mvgth=\frac{\mvb_+}{Q}.\]
	Then vector $\mvW:=\left(\mvW^0_i/\gs_{\mvW^0_i}\right)_{1\le i\le d}$ satisfies Assumptions~\ref{ass:uncorr and exch}, ~\ref{ass:Y}, ~\ref{ass:M1}, and~\ref{ass:M2} with error terms $\widetilde{\mvR}_{1, \pm}$ and $\widetilde{\Gamma}_{2, \pm}$. In particular, we have
	\begin{align*}
		 & \widetilde{\mvR}_{1, \pm}=\gS^{-\half}_{W^0}(\widetilde{\mveps}_{0, \pm}+\widetilde{\mveps}_{1, \pm}+\widetilde{\mveps}_{2, \pm}+\widetilde{\mveps}_{3, \pm}),
	\end{align*}
	where
	\begin{align*}
		\widetilde{\mveps}_{0, \pm}                & :=\gl a_\pm\Psi\gA \mvX(Y\pm 1)-\frac{\gl\, \Psi\, \gA }{2} \mvX Y+\biggl(1\pm\gl \, \gA+ \gl\, \gA Y\biggr)\mvR_{1, \pm},  \\
		\widetilde{\mveps}_{1, \pm}                & :=\gl\left(\pm\gl\gA \mvb_\pm\mp\frac{a_\pm\mvgth}{2}+(\gA+\ga)\mvgth\left(Q-\frac13\gl a_\pm\right)\right)Y,               \\
		\widetilde{\mveps}_{2, \pm}                & :=\mp\gl^2 a_\pm(\gA +\ga)\mvgth Y^2,                                                                                       \\
		\text{and}\quad\widetilde{\mveps}_{3, \pm} & :=\frac13\gl(\gA +\ga)\mvgth Q-\gl^2 a_\pm(\gA +\ga)\mvgth Y^3-\frac{\gl^2\Psi (\gA +\ga)\mvgth}{6}\left(Y^3-\E Y^3\right),
	\end{align*}
	and
	\begin{align*}
		 & \norm{\widetilde{\Gamma}_{2, \pm}-\gS^{-\half}_{\mvW^0}\Gamma_{2, \pm}\gS^{-\half}_{\mvW^0}}_{p-\hs}                                                                                                                             \\
		 & \quad\lesssim \gl^{\half} \left(\frac{\norm{\gS_{\mvX}}_{\hs}}{\norm{\gS_{\mvW^0}}_{\hs}}\sqrt{\norm{\gl(\Psi+\gS_{\mvX}^{-\half}\Gamma_{2, \pm}\gS_{\mvX}^{-\half})}_{p-\hs}\norm{Y}_{2p}}+\norm{\gS_{\mvX}^{-\half}X}_p\right) \\
		 & \qquad+\gl^{\half}\norm{\mvgth}_p\cdot\norm{\gS_{\mvW^0}^{-\half}}_{\hs}\norm{Y}_p
		+\gl^{3/2}\norm{\mvgth}_p\norm{\gS_{\mvW^0}^{-\half}}_{\hs} \norm{Y^2}_p.
	\end{align*}
\end{prop}

\begin{rem}\label{rem:changes of var}
	Note that the change of variable~\eqref{eq:multichange} with $d=1$ agrees with univariate change as in~\eqref{eq:change of var}. The terms play analogous role as described in the univariate case. From the Assumption~\ref{ass:M1g} with $d=1$ we have that $\Psi_\pm=a_\pm\psi$ and hence $\gA=(\Psi_+-\Psi_-)/2=\psi(a_+-a_-)/2=\psi\ga$ making the $XY$ term of~\eqref{eq:change of var} match the $\mvX Y$ term of~\eqref{eq:multichange}, while $\gA+\ga=\psi\ga+\ga=(\psi+1)\ga$ matches the cubic terms of these changes of variables.
\end{rem}

\section{Applications}\label{sec:appl}
\subsection{Classical methods in CCLT}\label{ssec:applications of classic}
In this section we present several derivation of CCLT using classical Stein's method. As we mentioned in Remark~\ref{rem:stand} to apply this technique to a random vector $(X, Y)$ it is crucial to work with an exchangeable pair $((X, Y), (X', Y'))$ such that $\pr(\gD X\neq0, \DY =0)>0.$

In each of the following subsections we first describe the model and then state the relevant CCLT result.
\subsubsection{Variant of an occupancy problem}\label{ssec:urn}
Suppose there are three urns and $n$ many distinct balls. At time $i$ we put the $i^{\textrm{th}}$ ball into an urn numbered $U_i$, where $U_i$'s are i.i.d.~random variables that are equal to $1$ with probability $p_1$, to $2$ with probability $p_2$, and to $3$ with probability $p_3:=1-p_1-p_2$. Define
\begin{align}\label{def:yw}
	V & :=\sum_{i=1}^n\1_{U_i=2}
	\text{ and }
	W := n^{-\half}\sqrt{\frac{p_1p_2}{1-p_2}}\sum_{i=1}^n\left(\frac{1}{p_1}\1_{U_i=1}-\frac{1}{p_3}\1_{U_i=3}\right)
\end{align}
as the number of balls in the second urn and the scaled difference between the number of balls in the first and the third urn, respectively, at time $n$.
Notice that $\gs_V^2=np_2(1-p_2)$. We have the following CCLT result with an explicit rate of convergence for $(W\mid V=k)$ when $|k-np_2|\ll n^{\half}$.

\begin{lem}\label{ex: urn}
	Let $W$ and $Y$ be as defined in~\eqref{def:yw}. For any $k\in \dZ$ with $|k-np_2|\ll n^{\half}$, we have
	\begin{align*}
		\dwas\left((W \mid V =k), \, Z\right)\lesssim n^{-\half}.
	\end{align*}
\end{lem}
\begin{proof}
	Notice that $W $ is mean zero variance one random variable.
	Now conciser the following Glauber dynamics Markov chain, for a given $(W , V )$ we pick index of one of the balls uniformly at random, \ie~$I\sim \textrm{Uniform}\{1, 2, \ldots, n\}$, and then re-sample its placing, \ie~place $I^{\textrm{th}}$ ball into the urn $U_I$ where $U_I\eqd U_1$ and independent of everything else. Call the result $(W', V')$. Clearly $(W , V )$ and $(W', V')$ are equidistributed. Notice that, with $\gl=(1-p_2)/n$, we have
	\begin{align}
		 & \pr(V' =V \mid V,W )=\frac{1}{n}\left(p_2V +(1-p_2)(n-V )\right) = 1-p_2 + (2p_2-1)\cdot \frac{V}{n}, \label{eq:y'=y} \\
		 & \E\left(\gD W \1_{V'=V}\mid \mvU \right)=-\gl W, \notag                                                             \\
		\text{and}\quad
		 & \E\left(|\gD W |^2\1_{V'=V}\mid \mvU\right)=2\gl\left(1+R_2\right), \notag
	\end{align}
	where $\mvU=(U_1, U_2, \ldots, U_n )$ and
	\[
		R_2=\frac{p_1p_3}{n(1-p_2)}\E\left(\sum_{i=1}^n p_1^{-2}(\1_{U_i=1}-p_1) + p_3^{-2}(\1_{U_i=3}-p_3)\ \bigl|\ \mvU\right).
	\]
	For any piece-wise three times differentiable function $F$ such that $F'=f\in\cA$, the random variable $(F(W' )-F(W ))\1_{V'=V=k}$ has mean zero. Thus using standard techniques we have
	\begin{align*}
		\E\biggl(\bigl( \gD W \cdot f(W )+\frac12 \abs{\gD W}^2\cdot f'(W )+R \bigr)\1_{V'=V =k}\biggr)=0,
	\end{align*}
	where $R \le \frac{\abs{f}_2}{6}|\gD W |^3$.
	Simplifying, we arrive at
	\begin{align*}
		\E\biggl(\left(-\gl W f(W )+\gl f'(W )(1+R_2)+R \right)\1_{V'=V}\1_{V =k}\biggr)=0.
	\end{align*}
	Thus
	\begin{align*}
		 & \E\left((f'(W )-W f(W ))\1_{V'=V}\mid V =k,W\right)\\
		 & \quad\le\abs{f}_1\cdot\E\left(\abs{R_2}\1_{V'=V}\mid V=k,W\right) +\frac{\abs{f}_2}{6\gl}\cdot\E\left(|\gD W |^3\1_{V'=V}\mid V =k,W\right)
		\lesssim n^{-\half}.
	\end{align*}
	By~\eqref{eq:y'=y} we have,
	$
		\pr(V'=V \mid V =k,W)\approx 1-p_2.
	$
	Thus if $k$ is near the mean of $V $, more specifically if $|k-np_2|\ll \sigma_V$ using relation~\eqref{eq:stein de} we conclude that
	\begin{align*}
		\dwas\bigl((W\mid V=k), Z\bigr) & \lesssim n^{-\half}.
	\end{align*}
	This completes the proof.
\end{proof}


\subsubsection{Uniform darts given number of misses}\label{ssec:darts}
Let $T\subset[0, 1]^2$ be a set, called target, of area $q\in (0, 1)$. Suppose $U_1, U_2, \dots$ are i.i.d. uniform random vectors taking values in $[0, 1]^2$, called darts. In particular, a dart hits the target with probability $\pr(U_i\in T)=q$ and misses with probability $\pr(U_i\notin T)=1-q=p$. Let $s:T\to\dR$ be a bounded non constant score function such that $\E s(U_1)\1_{U_1\in T}=0$ and $\E s(U_1)^2\1_{U_1\in T}=1$ .

Define
\[
	Y:=\sum_i\1_{U_i\notin T}-np \quad\text{ and }\quad W:=n^{-\half}\sum_i^ns(U_i)\1_{U_i\in T},
\]
as the centered number of darts that missed the target and the total score, respectively. Notice that $Y$ takes values in $\{-np\}+\dZ$ and that $\gs_Y^2=npq$.
\begin{lem}
	For $W$ and $Y$ as above and any $k\in \{np\}+\dZ$ with $|k|\ll n^{\half}$, we have that
	\begin{align*}
		\dwas\left((W\mid Y=k), Z\right) \lesssim n^{-\half}.
	\end{align*}
\end{lem}
\begin{proof}
	First, we notice that the $W$ and $Y$ are uncorrelated. As in Example~\ref{ex: urn}, we will consider Glauber dynamics conditioned on keeping the value of $Y$ to remain the same. Namely at time $n$ we choose a dart $I\sim \textrm{Uniform}\{1, \ldots, n\}$ and define
	\[
		W'=W-n^{-\half}s(U_I)\1_{U_I\in T}+n^{-\half}s(U')\1_{U'\in T},
	\]
	where $U'\sim \textrm{Uniform}\left([0, 1]^2\right)$ and independent of everything else. It is easy to see that resulting vectors $(W, Y)$ and $(W', Y')$ are exchangeable, and with a positive probability, the total score changes while the number of misses remains the same. Notice that we utilize that the score function $s(u)$ is not constant on $T$; otherwise, $Y'=Y$ would have implied $W'=W$.
	First we compute that
	\begin{align*}
		\E(\gD W\cdot \1_{Y'=Y}\mid \mvU )     & =-\frac{q}{n}W                    \\
		\textrm{ and }
		\E(|\gD W|^2\cdot \1_{Y'=Y}\mid \mvU ) & =\frac{2q}{n} \left(1+R_2\right),
	\end{align*}
	where $\mvU=(U_1, U_2, \ldots, U_n)$ and
	\[
		R_2:=-\frac{1}{2nq}Y+\frac{1}{2n}\E\left(\sum_{i=1}^n(s(U_i)^2\1_{U_i\in T}-1)\ \biggl|\ W, Y\right).
	\]
	Since $s(x)$ is bounded we can bound $|\gD W| \le\norm{s}_\infty$. By application of standard technique, we derive that for any $k\in \{np\}+\dZ$ with $|k|\ll n^{\half}$ we have
	\begin{align*}
		\dwas\bigl((W\mid Y=k), Z\bigr) & \le\sup_{f\in\cA}\left\{\abs{f}_1\cdot\E\abs{R_2}+\frac{\abs{f}_2}{6\gl}\E\left(|\gD W|^3\1_{Y'=Y}\mid Y=k\right)\right\}\lesssim n^{-\half}.
	\end{align*}
	This completes the proof.
\end{proof}


\subsubsection{Number of $01$'s given the number of $1$'s in a random binary sequence}\label{ssec:pattern1/2}
Let $(\go_1, \go_2, \ldots, \go_n)$ be a sequence of i.i.d. Bernoulli$\left(p\right)$ random variables with $\go_{n+1}=\go_1$. Let $V :=\sum_{i=1}^n \go_i$ be the number of $1$'s in it and define
\begin{align*}
	U :=\sum_{i=1}^{n}\1_{\go_i=0}\1_{\go_{i+1}}=\sum_{i=1}^{n}(1-\go_i)\go_{i+1}
\end{align*}
be the number of times a zero is followed by a one. Define
\[
	W^0 :=U-\E\left(U\mid V=m\right)
	\text{ and }
	W:={W^0}/{\gs_{W^0}}.
\]
One can easily compute that
\begin{align*}
	\E(U\mid V=m)                & = \frac{m(n-m)}{n-1}                                                                          \\
	\text{ and }
	\gs_{W_0}^2 =\var(U\mid V=m) & = \frac{ \binom{m}{2} \binom{n-m}{2} }{ (n-1)\binom{n-1}{2} }\approx n\cdot (m/n)^2(1-m/n)^2.
\end{align*}
Moreover, we have the following result.

\begin{lem}\label{lem:01 class}
	Let $W$ and $V$ be as above. For any $m\in\dN$ with $m/n\in (\eps, 1-\eps)$ for some $\eps\in (0, \half)$, we have that
	\begin{align*}
		\dwas\left((W \mid V=m), Z\right)\lesssim n^{-\half}.
	\end{align*}
\end{lem}

\begin{proof}
	We will work with the random variable $X:=V-U=\sum_{i=1}^n\go_i\go_{i+1}$. Consider the following construction of an exchangeable pair. For a binary sequence $\mvgo=(\go_1, \go_2, \ldots, \go_n)$, pick two indices $I$ and $J$
	uniformly at random from $\{1, 2, \ldots, n\}$ and create a new binary sequence by swapping the bits at those locations $\mvgo'=(\go_1', \go_2', \ldots, \go_n')$
	such that $\go_i'=\go_i$ for all $i\notin\{I, J\}$, $\go_I'=\go_J$, and $\go_J'=\go_I$. We have
	\begin{align}\label{eq:gDX 01}
		 & \gD X
		=(\go_I-\go_J)(\go_{J+1}+\go_{J-1}-\go_{I+1}-\go_{I-1}) -(\go_I-\go_J)^2\1_{|I-J|=1}.
	\end{align}
	One can easily check that $|\gD X |\le 2$ and $$\E(\gD X \mid V=m) = - \frac{4(n-1)}{n^{2}}\left(X-\frac{m(m-1)}{n-1}\right)$$
	This gives another proof of the fact that $\E\left(X\mid V=m\right)=\frac{m(m-1)}{n-1}$. In particular, we have
	\begin{align*}
		\E(\gD W \mid V=m)                    & = - \frac{4(n-1)}{n^{2}}W      \\
		\text{and }\E(\abs{\gD W}^2 \mid V=m) & = \frac{8(n-1)}{n^{2}}(1+R_2),
	\end{align*}
	where $R_{2}$ is a centered random variable with $ \E\abs{R_2}\lesssim n^{-\half}$ that can be computed explicitly using~\eqref{eq:gDX 01}.
	By a similar derivation as in the Lemma~\ref{ex: urn} we have that
	\begin{align*}
		\dwas\bigl((W\mid V=m), Z\bigr) & \le\sup_{f\in\cA}\left\{\abs{f}_1\E\abs{R_2}+\frac{\abs{f}_2}{6\gl}\E\left(|\gD W|^3\mid V=m\right)\right\}
		\lesssim \frac{n}{n^{3/2}}= n^{-\half}.
	\end{align*}
	This completes the proof.
\end{proof}

\subsubsection{Number of wedges in a uniform graph with $m$ edges}\label{ssec:degsq}
Let $\cG\sim G(n, m)$, a graph on $n$ vertices with $m$ edges chosen uniformly at random.
For Erd\H{o}s--R\'enyi random graph $G_{n, p}$ the number of edges $E$ is a sufficient statistic for the parameter $p$, thus the
model $(G_{n, p}\mid E=m)$ is equivalent to $G(n, m)$. In other words, deriving CCLT in $G_{n, p}$ conditioned on the number of edges being $m$ can be converted into deriving regular CLT in $G(n, m)$.

Let $d_i=\sum_{j}\1_{i\sim j}$ be the degree of vertex $i$ in $\cG$, note that $\sum_i d_i=2m$ and $\overline{d}_i:=d_i-\frac{2m}{n}$. The number of wedges $U:=\sum \1_{i\sim j}\1_{j\sim \ell}$ and can be written as
\[
	U=\frac{1}{2}\sum_i d_i(d_i-1)=\frac{1}{2}\sum_i d_i^2-m.
\]
\begin{lem}\label{lem:wedge class}
	We have $\E(U)=2m(m-1)/(n+1)$. Moreover, for $m\in\dN$ with $m/N\in (\eps, 1-\eps)$ for some $\eps\in(0, \half)$, we have that
	\begin{align*}
		\dwas((U-\E U)/\gs_U, Z)\lesssim n^{-\half}.
	\end{align*}
\end{lem}

\begin{proof}
	Define $W:=(U-\E U)/\gs_U$ and $X:=\frac{1}{2}\sum_i\overline{d}_i^2$. Notice that \[
		U=X +\frac{2m^2}{n} - m.
	\]
	Similar to the example in Section~\ref{ssec:pattern1/2} we will consider the Markov Chain that swaps two uniformly chosen edges to create an exchangeable pair preserving the total amount of edges. Then
	\begin{align*}
		\gD d_i=\begin{cases}+1 & \textrm{w.p.~}\frac{1}{4 N \left( N -1\right)}(n-1-d_i)(m-d_i)                \\
             -1 & \textrm{w.p.~}\frac{1}{4 N \left( N -1\right)}d_i\left( N -m-(n-1-d_i)\right) \\
             0  & \textrm{otherwise}.
		        \end{cases}
	\end{align*}
	Simple computations show that, for $i=1, 2, \ldots, n$, we have
	\begin{align*}
		\E\left(\gD d_i \mid \cG\right)
		 & =-\frac{1}{4\left( N -1\right)}\overline{d_i}
		\qquad \text{ and}                                                                                                                                                                            \\
		\E\left(|\gD d_i|^2 \mid \cG\right)
		 & =\frac{1}{4\left( N -1\right)}\left(\frac{2(n-1)(n-2)}{n}\frac{m}{N}\left( 1-\frac{m}{N}\right)+\frac{n-4}{n}\left(1-\frac{2m}{N}\right)\overline{d}_i+\frac{2\overline{d}_i^2}{N}\right).
	\end{align*}
	Therefore, using
	\begin{align*}
		\gD X=\frac12\sum_{i=1}^n\gD d_i\left(\gD d_i+2\overline{d}_i\right).
	\end{align*} we have that
	\begin{align*}
		\E(\gD X\mid\cG) & =-\frac{1}{2N}\left(X-\frac{n-1}{n+1}\cdot m\left(1-\frac{m}{N}\right)\right).
	\end{align*}
	In particular, this implies that
	\[
		\E(X)=\frac{n-1}{n+1}\cdot m\left(1-\frac{m}{N}\right)
		\text{ and }
		\E(U)=\frac{2 m (m - 1)}{n + 1}.
	\]

	So, after scaling by $\gs_U\approx n^{3/2}$, we have
	\begin{align*}
		\E\left(\gD W\mid \cG\right)                    & =-\frac{1}{2 N}W        \\
		\text{and }\E\left(\abs{\gD W}^2\mid \cG\right) & =\frac{1}{2 N}2(1+R_2),
	\end{align*}
	where $R_2$ has mean zero and $\E\abs{R_2}\lesssim n^{-\half}$.
	Letting $\gl:=\frac{1}{2 N}$. To derive CCLT it remains to upper bound the error term $\gl^{-1}\E\left(|\gD W|^3\mid E=m\right)$, which we do as follows
	\begin{align*}
		\gl^{-1}\E\bigl(\abs{\gD W}^3\bigr)
		 & =\gl^{-1}\E\left(|\gD W|\cdot\gD W^2\right)
		\lesssim \gl^{-1}n^{-\half}\cdot 2\gl\approx n^{-\half},
	\end{align*}
	and get that
	\begin{align*}
		\dwas(W, Z)\le
		\sqrt{\frac{2}{\pi}}\E\abs{R_2}+\frac{2}{3\gl}\E\abs{\gD W} ^3\lesssim n^{-\half}
	\end{align*}
	to complete the proof.
\end{proof}


\subsection{Applications of main results in one dimension}\label{ssec:applications of main thm}
In this section, we present several applications of Theorem~\ref{th: imp symcase}. In Section~\ref{ssec:even odd} we present a CCLT where due to inhomogeneity of the model, swapping the Markov chain does not give an exchangeable pair, and hence full strength of our main result is needed. In Sections~\ref{ssec:pattern} and~\ref{ssec:wedgeedge} we explore the same models as in Sections~\ref{ssec:pattern1/2} and~\ref{ssec:degsq}, respectively, but under with exchangeable pair created by Glauber dynamics rather than the swapping Markov chain.
We bounded all of the error terms as described in Remark~\ref{rem: eps}, and hence we have the $\eps$ in the exponent.

\subsubsection{Difference between the number of $11$ patterns that start at odd and even bits in a two-species binary sequence}\label{ssec:even odd}

Let $\mvgo=(\go_1, \go_2, \ldots, \go_n)$ be a sequence of independent Bernoulli$\left(p_i\right)$ random variables. We assume that $n$ is even and $\go_{n+1}:=\go_1$. Let $p_i$ be equal to $p$ if $i$ is odd and equal to $q=1-p$ if $i$ is even.
Consider the random variable $V:=\sum_{i=1}^{n}\go_i$, the number of $1$'s in the $\mvgo$, and $X :=\sum_{i=1}^{n}(-1)^i\go_i\go_{i+1}$, the difference between the number of $11$'s that start at even and odd positions.
Define $Y :=V-\frac{n}{2}$, notice that $\E Y=0$ and $\gs_Y^2=npq$. Also, notice that $Y$ is not a sufficient statistic for $p$, making this example particularly interesting. The random variable $X$ can be rewritten as
\begin{align*}
	X & =\sum_{i=1}^{n}(-1)^i\go_i\, \go_{i+1}=\sum_{i=1}^{n}(-1)^i\bgo_i\, \bgo_{i+1}.
\end{align*}
This representation of $X$ is particularly convenient in computations for the variance $\gs_X^2=np^2(1-p)^2$ and a variety of terms in the following lemma.

\begin{lem}\label{lem:11dif}
	Let $X$ and $Y$ be as above, define $W:=\frac{X}{\gs_{X}}$. For $k\in\dZ$ with $|k|\ll n^{\half}$, we have that
	\begin{align*}
		\dwas\left((W \mid Y =k)\, , \, Z\right)\lesssim n^{-\half+\eps}.
	\end{align*}
\end{lem}
\begin{proof}
	Notice that $W$ and $Y$ are uncorrelated random variables, $\gs_{X}^2=np^2q^2$, and consider the following construction of an exchangeable pair. Pick a position $I$ uniformly at random and replace it with an independent $\go_I'\sim$ Bernoulli$(p_I)$.
	It follows that $(W, Y)$ satisfies Assumptions~\ref{ass:uncorr and exch}, ~\ref{ass:Y}, ~\ref{ass:M1} and~\ref{ass:M2} with $\gl=\frac{1}{n}$ and $\psi=2$ In particular,
	\begin{align*}
		M_{1, \pm}(W, Y) & =-\frac{1}{n}(p+q)W=-\frac{1}{n}\frac{1}{2}\cdot2W, \\
		M_{2, \pm}(W, Y) & =\frac{1}{n}\left(2+R_{2, \pm}\right),
	\end{align*}
	where
	\begin{align*}
		R_{2, +}            & =-\sum_{i=1}^n p_i\bgo_{i-1}^2\bgo_{i}+ p_i\bgo_{i+1}^2\bgo_{i}+2pq\bgo_{i-1}\bgo_{i+1}-2p_i\bgo_{i-1}\bgo_{i}\bgo_{i+1} \\
		\text{and }	R_{2, -} & =\sum_{i=1}^n p_i\bgo_{i-1}^2\bgo_{i}+ p_i\bgo_{i+1}^2\bgo_{i}-2pq\bgo_{i-1}\bgo_{i+1}-2p_i\bgo_{i-1}\bgo_{i}\bgo_{i+1}.
	\end{align*}
	Notice that $R_{2, +}$ and $R_{2, -}$ have means equal to zero and variances of order $n$.

	For $\hat{k}\in\{k-1, k\}$ with $\abs{k}\ll\gs_Y$, by LLT we have \[\pr(Y=\hat{k})\approx\gs_Y^{-1}\approx n^{-\half}.\] Thus the error terms from Theorem~\ref{th: imp symcase} can be upper bounded as follows
	\begin{align*}
		 & \E(\abs{R_{2, \pm}}\mid Y=\hat{k})\lesssim n^{-\half+\eps},        \\
		 & \E(\abs{W}\abs{R_{0, \pm}}\mid Y=\hat{k})\lesssim n^{-\half+\eps}, \\
		 & \E(\abs{\gD W}\mid Y=\hat{k})\lesssim n^{-\half+\eps},
	\end{align*}
	and
	\begin{equation*}
		\abs{\E\left(\abs{\gD W}^3 \mid Y\in\{k-1, k\}\right)}=\abs{\E\left(|\gD W|\cdot\gD W^2\mid Y\in\{k-1, k\}\right)}\lesssim \gl n^{-\half+\eps}.
	\end{equation*}

	In the last bound we used that $\abs{\gD W}\le 2/\gs_{X}$.
	Thus by Theorem~\ref{th: imp symcase} we get that
	\begin{align*}
		\dwas\left((W \mid Y =k)\, , \, Z\right)\lesssim n^{-\half+\eps}
	\end{align*}
	where the constant in the right hand side depends on $p$ and $\eps$.
\end{proof}
\begin{rem}[Change of variable and the conditional mean in Lemma~\ref{lem:11dif}]\label{rem:cond mean 11dif}
	In the lemma above, we do not use the change of variable because the model already satisfies Assumptions~\ref{ass:uncorr and exch}, ~\ref{ass:M1} and~\ref{ass:M2}. It is already symmetric due to the choice of $p_{2i}=1-p_{2i+1}$, for general values of $p_{2i}$ and $p_{2i+1}$ one would need to have the $\gl\psi \ga XY$ term as in~\eqref{eq:change of var}. However, the square and the cubic terms are not needed here because $\E (W\mid Y=k)\approx0$, which follows from the fact that
	\[
		\E\left(\gD W\mid W, Y\right)=-\frac{2}{n}W.
	\]
\end{rem}

\subsubsection{Number of $01$'s given the number of $1$'s in a random binary sequence}\label{ssec:pattern}

Similarly to the model in Section~\ref{ssec:pattern1/2}, let $(\go_1, \go_2, \ldots, \go_n, \go_1)$ be a sequence of independent Bernoulli$\left(p\right)$ random variables, with ends glued together for simplicity. Let $V:=\sum_{i=1}^{n}\1_{\go_i=1}$ be the number of $1$'s in it and define $U:=\sum_{i=1}^{n}\1_{\go_i=0}\1_{\go_{i+1(\md n)}=1}$ be the number of times zero is followed by a one.
Define $Y :=V -np$ and $X :=U -(1-2p)Y -np q$. Notice that $\E Y=0$, $\gs_Y^2=npq$ and $Y$ takes values in $\gz+\dZ$ for $\gz:=\{-np\}$.

\begin{lem}\label{lem:01given1}
	For random variables $X$ and $Y$ as above and $\ga=\frac{2p-1}{2p q}$ define
	\begin{align}\label{eq:change in pattern}
		W^0:=X+\frac{2\ga}{n}XY+\frac{1}{n}\left(Y^2-\E Y^2\right)+\frac{2\ga}{n^2} (Y^3-\E Y^3) \quad\textrm{ and }\quad W=\frac{W^0}{\gs_{W^0}}.
	\end{align}
	For $k\in \gz+ \dZ$ with $|k|\ll\gs_Y$ we have
	\begin{align*}
		\dwas\left((W \mid Y =k), Z\right)\lesssim n^{-\half+\eps}.
	\end{align*}
\end{lem}
\begin{proof}
	One can check that $X$ and $Y$ are mean zero uncorrelated random variables satisfying Assumptions~\ref{ass:uncorr and exch}, ~\ref{ass:Y}, ~\ref{ass:M1g} and~\ref{ass:M2} with $Q=pq$, $\gl=\frac{1}{n}$ and $\psi=2$. In particular, we have
	\begin{alignat*}{3}
		 & M_{0, +}(X , Y ) = p q- \frac{1}{n}pY,
		 &                                                          & M_{0, -}(X , Y ) = p q+ \frac{1}{n} qY,                   & \\
		 & M_{1, +}(X , Y ) =- \frac{1}{n}\left(2pX -2 p qY \right)
		 & \text{\ and\ }                                           & M_{1, -}(X , Y ) =- \frac{1}{n}\left(2qX +2 p qY \right). &
	\end{alignat*}
	Moreover,
	\begin{align*}
		M_{2, +}(X , Y ) & =\frac{1}{n}\left(2\cdot np^2 q^2+4(1-2p)pX +4p q(1-3p)Y -2p\E(\overline{\#001}\mid X, Y)\right)     \\
		\text{and }
		M_{2, -}(X , Y ) & = \frac{1}{n}\left(2\cdot np^2 q^2-4(1-2p) qX +4p q(2-3p)Y - 2q\E(\overline{\#011}\mid X, Y)\right),
	\end{align*}
	where $\overline{\#011}$ and $\overline{\#001}$ are centered random variables that count the number of times the respective (consecutive) sub-sequence appears in the sequence.
	We apply the change of variable~\eqref{eq:change of var} with
	$\gl=\frac1n$, $\psi =2$, $\ga=\frac{2p-1}{2p q}$, and $\theta=2$ to define the random variable $W^0$ and its scaled version $W$ as in the statement of the Lemma~\ref{lem:01given1}.

	For $i\in\{k-1, k\}$ with $\abs{k}\ll\gs_Y$, by LLT we have \[\pr(Y=i)\approx\gs_Y^{-1}\approx n^{-\half}\]
	By Proposition~\ref{prop: change}, $(W, Y)$ satisfies Assumptions~\ref{ass:uncorr and exch}, ~\ref{ass:Y}, ~\ref{ass:M1} and~\ref{ass:M2}, in particular
	\begin{align*}
		M_{1, \pm}\left(W , Y \right) & =-\frac{1}{n}\biggl(\frac12\cdot2 W +\widetilde{R}_{1, \pm}\biggr), \quad\textrm{where}\quad \E\left(|\widetilde{R}_{1, \pm}|\ \biggl|\ Y =i\right)\lesssim n^{-\half+\eps}.
	\end{align*}
	and
	\begin{align*}
		M_{2, \pm}\left(W , Y \right) & =\frac{1}{n}\left(2+\widetilde{R}_{2, \pm}\right), \quad\textrm{where}\quad\E\left(|\widetilde{R}_{2, \pm}|\ \biggl|\ Y =i\right)\lesssim n^{-\half+\eps}.
	\end{align*}

	The remaining error terms can be bounded by
	\begin{equation*}
		\abs{\E\left(\abs{\gD W}^3 \mid Y\in\{k-1, k\}\right)}=\abs{\E\left(|\gD W|\cdot\gD W^2\mid Y\in\{k-1, k\}\right)}\lesssim \gl n^{-\half-\eps}.
	\end{equation*}
	and
	\begin{align*}
		 & \E\left(\abs{W}\cdot \abs{R_{0, \pm}}\mid Y=i\right)\lesssim n^{-\half+\eps}, \quad\textrm{ and }\quad\E\left(\abs{\gD W}\mid Y=i\right)\lesssim n^{-\half+\eps}.
	\end{align*}
	Thus by Theorem~\ref{th: imp symcase}
	for any number $k\in \gz+\dZ$ such that $|k |\ll \sqrt{n}$ we conclude that
	\begin{align*}
		\dwas\left((W \mid Y =k), Z\right)\lesssim n^{-\half+\eps}
	\end{align*}
	and complete the proof.
\end{proof}

\begin{rem}[Change of variable and the conditional mean in Lemmas~\ref{lem:01given1} and~\ref{lem:01 class}]\label{rem:cond mean pattern}
	We define $m=np+k$, so that $\{V=m\}=\{Y=k\}$. With the notations as above, using Lemma~\ref{lem:01 class} one gets the exact formula for the conditional mean given by
	\begin{align*}
		\E \left(X\mid Y=k\right)
		 & = \frac{(np+k)(nq-k)}{n-1} -(1-2p)k -npq              \\
		 & =\frac{np q}{n-1}+(1-2p)\frac{k}{n-1}-\frac{k^2}{n-1}
		=\frac{np q}{n-1}\left(1-\frac{2\ga k}{n}-\frac{k^2}{np q}\right),
	\end{align*}
	where $\ga=\frac{2p-1}{2p q}$ is the same as $\ga$ in the change of variable~\eqref{eq:change in pattern}. In particular, we have
	\begin{align*}
		\E \left(X+\frac{2\ga}{n}XY\ \biggl|\ Y=k\right)
		 & =\E \left(\left(1+\frac{2\ga k}{n}\right)X\ \biggl|\ Y=k\right)                                \\
		 & =\frac{np q}{n-1}\left(1-\frac{4\ga^2k^2}{n^2}-\frac{k^2}{np q}-\frac{2\ga k^3}{n^2p q}\right) \\
		 & =-\frac{k^2-np q}{n-1}-\frac{2\ga k^3}{n(n-1)}.
	\end{align*}
	This matches the $Y^2$ and $Y^3$ terms in~\eqref{eq:change in pattern} upto a small error caused by dividing by $n-1$ instead of $n$.
\end{rem}


\subsubsection{number of wedges given the number of edges in a random graph}\label{ssec:wedgeedge}

Let $G_{n, p}$ be the Erd\H{o}s-R\'enyi random graph on $n$ vertices. Consider $E :=\sum_{x<y}\go_{xy}$, the number of edges in $G_{n, p}$ and its centered version $Y:=E -\E E$. Notice that $Y$ takes values in $\gz+\dZ$, where $\gz:=\{-Np\}$, and $\gs_Y^2=Npq$. Let
\[
	U :=\sum_{x< y, z\neq x, y}\go_{xy}\go_{yz}
\]
be the number of wedges in $G_{n, p}$ and
\begin{equation}\label{eq:uncor wedge}
	X :=\sum_{x< y, z\neq x, y}(\go_{xy}-p)(\go_{yz}-p)=U -2(n-2)pY -(n-2) N p^2.
\end{equation}
It is also straight forward to check that $X$ and $Y$ are uncorrelated centered random variables with
\begin{align*}
	\sigma_X^2=\frac12 n(n-1)(n-2)p^2 q^2.
\end{align*}
\begin{lem}\label{lem:wedgeedge} Let random variables $X$ and $Y$ as above and $\ga=\frac{2p-1}{2p q}$ define
	\begin{align}\label{eq:change wedgeedge}
		W^0:=X+\frac{2\ga}{N}XY+\frac{n-2}{N}\left(Y^2-\E Y^2\right)+2\frac{n-2}{N^2}\ga (Y^3-\E Y^3)\textrm{ and } W=\frac{W^0}{\gs_{W^0}}.\end{align}
	For $k\in \zeta+ \dZ$ with $|k|\ll n$ we have
	\begin{align*}
		\dwas\left((W \mid Y =k)\, , \, Z\right)\lesssim n^{-\half+\eps}.
	\end{align*}
\end{lem}
In the following proof we omit most of the computations, even though they are typical for such applications, we present them in Appendix~\ref{sec: wedge comp}.
\begin{proof}
	We first compute the following terms
	\begin{align*}
		M_{0, +}(X, Y) & =p q-\gl pY\quad\textrm{ and } \quad M_{0, -}(X, Y)=p q+\gl qY.
	\end{align*}

	Next we compute
	\begin{align*}
		M_{1, +}(X , Y )   & =-\frac{1}{N}\left(2pX -2(n-2)p qY \right)            \\
		M_{1, -}(X , Y )   & =-\frac{1}{N}\left( 2qX +2(n-2)p qY \right)           \\
		M_{2, \pm}(X , Y ) & =\frac{1}{N}\left(2\cdot\sigma_X^2+R_{2, \pm}\right),
	\end{align*}
	where using computations form Section~\ref{sec: wedge comp} one can see that $\norm{R_{2, \pm}}_p\lesssim n^{5/2}$. We apply the change of variable~\eqref{eq:change of var} with $\gl=\frac{1}{N}, \ga=\frac{2p-1}{2p q}$, and $\theta=2(n-2)$ to define the random variable $W^0$ and its scaled version $W$ as in the statement of the Lemma~\ref{lem:wedgeedge}.
	For all $i\in\gz+\dZ$ with $\abs{i}\ll n$, by LLT $\pr(Y=i)\approx\gs_Y^{-1}\approx n^{-1}$. By Proposition~\ref{prop: change} we have that
	\begin{align*}
		M_{1, \pm}(W , Y ) & =-\frac{1}{N}\biggl(\frac12\cdot2 W +\widetilde{R}_{1, \pm}\biggr)\qquad\text{and}\qquad M_{2, \pm}\left(W , Y \right) & =\frac{1}{N}\biggl(2+\widetilde{R}_{2, \pm}\biggr),
	\end{align*}
	where
	\begin{align*}
		\E\left(|\widetilde{R}_{1, \pm}|\ \bigl|\ Y=i\right)
		\lesssim n^{-\frac32+\eps}
		\qquad\text{and}\qquad
		\E\left(|\widetilde{R}_{2, \pm}|\ \bigl|\ Y=i\right)
		\lesssim n^{-\half+\eps}.
	\end{align*}
	Thus $(W, Y)$ satisfies Assumptions~\ref{ass:uncorr and exch}, ~\ref{ass:Y}, ~\ref{ass:M1} and~\ref{ass:M2} with $Q=pq$, $\gl=\frac{1}{N}$ and $\psi=2$.
	The remaining terms can be bounded as follows:
	\begin{equation*}
		\frac{1}{\gl}\E \left(\abs{\gD W}^3\mid Y\in\{k-1, k\} \right)\lesssim n^{-\frac{1}{2}+\eps},
	\end{equation*}
	\begin{align*}
		 & \E\left(\abs{W}\abs{R_{0, \pm}}\mid Y=i\right)\lesssim n^{-2+\eps} \quad\textrm{ and }\quad \E\left(\abs{\gD W}\mid Y=i\right)\lesssim n^{-\half+\eps}.
	\end{align*}
	Thus, Theorem~\ref{th: imp symcase} for any $k\in \gz+\dZ$ such that $|k |\ll n$ we conclude that
	\begin{align*}
		\dwas\bigl((W \mid Y =k)\, , \, Z\bigr)\lesssim n^{-\half+\eps}
	\end{align*}
	for $\eps>0$ small.
\end{proof}

\begin{rem}[Change of variable and the conditional mean in Lemma~\ref{lem:wedgeedge} and~\ref{lem:wedge class}]\label{rem:cond mean change}
	With the notations as above and $m=Np+k$, we have $\{E=m\}=\{Y=k\}$. Now, using Lemma~\ref{lem:wedge class}, one gets the exact conditional mean
	\begin{align*}
		\E \left(X\mid Y=k\right)
		 & =\frac{2m(m-1)}{n+1} -2(n-2)pk -(n-2) N p^2                           \\
		 & =-2(1-2p)\frac{k}{n+1}+\frac{2k^2}{n+1}-\frac{2 Np q}{n+1}            \\
		 & =-\frac{2 Np q}{n+1}\left(1-\frac{2\ga k}{N}-\frac{k^2}{Np q}\right),
	\end{align*}
	where $\ga=\frac{2p-1}{2p q}$ is the same as $\ga$ in the change of variable~\eqref{eq:change wedgeedge}. In particular, we have
	\begin{align*}
		\E & \left(X+\frac{2\ga}{N}XY\ \biggl|\ Y=k\right)                                                                                  \\
		   & =\E \left(\left(1+\frac{2\ga k}{N}\right)X\ \biggl|\ Y=k\right)                                                                \\
		   & =-\frac{2 Np q}{n+1}\left(1-\frac{4\ga^2k^2}{N^2}-\frac{k^2}{Np q}\left(1+\frac{2\ga}{N}k\right)\right)                        \\
		   & =2\frac{k^2- Np q}{n+1}+\frac{4\ga k^3}{N\left(n+1\right)} - \frac{8pq\ga^2k^2}{N(n+1)}                                        \\
		   & = \frac{N}{N-1}\left( \frac{n-2}{N}\cdot(k^2- Np q)+2\cdot\frac{(n-2)}{N^2}\cdot \ga k^3 \right) - \frac{8pq\ga^2k^2}{N(n+1)}.
	\end{align*}
	This matches $Y^2$ and $Y^3$ terms in~\eqref{eq:change wedgeedge} with a $1+O(1/N)$ factor.
\end{rem}


\subsection{Applications of main result in higher dimension}\label{ssec:applications multistein}
In this section, we present several applications of Theorem~\ref{th: multi case}, the multivariate result.

\subsubsection{Multivariate version of darts model given number of misses}\label{ssec:Multi darts}
Similar to the model from Section~\ref{ssec:darts}, let $\{S_i\}$ be a sequence of i.i.d. random variables with mean zero and unit variance uniformly bounded by $s>0$. Let $\{V_i\}$ be a sequence of i.i.d. Bernoulli$(\half)$ random variables. While the following lemma could be derived using classical methods analogous to the methods described in Section~\ref{ssec:applications of classic}, we use it as a toy example to illustrate an application of Theorem~\ref{th: multi case}.
\begin{lem}
	Let $\mvX:=\left(\sum_iS_i\, \overline{V}_i, \sum_i S_i\right)^T$ and $Y:=\sum_i \overline{V}_i$.
	Define
	\begin{align*}
		\mvW^0
		=\mvX-\frac{1}{n}
		\begin{pmatrix}0 & \half \\
               2 & 0
		\end{pmatrix}
		\mvX Y\quad \text{and \quad} \mvW:=\left(W^0_1/\gs_{W^0_1}, W^0_2/\gs_{W^0_2}\right)^T
	\end{align*}
	For any $k\in \{\frac{n}{2}\}+\dZ$ with $|k|\ll \sqrt{n}$ we have that
	\begin{align*}
		\dwas\left((\mvW\mid Y=k), \mvZ\right) & \lesssim n^{-\half+\eps}.
	\end{align*}
\end{lem}

\begin{proof}
	Consider the following Markov Chain to generate an exchangeable pair. Let $I$ be chosen uniformly at random from $\{1, 2, \ldots, n\}$. Replace both $S_I$ and $V_I$ with independent copies $S_I'$ and $V'_I$, respectively. It is easy to check that
	\begin{align*}
		M_{1, \pm}(\mvX, Y)=-\frac{1}{n}\begin{pmatrix}\frac{1}{4}&\mp\frac{1}{8}\\\mp\frac{1}{2}&\frac{1}{4}\end{pmatrix}\mvX
	\end{align*}
	and\begin{align*}
		\E\left(\gD\mvX\gD\mvX^T\1_{\DY =\pm1}\mid\mvX, Y\right)=\frac1n\left(\begin{pmatrix}
			                                                                      \frac12&0\\0&\frac12\end{pmatrix}\begin{pmatrix}\frac{n}{4}&0\\0&n\end{pmatrix}+\Gamma_{2, \pm}\right).
	\end{align*}
	For $k\in\{\frac{n}{2}\}+\dZ$ with $|k|\ll \sqrt{n}$ we have
	$\E\left(\norm{\Gamma_{2, \pm}}_{\hs}\mid Y=k\right)\lesssim n^{-\half+\eps}$.

	Define
	\begin{align*}
		\Psi_\pm:=\begin{pmatrix}\frac{1}{4}&\mp\frac{1}{8}\\\mp\frac{1}{2}&\frac{1}{4}\end{pmatrix},
	\end{align*}
	and notice that
	\begin{align*}
		\Psi=\Psi_++\Psi_-= \begin{pmatrix}	\frac12&0\\0&\frac12\end{pmatrix}.
	\end{align*}
	Since $\sum_iS_i\, \overline{V}_i$ and $\sum_i S_i$ are uncorrelated, it follows $\mvX$ satisfies Assumptions~\ref{ass:uncorr and exch}, ~\ref{ass:Y}, ~\ref{ass:M1g} (with $\mvb_\pm=0$), and~\ref{ass:M2} with variance-covariance matrix
	$
		\begin{pmatrix}n/4&0\\0&n\end{pmatrix}.$ Using the change of variable~\eqref{eq:multichange} and Proposition~\ref{prop: multichange} we get that for
	\begin{align*}
		\mvW^0:=\mvX-\frac{1}{n}\begin{pmatrix}0&\frac{1}{2}\\2&0\end{pmatrix} \mvX Y \quad\text{and \quad} \mvW:=\left(W^0_1/\gs_{W^0_1}, W^0_2/\gs_{W^0_2}\right)^T
	\end{align*}
	$(\mvW, Y)$ satisfies Assumptions~\ref{ass:uncorr and exch}, ~\ref{ass:Y}, ~\ref{ass:M1}, and~\ref{ass:M2}.
	Furthermore, notice that $|\gD W_i|\le4s/\sqrt{n}$ thus for all $|k|\ll\sqrt{n}$
	\begin{align*}
		\frac{1}{\gl}\E\left(\abs{\DW}^4\mid Y=k\right) & \lesssim \frac{s^2}{\gl n}\E\left(\abs{\DW}^2\mid Y=k\right)\lesssim n^{-1+2\eps}.
	\end{align*}
	Finally, using Theorem~\ref{th: multi case} with $\E\abs{\mvW}^4<\infty$ we have that for any number $k\in \{\frac{n}{2}\}+ \dZ$ with $|k | \ll\sqrt{n}$ the following bound holds
	\begin{align*}
		\dwas\left((\mvW\mid Y=k), \mvZ\right) & \lesssim n^{-\half+\eps}
	\end{align*}
	for some $\eps>0$.
\end{proof}

\subsubsection{Number of $($triangles, wedges$)$ given the number of edges in a random graph}\label{ssec:triangle}
As before, let $G_{n, p}$ be the Erd\H{o}s--R\'enyi random graph and denote
$\go_{xy}:=\1_{x\sim y}$.

Define
\[
	E :=\sum_{x<y}\go_{xy}
	, \quad
	U :=\sum_{x< y, z\neq x, y}\go_{xy}\go_{yz}
	, \quad\text{and}\quad
	T :=\sum_{x<y<z}\go_{xy}\go_{yz}\go_{zx},
\]
the number of edges, wedges, and triangles, respectively, in $G_{n, p}$. It is straight forward to check that, letting $\gl=\frac{1}{N}$, we get
\begin{align*}
	 & M_{1, +}(U, E) =-\gl\left(2pU-2p(n-2)E\right) & M_{1, -}(U, E) =-\gl2 qU  \\
	 & M_{1, +}(T, E) =-\gl(3pT-pU)                  & M_{1, -}(T, E) =-\gl3 qT.
\end{align*}
Define $Y:=E-\E E$ and notice that $\E Y=0$, $\gs_Y^2=Npq$, and it takes values in $\gz+\dZ$, where $\gz:=\{-Np\}$. Moreover,
\begin{align*}
	M_{0, +}(U, T, Y)=p q-\gl pY \quad \textrm{ and } \quad M_{0, -}(U, T, Y)=p q+\gl qY.
\end{align*}

Considering the following representation of our random variables allows us to pass to uncorrelated random variables in a natural way,
\begin{align}
	\widetilde{T} & =\sum_{i<j<k}(\go_{ij}-p)(\go_{jk}-p)(\go_{ki}-p)=T-p(U-\E U)+p^2(n-2)Y-\binom{n}{3}p^3\label{eq:triangle count} \\
	\widetilde{U} & =\sum_{j< k, i\neq j, k}(\go_{ij}-p)(\go_{ik}-p)=U-2p(n-2)Y+\frac{n(n-1)(n-2)}{2}p^2.\label{eq:wedge count}
\end{align}
Notice that $\widetilde{U}=X$ as in \eqref{eq:uncor wedge} from Section~\ref{ssec:wedgeedge}.
\begin{lem}\label{lem:twe}
	Let $\widetilde{T}$, $\widetilde{U}$, and $Y$ be as above. Define $\mvX=\left(\widetilde{T}, \widetilde{U}\right)^T$
	and
	\begin{align}\label{eq:change twe}
		\mvW^0:=\mvX+\gl \Alpha \mvX Y+\frac{\gl\mvgth}{2}\left(Y^2-\E Y^2\right)+\frac{\gl^2(\Alpha+\ga)\mvgth}{3}(Y^3-\E Y^3),
	\end{align}
	where
	\[
		\gA=\frac{1}{2Q}\begin{pmatrix}
			3(2p-1)&0\\-2p q&(2p-1)\end{pmatrix}, \quad \ga=\frac{2p-1}{2p q}, \quad\text{and}\quad\mvgth=-\begin{pmatrix}
			0\\- 2(n-2)\end{pmatrix}.
	\]
	Define the random vector
	\begin{equation}\label{eq:tw}
		\mvW:=\begin{pmatrix}
			\widehat{T} \\\widehat{U} \end{pmatrix}:=\begin{pmatrix}
			W^0_1/\gs_{W^0_1}\\W^0_2/\gs_{W^0_2} \end{pmatrix}.
	\end{equation}
	For any $k\in \gz+ \dZ$ with $|k|\ll n$ we have that
	\begin{align*}
		\dwas\left((\mvW\mid Y=k), \mvZ\right) & \lesssim{n}^{-\half+\eps}.
	\end{align*}

\end{lem}
\begin{proof}
	Computing $M_{1, \pm}$ terms yields that

	\begin{align*}
		M_{1, +}\left(\mvX, Y\right) & =-\frac{1}{N}\left(
		\begin{pmatrix}
				3p & -2pq \\
				0  & 2p
			\end{pmatrix}\mvX
		-\begin{pmatrix}0\\ 2(n-2)pq\end{pmatrix}Y \right)
	\end{align*}
	and
	\begin{align*}
		M_{1, +}\left(\mvX, Y\right) & =-\frac{1}{N}\left(
		\begin{pmatrix}
				3q & 2pq \\
				0  & 2q
			\end{pmatrix}\mvX
		+\begin{pmatrix}0\\ 2(n-2)pq\end{pmatrix}Y \right).
	\end{align*}
	Computing the second order terms yields
	\begin{align*}
		\E\left(\begin{pmatrix}
			        \gD \widetilde{T}^2 & \gD \widetilde{T} \gD\widetilde{V} \\\gD \widetilde{T} \gD\widetilde{V} &\gD\widetilde{V}^2
		        \end{pmatrix} \1_{\DY =\pm1} \ \biggl|\ \widetilde{T}, \widetilde{V}, Y\right)=\frac{1}{N}\left(\begin{pmatrix}
			                                                                                                        3&0\\0&2\end{pmatrix}\begin{pmatrix}
			                                                                                                                             \gs^2_{\widetilde{T}}&0\\0&\gs^2_{\widetilde{V}}\end{pmatrix}+\Gamma_{2, \pm}\right),
	\end{align*}
	where $\norm{\Gamma_{2, \pm}}_{\hs}\lesssim n^{5/2}$.
	Define $\gl=\frac{1}{N}$, $\Psi=\begin{pmatrix}
			3&0\\0&2\end{pmatrix}$, and $\mvb_{\pm}=\begin{pmatrix}
			0\\\mp 2(n-2)p q\end{pmatrix}$.
	Notice that the random vector $(\mvX, Y)=\left(\begin{pmatrix}
				\widetilde{T} \\\widetilde{V} \end{pmatrix}, Y\right)$
	satisfies Assumptions~\ref{ass:uncorr and exch}, ~\ref{ass:Y}, ~\ref{ass:M1g}, and~\ref{ass:M2}. Using the change of variable~\eqref{eq:multichange} define the random vector $\mvW$ as in the statement of the Lemma~\ref{lem:twe}.

	By Proposition~\ref{prop: multichange} we get that the random vector $(\mvW, Y)$ satisfies Assumptions~\ref{ass:uncorr and exch}, ~\ref{ass:Y}, ~\ref{ass:M1}, and~\ref{ass:M2}; and for $k\in\gz+\dZ$ with $\abs{k}\ll n$, we can bound the error terms $\widehat{A}_k$ and $\widehat{B}_k$ using the following inequalities for $\hat{k}\in\{k-1, k\}$
	\begin{align*}
		\E\left(\abs{\widetilde{\mvR}_{1, \pm}}\mid Y=\hat{k}\right)\lesssim n^{-3/2+\eps}\textrm{ and }\E\left(\norm{\widetilde{\Gamma}_{2, \pm}}_{\hs}\mid Y=\hat{k}\right)\lesssim n^{-1/2+\eps},
	\end{align*}
	similarly the terms $\widehat{C}_k$, $\widehat{D}_k$, $\widehat{E}_k$, and $\widehat{F}_k$ can be bounded by using the following inequalities for $\hat{k}\in\{k-1, k\}$
	\begin{align*}
		                & \E(\abs{\mvW}\left(\abs{R_{0, +}}+\abs{R_{0, -}}\right)\mid {Y=\hat{k}})\lesssim n^{-\half+\eps}, \\
		                & \E\left(\abs{\DW}\mid {Y=\hat{k}}\right)\lesssim n^{-\half+\eps},                                 \\
		\text{and}\quad & \sqrt{\frac{1}{\gl }\E\left(\abs{\DW}^4\mid Y=\hat{k}\right)}\lesssim n^{-\half+\eps}.
	\end{align*}
	Thus by Theorem~\ref{th: multi case} with $\E\abs{\mvW}^4<\infty$ for $k\in\gz+\dZ$ with $\abs{k}\ll n$ we have that
	\begin{align*}
		\dwas\left((\mvW\mid Y=k), \mvZ\right)
		\lesssim {n}^{-\half+\eps}
	\end{align*}
	for some $\eps>0$.
\end{proof}

\subsection{General subgraph count given the number of edges}\label{ssec: general subgraph}

Stein's method for normal approximation and Stein--Chen's method for Poisson convergence were contextualized to random graphs by Barbour in~\cite{Barbour82}. The techniques for proving Poisson convergence described in~\cite{Barbour82} were used shortly after in many results, such as~\cite{Rucinski87} where many distributional convergence theorems were established in different regimes of Erd\H{o}s--R\'enyi random graph. Due to technical difficulties, this method for normal approximation in the context of random graphs was limited to particular examples until the following general theorem was shown in~\cite{BarbourKaronski89}. The original result in~\cite{BarbourKaronski89} is stated in terms of the Fortet--Mourier distance. However, the proof can be adopted to derive analogous statement in Wasserstein distance using standard machinery.
\begin{thm}[{\cite[Theorem 2]{BarbourKaronski89}}]
	Let $\cH$ be a graph on $v$ vertices with $m$ edges, $H$ be the number of its copies in $G_{n, p}$, then
	\begin{align*}
		\dwas\left((H-\E H)/\gs_H, Z\right)=k_{\cH}\cdot\begin{cases}
			                   r(n, p)^{-\half}     & p\ge \half \\
			              n^{-1}(1-p)^{-\half} & p< \half,
		                    \end{cases}
	\end{align*}
	where $k_{\cH}$ is a constant depending on $\cH$,
	\[
		r(n, p):=\min_{\cH'\subseteq \cH, e(\cH')>0}\left\{n^{v(\cH')}p^{e(\cH')}\right\},
	\]
	 and $v(\cH')$, $e(\cH')$ denote the number of vertices and edges in a graph $\cH'$, respectively.
\end{thm}
\begin{cor}[{\cite[Remark 1 on p.133]{BarbourKaronski89}}]\label{cor:barbour} For fixed $p\in(0, 1)$ we have
	\begin{equation*}
		\dwas\left((H-\E H)/\gs_H, Z\right)\lesssim n^{-1}.
	\end{equation*}
\end{cor}

It is important to notice that this result is based on the dependency graph approach. In this section, we will derive an analogous CCLT as a consequence of Theorem~\ref{th: multi case} and Lemma~\ref{lem:twe}.

From the applications in Section~\ref{ssec:triangle} and the decomposition in~\eqref{eq:triangle count} and~\eqref{eq:wedge count}, it is natural to expect that in order to derive CCLT for a graph $\cH$ given the number of edges, one would have to do it for a vector containing all possible sub-graphs of $\cH$. Surprisingly it is not the case, as illustrated by the following lemma.

\begin{lem}\label{lem:gengraph order}
	Let $\cH$ be a graph on $v$ vertices with $m$ edges, $H$ be the number of its copies in $G_{n, p}$ and $E$ be the number of edges, then \begin{align*}
		\var{(H)}\approx n^{2v-2}
		\quad\text{ and }\quad
		\var{\left(H-\gs_{H, E}/{\gs_E^2}\cdot E\right)}\approx n^{2v-3}.
	\end{align*}
\end{lem}
\begin{proof}
	Let $\mvs$ denote a subset of edges in the complete graph $\cK_n$ on $n$ vertices that forms an isomorphic copy of $\cH$ and $\cS=\cS_n(H)$ denote the collection of all such $\mvs$'s. Define $|\cS|$ to be the size of the collection $\cS$. This allows us to rewrite the sub-graph count as
	\begin{align*}
		H
		=\sum_{\mvs\in\cS}\prod_{e\in\mvs}\go_e=\sum_{\mvs\in\cS}\go_{\mvs},
	\end{align*}
	where $\go_{\mvs}:=\prod_{e\in\mvs}\go_e$. It is easy to check that
	\begin{align*}
		\gs_{H, E}:=\cov(H, E)=mp^m q\cdot |\cS|.
	\end{align*}
	In particular we can write
	\begin{align}
		\widehat{H} & :=H-\frac{\gs_{H, E}}{\gs_E^2}E
		=\sum_{\mvs\in\cS}\go_{\mvs}-\frac{m}{N}p^{m-1}|\cS|\cdot E.\label{eq:uncor H}
	\end{align}
	We rewrite $\go_{\mvs}$ as
	\begin{align*}
		\go_{\mvs}=\prod_{e\in\mvs}\go_e=\prod_{e\in\cS}(\bgo_e+p)=\sum_{\ell=0}^mp^{m-\ell}\sum_{\mvs'\subseteq\mvs, |\mvs'|=\ell}\widetilde{\go}_{\mvs'}
	\end{align*}
	where $\widetilde{\go}_{\mvs'}:=\prod_{e\in\mvs'}\bgo_{e}$.
	Given an edge-set $\mvs'$, we define \begin{align*}
		\cS_{\mvs'}:=\{\mvs\in \cS\mid \mvs'\subseteq \mvs\}.
	\end{align*}
	Given an edge $e$, size of the set $\{\mvs\in\cS\mid e\in \mvs\}$ is independent of $e$ and thus by symmetry the common size is given by $\frac{m}{N}|\cS|$. In particular, $\sum_{\mvs\in\cS}\sum_{e\in\mvs}\bgo_e=\frac{m}{N}|\cS|\cdot \overline{E}$ and from~\eqref{eq:uncor H} we get
	\begin{align}\label{eq:H decomp}
		\widehat{H}-\E \widehat{H} & =\sum_{\ell=2}^m p^{m-\ell}\sum_{\mvs':|\mvs'|=\ell}\widetilde{\go}_{\mvs'}\cdot |\cS_{\mvs'}|.
	\end{align}
	It is easy to see that $\widetilde{\go}_{\mvs}$ and $\widetilde{\go}_{\mvs'}$ are uncorrelated whenever $\mvs\neq \mvs'$.
	Let $v'$ be the number of vertices of a graph whose edges form $\mvs'$. Recall that $v=|V(\cH)|$. Thus it remains to estimate the following quantity for $\ell \ge2$,
	\begin{align*}
		\var\left(\sum_{\mvs':|\mvs'|=\ell}\widetilde{\go}_{\mvs'}\cdot |\cS_{\mvs'}|\right)
		 & =\sum_{|\mvs'|=\ell}|\cS_{\mvs'}|^2 \cdot (p q)^{\ell} \\
		 & \approx n^{2(v-v')}n^{v'} \approx n^{2v-3} ,
	\end{align*}
	where in the second equality we used the fact that $|\cS_{\mvs'}|\approx n^{v-v'}$ for $\mvs'\subseteq \mvs$, $|\{\mvs': |\mvs'|=\ell\}|\approx n^{v'}$; and in the last equality we used the fact that $\ell \ge 2$ implies $v' \ge 3$.
\end{proof}

\begin{defn}[Extension function]
	Let $\cK_n$ be the complete graph on $n$ vertices, and $\cH'$ be a sub-graph of $\cK_n$. For any fixed graph $\cH$, we define $\ext_n(\cH', \cH)$ to be the number of ways to extend the given graph $\cH'$ to an isomorphic copy of $\cH$ inside of $\cK_n$.
\end{defn}

For example if $\cT$ is the triangle and $\cK_4$ is the complete graph on four vertices then $\ext_n(\cT, \cK_4)=n-3$, while if $\cV$ is the wedge graph and $\cP_4$ is the path with four vertices then $\ext_n(\cV, \cP_4)=2(n-3)$. As a consequence of Lemma~\ref{lem:gengraph order} and a bit more analysis, we have the following lemma.

Let $\widetilde{T}$ and $\widetilde{U}$ are centered edge counts for triangles and wedges as in~\eqref{eq:triangle count} and~\eqref{eq:wedge count} respectively.
\begin{lem}\label{lem:RH}
	Let $\cH$ be a graph on $v$ vertices with $m$ edges, let $\cT$ be the triangle and $\cU$ be the wedge graphs. One can represent the random variable $\widehat{H}$, defined in~\eqref{eq:uncor H}, in the following way
	\begin{align}\label{eq:gengraph}
		\widehat{H} - \E\widehat{H}
		=\ext_n(\cT, \cH)\cdot \widetilde{T}+\ext_n(\cU, \cH)\cdot \widetilde{U}+R_H,
	\end{align}
	Moreover,
	$
		\var\bigl(\widetilde{T}\bigr)
		\approx \var\bigl(\widetilde{U}\bigr)
		\approx n^{3}
	$,
	$\ext_n(\cT, \cH)
		\approx \ext_n(\cU, \cH)
		\approx n^{2v-6}
	$
	and
	\begin{align*}
		\E R_H=0, \quad \E \abs{R_H}^{2\theta}= O\bigl(\gs_{\widehat{H}}^{2\theta} \cdot n^{-\theta}\bigr)\quad \text{ for }\quad\theta \ge1.
	\end{align*}
\end{lem}

\begin{proof}
	Everything besides the statement about $\E|R_H|^\theta$ follows directly from the representation of $\widehat{H}-\E \widehat{H}$ in~\eqref{eq:H decomp}, the fact that the only connected graphs on three vertices are the triangle and the wedge graphs, and the computations of $|\cS_{\mvs'}|$. Notice that if $\mvs'$ induces the graph $\cH'$ then $|\cS_{\mvs'}|=\ext_n(\cH', \cH)$.

	Thus it remains to compute $\E|R_H|^{2\theta}$ for an integer $\theta\ge 1$. Let $\mvs'$ denote a subset of edges in the complete graph $\cK_n$ on $n$ vertices that forms an isomorphic copy of $\cH'$, such that $\cH'\subseteq\cH$ and $|V(\mvs')|:= |V(\cH')| \ge4$. Then
	\begin{align*}
		\E \abs{R_H}^{2\theta} & =\sum_{\mvs_1, \mvs_2, \ldots, \mvs_{2\theta}}\prod_{i=1}^{2\theta} \widetilde{\go}_{\mvs_i}|\cS_{\mvs_i}|
		\approx n^{2\theta v-\sum_{i=1}^{2\theta}V(\mvs_i)+|\cup_{i=1}^{2\theta}V(\mvs_i)|}.
	\end{align*}
	Notice that $\E\go_{\mvs_1}\go_{\mvs_1}\cdots\go_{\mvs_{2\theta}}\neq0$ if there is and edge $e\in\cup_{i=1}^{2\theta}\mvs_i$ that appears in this product only once. So in contributing terms each vertex appears at least twice and thus $|\cup_{i=1}^{2\theta}V(\mvs_i)| \le\frac12\sum_{i=1}^{2\theta}V(\mvs_i).$ Now since $V(\mvs_i) \ge4$ and $\gs_{\widehat{H}}\approx n^{2v-3}$ we conclude that
	\begin{align*}
		\E \abs{R_H}^{2\theta}\lesssim n^{2\theta v-4\theta}\approx \gs_{\widehat{H}}^{2\theta}n^{-\theta}.
	\end{align*}
	This completes the proof.
\end{proof}

The decomposition from Lemma~\ref{lem:RH} yields that the terms involving $\widetilde{U}$ and $\widetilde{T}$ are both of order $n^{2v-3}$. After dividing both sides of~\eqref{eq:gengraph} by $\gs_{\widehat{H}}$ by multivariate CLT~\cite[Theorem 2.1]{BarbourRollin19} right hand side is asymptotically Gaussian with the same rate of convergence as the left hand side. However to derive CCLT simply stating that $R_H$ is of the smaller order is not enough as it could be significant after multiplication by $p_k^{-1}=\pr(Y=k)^{-1}\approx n$. Thus for an appropriate function $h$ we need to bound the following expression
\begin{align}
	\E\left(\abs{h\bigl(\widehat{H}-R_H\bigr)-h\bigl(\widehat{H}\bigr)}\big| Y=k\right)
	 & \le \abs{h}_1\cdot p_k^{-1} \E(|R_H|\cdot \1_{Y=k})\notag                                           \\
	 & \lesssim p_k^{-1}\cdot (\E|R_H|^{2\theta})^{\frac{1}{2\theta}}\cdot p_k^{1-\frac{1}{2\theta}}\notag \\
	 & \lesssim \gs_{\widehat{H}}\cdot n^{-(\theta-1)/(2\theta)},
\end{align}
for all $\theta>1$. Hence by taking $\theta$ to infinity and scaling by $\gs_{\widehat{H}}$ the right hand side goes to $0$ at the rate of $n^{-\half+\eps}$. Let
\begin{align*}
	\rho_1:=\lim_{n\to\infty}\frac{\ext_n(\cT, \cH)\cdot\gs_{\widetilde{T}}}{\gs_{\widehat{H}}}\quad \text{ and }\quad\rho_2:=\lim_{n\to\infty}\frac{\ext_n(\cV, \cH)\cdot\gs_{\widetilde{U}}}{\gs_{\widehat{H}}}.
\end{align*}

Recall that in Lemma~\ref{lem:twe} we applied a change of variables to $(\widetilde{T}, \widetilde{U})^T$ as in \eqref{eq:change twe}. After appropriate scaling we defined a random vector $(\widehat{T}, \widehat{U})^T$ in \eqref{eq:tw} and derived CCLT given the number of edges for this random vector.
Lemma~\ref{lem:twe} together with Lemma \ref{lem:RH} gives the following general result.

\begin{thm}\label{th: gengraph}
	Let $\cH$ be a finite graph, $H$ be the number of times $\cH$ appears in $G_{n, p}$, $\widehat{H} :=H-\frac{\gs_{H, E}}{\gs_E^2}E$, $\widetilde{T}$ and $\widetilde{U}$, and $Y$ be the centered edge counts of triangles, wedges, and edges, respectively. Let $(\widehat{T}, \widehat{U})$ be the scaled random vector defined in \eqref{eq:tw}.
	For any $k\in \{-Np\}+\dZ$ with $|k|\ll n$ we have
	\begin{align*}
		\dwas\biggl(\left(\left((\widehat{H}-\E\widehat{H})/\gs_{\widehat{H}}, {\widehat{T}}, {\widehat{U}}\right)\, \big|\, Y=k\right), \left(\rho_1Z_1+\rho_2Z_2, Z_1, Z_2\right)\biggr)
		\lesssim n^{-\half+\eps},
	\end{align*}
	where $Z_1$ and $Z_2$ are independent standard normal random variables and $\eps>0$.
\end{thm}

There are several reasons why the upper bound of the rate of convergence in this theorem is slower than that in Corollary~\ref{cor:barbour}. First, it might be due to the limitation of exchangeable pairs in the context of random graphs. As we saw before in Lemma~\ref{lem:wedge class}, the upper bound on the rate of convergence can be slower than $n^{-1}$ even with classical techniques. Another possibility is that it is caused by our technique. At the end of the proof of Theorem~\ref{th: imp symcase}, one can see when we invoke the Lipschitz property; our upper bound cannot be better than $|\gD W|/\gs_W$, which in this case is of order $n^{-\half}$. This naturally leads to Question \ref{q: locdep} in Section~\ref{ssec:closing}.


\section{Proofs of Main Results}\label{sec:proofs}
As we mentioned above, one of the proofs of Theorem~\ref{th: imp symcase} and Theorem~\ref{th: multi case} have two main components. One is to derive a CCLT for $(W\mid Y\in\{k-1, k\})$ which we do in Lemma~\ref{lem:consecutive k} and Lemma~\ref{lem: multi conseq k} for the univariate case and multivariate case respectively. The other step is to quantitatively bound the difference between $(W\mid Y=k)$ and $(W\mid Y=k-1)$, which we do in the next lemma.

\begin{lem}\label{lem:stein dif}
	For any $d\ge 1$. Suppose $\mvW$ is a $d$-dimensional random vector such that $(\mvW, Y)$ satisfy Assumptions~\ref{ass:uncorr and exch} and~\ref{ass:Y}, then for any $1$-Lipschitz function $h:\dR^d\to\dR$ and for any $k$ such that $\pr(Y=k)>0$ and $\pr(Y=k-1)>0$ we have that
	\begin{align*}
		 & \abs{\E (h(\mvW)-h(\gS^{\half}\mvZ))\left(\ind_{Y=k}-\ind_{Y=k-1}\right) }                                                                                                         \\
		 & \qquad\leq \frac{1}{Q}\E\left(|\gD \mvW|\cdot\1_{Y\in\{k-1, k\}}\right) +\frac{1}{Q}\E\left((\abs{\mvW}+\sqrt{\tr(\gS)})(\abs{R_{0, +}}+\abs{R_{0, -}})\1_{Y\in\{k-1, k\}}\right),
	\end{align*}
	where $\mvZ$ is a $d$-dimensional standard normal random vector that is independent of everything else.
\end{lem}

\begin{proof}
	Without loss of generality we may assume that $\E h(\gS^{\half}\mvZ)=0$.
	Assumption~\ref{ass:Y} implies that
	\[
		Q=\E (\1_{\DY =-1}\mid \mvW, Y)-R_{0, -}.
	\]
	Hence,
	\begin{align*}
		\E\left(h(\mvW)\1_{Y=k}\right) & =\frac{1}{Q}\E\left(h(\mvW)\1_{Y=k}\1_{\DY =-1}\right)-\frac{1}{Q}\E\left(h(\mvW)\1_{Y=k}R_{0, -}\right) \\
		                               & =\frac{1}{Q}\E\left(h(\mvW)\1_{Y=k}\1_{Y'=k-1}\right)-\frac{1}{Q}\E\left(h(\mvW)R_{0, -}\1_{Y=k}\right).
	\end{align*}
	Similarly using the fact that $Q=\E (\1_{\DY =1}\mid \mvW, Y)-R_{0, +}$ we get that
	\begin{align*}
		\E\left(h\left(\mvW\right)\1_{Y=k-1}\right) & =\frac{1}{Q}\E\left(h\left(\mvW\right)\1_{Y'=k}\1_{Y=k-1}\right)-\frac{1}{Q}\E\left(h\left(\mvW\right)R_{0, +}\1_{Y=k-1}\right).
	\end{align*}
	Moreover, exchangeability of $(\mvW, Y)$ and $(\mvW'.Y')$ implies
	\[
		\E\left(h\left(\mvW\right)\1_{Y=k-1}\1_{Y'=k}\right)=\E\left(h\left(\mvW'\right)\1_{Y'=k-1}\1_{Y=k}\right).
	\]
	In particular, since $h(\mvw)$ is $1$-Lipschitz and $\E h(\gS^{\half}\mvZ)=0$ it follow that
	\begin{align*}
		 & \abs{\E\left(h\left(\mvW\right) \1_{Y=k}\right)-\E\left(h\left(\mvW\right)\1_{Y=k-1}\right)}                                  \\
		 & \quad\le \frac{1}{Q}\E\left(\abs{\DW}\cdot\1_{Y\in\{k-1, k\}}\right)                                                          \\
		 & \qquad\quad+\frac{1}{Q}\E\left((\abs{\mvW}+\E\abs{\gS^{\half}\mvZ})(\abs{R_{0, +}}+\abs{R_{0, -}})\1_{Y\in\{k-1, k\}}\right),
	\end{align*}
	where $Z\sim N(0, 1)$ and is independent of everything else. Recalling the fact that $\E\abs{\gS^{\half}\mvZ}\le\sqrt{\E(|\gS^{\half}\mvZ|^{2})}=\sqrt{\tr(\gS)}$ completes the proof.
\end{proof}

\subsection{Proof of Theorem~\ref{th: symcase}}\label{ssec:symcase}
Fix $k\in \zeta+\dZ$ with $p_k:=\pr(Y=k)>0$. We define
\begin{align}\label{eq:g}
	g(y):=\left(-1\right)^{y} \cdot \ind_{y\le 0} \text{ for } y\in \dZ.
\end{align}
Note that
\[
	g(y)+g(y+1)\equiv \1_{y=0}\text{ for all } y.
\]
Recall from equation~\eqref{eq:theta}, the following mean zero random variable
\begin{align*}
	\Theta_f(W, Y) & := \E\left( (F(W')-F(W)) \cdot (g(Y'-k)\cdot \1_{\DY =1} +g(Y-k)\cdot \1_{\DY =-1})\mid W, Y\right),
\end{align*}
where $F(w)$ is a three times differentiable function. Denote, the first derivative of $F$ by $f$. We will use $$\abs{f}_i:=\norm{f^{(i)}}_\infty, $$ where $f^{(i)}$ is the $i^{\textrm{th}}$ derivative of $f$. Using Taylor approximation upto the third order we have
\begin{align*}
	\Theta_f(W, Y) & =
	f(W)\cdot \hat{g}_1 +
	\frac12f'(W)\cdot \hat{g}_2 + \frac{1}{6} \abs{f}_2\cdot \text{Err}
\end{align*}
where
\begin{align*}
	 & \hat{g}_i :=\E\left((\gD W)^i \left(g(Y-k+1)\1_{\{\DY =1\}}+g(Y-k)\1_{\{\DY =-1\}}\right) \mid W, Y\right),
\end{align*}
for $i=1, 2$ and
\begin{align*}
	\abs{\text{Err}} \le \E\left(|\gD W|^3\cdot (\abs{g(Y-k+1)}\1_{\DY =1}+\abs{g(Y-k)}\1_{\DY =-1}) \mid W, Y\right).
\end{align*}
Now we plug in~\eqref{eq:g} for $g(\cdot)$ and notice that
\begin{align*}
	 & \hat{g}_i=M_{i, -}\ind_{Y= k}+\bigl( M_{i, +} - M_{i, -}\bigr) g(Y-k+1), \text{ for } i=1, 2.
\end{align*}
By grouping the terms based on the indicator and the order of the $M$ terms to get the following
\begin{align*}
	\Theta_f(W, Y)
	 & =\bigl(f(W)M_{1, -}+\frac{1}{2}f'(W)M_{2, -}\bigr)\ind_{Y= k}                                  \\
	 & \quad+\sum_{i=1}^2 \frac{1}{i!}f^{(i-1)}(W)\cdot \bigl( M_{i, +} -M_{i, -}\bigr)\cdot g(Y-k+1)
	+ \frac{1}{6} \abs{f}_2\cdot \text{Err}.
\end{align*}
By using Assumptions~\ref{ass:M1}, ~\ref{ass:M2}, the fact that $\E\Theta_f(W, Y)=0$ and $\norm{g}_\infty\le 1$, we can write
\begin{align*}
	 & \frac12\gl \abs{\E\bigr(f'(W)(\psi+R_{2, -}) - f(W)(\psi W+2R_{1, -})\bigl)\ind_{Y= k}} \\
	 & \qquad \le \sum_{i=1}^2 \frac{1}{i!} \abs{f}_{i-1}\cdot \E\abs{M_{i, +} -M_{i, -}}
	+ \frac{1}{6} \abs{f}_2\cdot \E\abs{\text{Err}}.
\end{align*}
By moving the error terms to the right hand side and dividing both sides of the inequality by $\gl \psi /2$ we get
\begin{align*}
	 & \abs{\E(f'(W)- f(W)W )\1_{Y=k}}                                                                                     \\
	 & \quad \le \frac{2\abs{f}_0}{\psi}\cdot \E|R_{1, -}|\ind_{Y= k}+ \frac{\abs{f}_1}{\psi}\cdot \E|R_{2, -}|\ind_{Y= k}
	+\frac{\abs{f}_2}{3\gl \psi} \cdot \E\abs{\text{Err}}                                                                  \\
	 & \qquad+\frac{2\abs{f}_0}{\psi}\cdot \E|R_{1, +}-R_{1, -}|+\frac{\abs{f}_1}{\psi}\cdot \E|R_{2, +}-R_{2, -}|.
\end{align*}
In the last line, each term would have appeared with $\1_{Y\le k-1}$; however, since $\{Y\le k-1\}$ is a constant order event, we upper-bounded it by $1$.
Recall that using the relation~\eqref{eq:stein de} bounding $\abs{\E(f'(W)-Wf(W))\1_{Y=k}}$ for $f\in\cA$ allows us to bound the Wasserstein distance between $W\1_{Y=k}$ and the standard normal random variable $Z$.
Dividing both sides by $$p_k=\pr(Y=k)>0$$ we get the desired bound
\begin{align*}
	\abs{\E\left(h\left(W)\mid Y=k\right)-h\left(Z\right)\right)}\le \frac{2}{\psi}\left(A_k+\frac{1}{p_k} C\right)+ \sqrt{\frac{2}{\pi\psi^2}}\left(B_k+\frac{1}{p_k} D\right)+\frac{2}{3 \gl\psi p_k}E.
\end{align*}
This completes the proof.\qed


\subsection{Proof of Lemma~\ref{lem:consecutive k}}\label{ssec:cons}

By the same argument as in the Proof of Theorem~\ref{th: symcase} but with $g(y)=\1_{y=0}$ we get that
\begin{align*}
	\Theta_f(W, Y)
	 & =\bigl(f(W)M_{1, -}+\frac{1}{2}f'(W)M_{2, -}\bigr)\ind_{Y= k} +\bigl(f(W)M_{1, +}+\frac{1}{2}f'(W)M_{2, +}\bigr)\ind_{Y= k-1} \\
	 & \quad+\frac{1}{6} f''(W)\cdot \text{Err}\cdot\ind_{Y\in\{ k-1, k\}},
\end{align*}
where
\begin{align*}
	\abs{\text{Err}} \le \E\left(|\gD W|^3 \mid W, Y\right).
\end{align*}
Using Assumptions~\ref{ass:M1}, ~\ref{ass:M2}, and the fact that $\E\Theta_f(W, Y)=0$ and $\norm{g}_\infty\le 1$, we derive that
\begin{align*}
	 & \abs{\E(f'(W)- f(W)W )\ind_{Y\in\{ k-1, k\}}}                                                                                                                                                 \\
	 & \le \frac{2\abs{f}_0}{\psi}\cdot \E|R_{1, -}|\ind_{Y= k}+ \frac{\abs{f}_1}{\psi}\cdot \E|R_{2, -}|\ind_{Y= k}                                                                                 \\
	 & \qquad+ \frac{2\abs{f}_0}{\psi}\cdot \E|R_{1, +}|\ind_{Y= k-1}+ \frac{\abs{f}_1}{\psi}\cdot \E|R_{2, +}|\ind_{Y= k-1}+\frac{\abs{f}_2}{3\gl \psi} \E\abs{\textrm{Err}}\ind_{Y\in\{ k-1, k\}}.
\end{align*}
Recall that as defined in Assumption~\ref{ass:Y} we have that for $k$ such that $p_k>0$ and $p_{k-1}>0$ the ratio $$r_k:=\frac{p_{k-1}}{p_{k}}\in(0, \infty), $$
which implies that $$\frac{p_k}{p_k+p_{k-1}}= \frac{1}{1+r_k}\quad\text{and}\quad
	\frac{p_{k-1}}{p_k+p_{k-1}}= \frac{r_k}{1+r_k}. $$
The standard application of the relation~\eqref{eq:stein de} as well as dividing by $p_k+p_{k-1}$ now yields the desired bound
\begin{align*}
	 & \dwas((W \mid Y\in\{ k-1, k\}), Z)                                                                  \\
	 & \quad \le \frac{2}{\psi(1+r_k)}\left( \E(|R_{1, -}|\mid {Y=k})+r_k\E(|R_{1, +}|\mid{Y= k-1})\right) \\	 & \qquad+ \frac{\sqrt{2/\pi}}{\psi(1+r_k)}\left(\E(|R_{2, -}|\mid {Y= k}) + r_k\E(|R_{2, +}|\mid{Y= k-1}) \right)\\
	 & \qquad+\frac{2}{3\gl \psi} \E(\abs{\gD W}^3\mid Y\in\{ k-1, k\}).
\end{align*}
The proof is now complete.
\qed

\subsection{Proof of Theorem~\ref{th: multi case}}\label{ssec:multistein proof}
Since Lemma~\ref{lem:stein dif} is applicable in multivariate case, hence to mimic the univariate argument it is enough to derive an equivalent of Lemma~\ref{lem:consecutive k} in multivariate case. We first derive a bound
for the operator
\begin{align*}
	\sS f(\mvx) := \la \gS, \hess f(\mvx)\ra_{\hs} - \la \mvx, \nabla f(\mvx)\ra,
\end{align*}
where $f$ is three times differentiable function. This is equivalent to deriving a CCLT with explicit rate of convergence in the metric with smooth test functions. From that we derive the bound in terms of Wasserstein distance by following the argument of~\cite[Theorem 1.1]{FangKoike22} that builds up on the previously established techniques for exchangeable pair to the multivariate setting such as~\cite{Chatterjee07, Raic19, ReinertRollin09}.

We now introduce the notations needed for the arguments in this section. For further details on them we refer to \cite[Section 5]{Raic19} and~\cite[Section 1]{FangKoike22}.
For $r$-times differentiable function $f:\dR^d\to\dR$, we denote by $\nabla^rf(x)$ the $r^{\textrm{th}}$ derivative of $f$ at $x\in\dR^d$.

The value of $\nabla^rf(x)$ evaluated at $u_1, ..., u_r \in \dR^d$ is defined to be
\begin{equation*}
	\langle\nabla^rf(x), u_1\otimes u_2\cdots\otimes u_r\rangle:=\sum_{j_1, j_2, \ldots, j_r=1}^d\partial_{j_1, j_2, \ldots, j_r}f(x)u_{1, j_1}\cdots u_{r, j_r}.
\end{equation*}
We define the injective norm of an $r$-linear form $T$ on
$\dR^d$ by
\begin{equation*}
	\abs{T}_\vee:=\sup_{\abs{u_1}\vee\cdots\vee\abs{u_r}\le1}\abs{\langle T, u_1\otimes u_2\cdots\otimes u_r\rangle}.
\end{equation*}
Then for an $(r-1)$-times differentiable function $h:\dR^d\to\dR$, we define the operator
\begin{equation}\label{eq:Mr}
	M_r(h):=\sup_{x\neq y}\frac{\abs{\nabla^{r-1}h(x)-\nabla^{r-1}h(y)}_{\vee}}{\abs{x-y}}.
\end{equation}
Notice that in the denominator of \eqref{eq:Mr} $\abs{\cdot}$ is the Euclidean $\ell_{2}$ norm. Furthermore, if $h$ is $r$-times differentiable then $M_r(h) = \sup_{x\in\dR^d} \abs{\nabla^r h(x)}_\vee$.

\begin{lem}\label{lem: multi conseq k}
	Under assumptions of Theorem~\ref{th: multi case}, for any third order differentiable function $f$ such that the following terms make sense and for any $k$ such that $\pr(Y=k)>0$ and $\pr(Y=k-1)>0$ and we have that
	\begin{align*}
		 & \E\biggl(\sS f(\mvW)\cdot \1_{Y\in\{k-1, k\}}\biggr)                                                                                           \\
		 & \quad\leq M_1(f)\E\left(\abs{\Psi^{-1}\mvR_{1, +}}\1_{Y=k-1}+\abs{\Psi^{-1}\mvR_{1, -}}\1_{Y=k}\right)                                         \\
		 & \qquad+\sup_{\mvw}\norm{\hess f(\mvw)}_{\hs}\E\left(\norm{\Psi^{-1}\gC_{2, +}}_{\hs}\1_{Y=k-1}+\norm{\Psi^{-1}\gC_{2, -}}_{\hs}\1_{Y=k}\right) \\
		 & \qquad+\frac12M_3(f)\E\left(\abs{(\gl\Psi)^{-1}\DW}\abs{\DW}^2\1_{Y\in\{k-1, k\}}\right).
	\end{align*}
	Moreover if $\E\abs{\mvW}^4<\infty$ then the last error term on the right hand side can be replaced by
	$$\frac{M_4(f)}{4}\E\left(\abs{(\gl\Psi)^{-1}\DW}\abs{\DW}^3\1_{Y\in\{k-1, k\}}\right).$$
\end{lem}
\begin{proof}
	Define the symmetric function
	\begin{align*}
		G_k(Y, Y') & :=\1_{Y'=k, Y=k-1}+\1_{Y=k, Y'=k-1}   \\
		           & =\1_{Y'=k, \DY =1}+\1_{Y=k, \DY =-1}.
	\end{align*}
	We consider the following mean zero random variable
	\begin{align*}
		\Theta_f(\mvW, Y)=\frac1{\gl}\DW^T\Psi^{-T}\left(\nabla f(\mvW)+\nabla f(\mvW')\right)G_k(Y, Y').
	\end{align*}
	Clearly, we have
	\begin{align}\label{eq:sum}
		\begin{split}
			(\nabla f(\mvW)+\nabla f(\mvW')) & -(2\nabla f(\mvW)+\hess f(\mvW)\DW) \\
			& =\nabla f(\mvW')-\nabla f(\mvW)-\hess f(\mvW)\DW) \\
			& = \int_0^1 (\hess f(\mvW+u \DW)-\hess f(\mvW))du\; \DW.
		\end{split}
	\end{align}
	Recall that under Assumption~\ref{ass:M2} the pair $(\mvW, \mvW')$ satisfies the following two equalities
	\begin{align*}
		\E\left(\DW\DW^T\1_{\DY =\pm1}\mid \mvW, Y\right)=\gl\left(\Psi\gS+\gC_{2, \pm}\right),
	\end{align*}
	for some random matrices $\gC_{2, \pm}=\gC_{2, \pm}(\mvW, Y)$. Since $\DW\DW^T, \gS$ are symmetric matrices, Assumption~\ref{ass:M2} is equivalent to
	\begin{align*}
		\E\left(\DW\DW^T\1_{\DY =\pm1}\mid \mvW, Y\right)=\gl\left(\gS\Psi^{T}+\gC^{T}_{2, \pm}\right).
	\end{align*}

	Plugging $(\nabla f(\mvW)+\nabla f(\mvW'))$ from equation~\eqref{eq:sum} into $\Theta_f$ and using Assumptions~\ref{ass:M1} and \ref{ass:M2} as well as Taylor expansion we get
	\begin{align}
		0 & =-\E\left( \mvW^T\nabla f(\mvW)\1_{Y\in\{k-1, k\}}+2(\mvR_{1, +}^T\1_{Y=k-1}+\mvR_{1, -}^T\1_{Y=k})\Psi^{-T}\nabla f(\mvW)\right)\label{eq:multilin} \\
		  & \qquad+\E\tr\left((\gS\Psi^T+\gC_{2, +}^{T})\Psi^{-T}\hess f(\mvW)\right)\1_{Y=k-1}\label{eq:multisq1}                                              \\
		  & \qquad+\E\tr\left((\gS\Psi^T+\gC_{2, -}^{T})\Psi^{-T}\hess f(\mvW)\right)\1_{Y=k}\label{eq:multisq2}                                                \\
		  & \qquad+\textrm{Err}, \notag
	\end{align}
	where
	\begin{align}
		\textrm{Err}
		 & =\frac{1}{\gl}\E\bigl(\DW^T\Psi^{-T}\bigl((\hess f(\mvW+U_1\gD W)-\hess f(\mvW))\DW\cdot G_k(Y, Y')\bigr)\label{eq:multiErr}
	\end{align}
	and $U_1\sim \textrm{Uniform}[0, 1]$ is independent of everything else.
	Depending on the moment assumptions one can bound $\abs{ \textrm{Err}}$ in two different ways. If $\E\abs{\mvW}^3<\infty$ then
	\begin{align}
		\abs{\textrm{Err}}
		 & \leq \frac1{\gl}\E\sum_{i, j}^d \abs{ (\Psi^{-1}\DW)_i\cdot \DW_j\cdot (\partial_{ij} f(\mvW+U_1\DW)-\partial_{ij} f(\mvW))}\cdot G_k(Y, Y')\notag \\
		 & \leq \frac12M_3(f)\E\left(\abs{(\gl\Psi)^{-1}\DW}\abs{\DW}^2\1_{Y\in\{k-1, k\}}\right).\label{eq:M3bound}
	\end{align}
	On the other hand if $\E\abs{\mvW}^4<\infty$ one can use the exchangeability of $(\mvW, Y)$ and $(\mvW', Y')$ to derive a better bound. Notice that $G_k(Y, Y')$ is a symmetric function and hence remains the same during such change.
	\begin{align*}
		 & \E(\Psi^{-1}\DW)_i\DW_j \DW_\ell U_1\, \partial_{ij\ell} f(\mvW+U_1U_2\DW)G_k(Y, Y')                   \\
		 & =-\E(\Psi^{-1}\DW)_i\DW_j \DW_\ell U_1\, \partial_{ij\ell} f\left(\mvW+(1-U_1U_2)\DW\right)G_k(Y, Y'),
	\end{align*}
	where $U_2\sim\textrm{Uniform}[0, 1]$ and independent of everything else.
	This allows us to rewrite the error term \eqref{eq:multiErr} as the average
	\begin{align}
		\textrm{Err} & =\E\sum_{i.j, \ell}^d ((\gl\Psi)^{-1}\DW)_i\DW_j\gD\mvW_\ell\cdot U_1\notag                          \\
		             & \qquad\cdot (\partial_{ij\ell} f(\mvW+U_1U_2\DW)-\partial_{ij\ell} f(\mvW+(1-U_1U_2)\DW))G_k(Y, Y').
	\end{align}
	Thus, using that $\E(U_1\abs{1-2U_1U_2})=\frac{1}{4}$, we derive that
	\begin{align}\label{eq:M4bound}
		\begin{split}
			\abs{\textrm{Err}}
			& \leq M_4(f)\E(U_1\abs{1-2U_1U_2})\cdot \E\left(\abs{(\gl\Psi)^{-1}\DW}\abs{\DW}^3\cdot \1_{Y\in\{k-1, k\}}\right)\\
			& =\frac14 M_4(f) \E\left(\abs{(\gl\Psi)^{-1}\DW}\abs{\DW}^3\cdot \1_{Y\in\{k-1, k\}}\right).
		\end{split}
	\end{align}
	Combining \eqref{eq:multilin}, \eqref{eq:multisq1}, \eqref{eq:multisq2}, with \eqref{eq:M3bound} or \eqref{eq:M4bound}, and moving
	$$
		\E\left(-\DW^T\nabla f(\mvW)+\tr(\gS\hess f(\mvW))\right)\1_{Y\in\{k-1, k\}}
	$$
	to the left-hand side of the equation yields, the desired bound.
\end{proof}

\begin{proof}[Proof of Theorem~\ref{th: multi case}]
	Let
	\[
		A_0= \sqrt{\tr(\gS)+\E \left(\abs{\mvW}^{2}\, \big|\, Y\in\{k-1, k\}\right)}.
	\]
	If
	\[
		\abs{ \E(\sS f(\mvW)\mid B) } \le A_1\cdot M_1(f) + A_2\cdot \sup_w \norm{\hess f(w)}_{\hs} +A_4\cdot \frac14 M_4(f),
	\]
	for some event $B$, then using the approximation scheme from Rai\v c~\cite{Raic19} and Fang--Koike~\cite{FangKoike22}, we get that
	\begin{align}\label{eq:lip4}
		\begin{split}
			&\abs{\E((h(\mvW)-h(\gS^{\half}\mvZ)\mid B)}\\
			&\qquad \le M_1(h)\left(A_1+\norm{\gS^{-1/2}}_{\op}\cdot A_2 +\norm{\gS^{-1/2}}_{\op}^{3/2}\cdot \sqrt{c_3 A_0A_4} \right)
		\end{split}
	\end{align}
	where $c_3=(2+8e^{-3/2})/\sqrt{2\pi}<2$. 	The proof now follows from Lemmas~\ref{lem:stein dif} and~\ref{lem: multi conseq k}.

	Similarly, for
	\[
		\abs{ \E(\sS f(\mvW)\mid B) } \le A_1\cdot M_1(f) + A_2\cdot \sup_w \norm{\hess f(w)}_{\hs} +A_3\cdot M_3(f),
	\]
	we get that
	\begin{align}\label{eq:lip3}
		\begin{split}
			&\abs{\E((h(\mvW)-h(\gS^{\half}\mvZ)\mid B)}
			\le M_1(h)\biggl(A_1+\norm{\gS^{-1/2}}_{\op}\cdot A_2 \\
			&\qquad\qquad\qquad +
			c_{2}\cdot \norm{\gS^{-1/2}}_{\op}^{2}\cdot A_{3} \cdot (1+ \abs{\log{(c_2 \norm{\gS^{-1/2}}_{\op}^2 \cdot A_3/A_0)}}
			\biggr)
		\end{split}
	\end{align}
	where $c_{2}=4/\sqrt{2\pi e}<1$
	and the proof follows similarly. We provide a proof of \eqref{eq:lip4} and~\eqref{eq:lip3} in Lemma~\ref{lem:s2l} for completeness.
\end{proof}


\section{Change of variables}\label{ssec:change of var}

In this section we provide proofs of Proposition~\ref{prop: change} and Proposition~\ref{prop: multichange}. We treat each of the error terms in a separate lemma.

\begin{lem}\label{lem:1order}
	Under assumptions of Proposition~\ref{prop: change}, recall that we defined the change of variable
	\begin{align*}
		W^0:=X+\gl\, \psi \, \ga XY+\frac{\gl \, \theta}{2}\left(Y^2-\E Y^2\right)+\frac{\gl^2(\psi +1)\ga\theta}{3}(Y^3-\E Y^3),
	\end{align*}
	where $\ga=\frac{a_+-a_-}{2Q}$ and $\theta=\frac{b_+}{Q}$ and $W=\frac{W^0}{\gs_{W^0}}$.
	Then
	$W$ satisfies Assumption~\ref{ass:M1} with the error terms given by
	\begin{align*}
		 & \widetilde{R}_{1, \pm}=\frac{\gl\theta}{2}\left(1-\frac{\psi }{2}\right)\frac{\overline{Y^2}}{\gs_{W^0}}+\frac{1}{\gs_{W^0}}(\widetilde{\eps}_{0, \pm}+\widetilde{\eps}_{1, \pm}+\widetilde{\eps}_{2, \pm}+\widetilde{\eps}_{3, \pm}),
	\end{align*}
	where
	\begin{align*}
		\widetilde{\eps}_{0, \pm}                & :=\gl\psi^2\ga a_\pm X(Y\pm 1)-\frac{\gl\, \ga\, \psi }{2} XY+R_{1, \pm}\biggl(1\pm\gl\, \psi \, \ga+ \gl\, \psi \, \ga Y\biggr), \\
		\widetilde{\eps}_{1, \pm}                & :=\gl\left(\pm\gl b_\pm\psi\ga\mp\frac{a_\pm\theta}{2}+(\psi+1)\theta\ga\left(Q-\frac13\gl a_\pm\right)\right)Y,                  \\
		\widetilde{\eps}_{2, \pm}                & :=\mp\gl^2\theta(\psi +1)\ga a_\pm Y^2,                                                                                           \\
		\text{and}\quad\widetilde{\eps}_{3, \pm} & :=\frac13\gl\theta(\psi +1)\ga Q-\gl^2\theta(\psi +1)\ga a_\pm Y^3-\frac{\gl^2\psi (\psi +1)\ga\theta}{6}\left(Y^3-\E Y^3\right).
	\end{align*}
\end{lem}

\begin{proof}
	With change from $Y$ to $Y'$ we have the following change in $W^0$
	\begin{align}\label{eq:deltaW}
		\gD W^0 & =\gD X+\gl \psi \ga\left(\gD XY+\DY X+\gD X\DY \right)+\gl \, \theta\left(Y\DY +\frac12 \DY ^2\right)\notag \\
		        & \quad+\gl^2(\psi +1)\ga\theta\DY \left(Y^2+Y\DY +\frac{(\DY )^2}{3}\right).
	\end{align}
	Using Assumptions~\ref{ass:Y} and~\ref{ass:M1g}, it can be shown that
	\begin{align*}
		M_{1, \pm}\left(W^0, Y\right) & =M_{1, \pm}(X, Y)\left(1\pm\gl\, \psi \, \ga+ \gl\, \psi \, \ga Y\right)                                                                      \\
		                              & \quad+\gl \, \biggl(\pm\ga\psi X+\theta\left(\pm Y+\frac12\right)+\gl(\psi +1)\ga\theta\left(\pm Y^2+ Y\pm\frac{1}{3}\right)\biggr)M_{0, \pm} \\
		                              & =-\gl \biggl[a_{\pm}\psi X\mp\ga\psi Q X+b_{\pm} Y+\gl\psi \ga b_\pm Y^2+\eps_{0, \pm} +\left(\mp \theta Q\right)Y+\eps_{1, \pm}              \\
		                              & \qquad+\left(\mp \gl(\psi +1)\ga\theta Q+\gl\theta a_\pm\right)Y^2+\eps_{2, \pm}-\frac{\theta}{2}Q+\eps_{3, \pm}\biggr],
	\end{align*}
	where
	\begin{align*}
		\eps_{0, \pm}                & :=\gl\psi^2\ga a_\pm X(Y\pm 1)+R_{1, \pm}\biggl(1\pm\gl\, \psi \, \ga+ \gl\, \psi \, \ga Y\biggr),               \\
		\eps_{1, \pm}                & :=\gl\left(\pm\gl b_\pm\psi\ga\mp\frac{a_\pm\theta}{2}+(\psi+1)\theta\ga\left(Q-\frac13\gl a_\pm\right)\right)Y, \\
		\eps_{2, \pm}                & :=\mp\gl^2\theta(\psi +1)\ga a_\pm Y^2,                                                                          \\
		\text{and}\quad\eps_{3, \pm} & :=\frac13\gl\theta(\psi +1)\ga Q-\gl^2\theta(\psi +1)\ga a_\pm Y^3.
	\end{align*}
	Notice that $a_\pm\mp\ga Q=a_\pm\mp\frac{a_+-a_-}{2}=\frac{a_++a_-}{2}=\frac12$, thus the coefficient of $X$ is equal to $\frac12$ and we can use change of variable $~\eqref{eq:change of var}$ once again to get
	\begin{align*}
		 & M_{1, \pm}\left(W^0, Y\right)                                                                                                                                                             \\
		 & =-\gl \biggl[\frac12 \psi W^0-\frac{\gl\, \ga\, \psi }{2} XY-\frac{\gl\, \psi \, \theta}{4}\left(Y^2-\E Y^2\right)-\frac{\gl^2\psi (\psi +1)\ga\theta}{6}\left(Y^3-\E Y^3\right)          \\
		 & \quad\quad\quad+\left(b_{\pm}\mp \theta Q+O(\gl\theta)\right)Y+\left(\mp \gl(\psi +1)\ga\theta Q+\gl\theta a_\pm+\gl\psi \ga b_\pm\right)Y^2 -\frac{\theta}{2}Q+\sum_{i=0}^3\eps_i\biggr] \\
		 & =-\gl \biggl[\frac12W^0+\left(b_{\pm}\mp \theta Q+O(\gl\theta)\right)Y+\left(\mp \gl(\psi +1)\ga\theta Q+\gl\theta a_\pm+\gl \psi \ga b_\pm-\frac{\gl\, \psi \, \theta}{4}\right)Y^2      \\
		 & \quad\quad\quad+\frac{\gl\, \psi \, \theta}{4}\E Y^2-\frac{\theta}{2}Q+\widetilde{\eps}_{0, \pm}+\widetilde{\eps}_{1, \pm}+\widetilde{\eps}_{2, \pm}+\widetilde{\eps}_{3, \pm}\biggr],
	\end{align*}
	where $\widetilde{\eps}_{1, \pm}=\eps_{1, \pm}$, $\widetilde{\eps}_{2, \pm}=\eps_{2, \pm}$,
	\begin{align*}
		\widetilde{\eps}_{0, \pm}:=\eps_{0, \pm}-\frac{\gl\, \ga\, \psi }{2} XY\quad\text{and}\quad \widetilde{\eps}_{3, \pm}:=\eps_{3, \pm}-\frac{\gl^2\psi (\psi +1)\ga\theta}{6}\left(Y^3-\E Y^3\right).
	\end{align*}
	The coefficient of $Y$ cancels by Assumption~\ref{ass:M1g} that says $b_\pm\mp\theta Q=b_\pm\mp b_+=0$.
	One can rewrite the coefficient of $Y^2$ in the following way
	\begin{align}\label{eq:impact of cubic term}
		\begin{split}
			\mp \gl (\psi +1)\ga\theta Q &+\gl\theta a_\pm+\gl \psi \ga b_\pm-\frac{\gl\, \psi \, \theta}{4} \\
			& =\mp \gl \psi \ga\theta Q\mp \gl \theta\ga Q+\gl\theta a_\pm+\gl \psi \ga b_\pm-\frac{\gl\, \psi \, \theta}{4} \\
			& =\frac{\gl\theta}{2}-\frac{\gl\, \psi \, \theta}{4}
			=\frac{\gl\theta}{2}\left(1-\frac{\psi }{2}\right).
		\end{split}
	\end{align}
	The constant term, using $Q=\gl\E Y^2$ from Assumption~\ref{ass:Y}, can be rewritten in the similar fashion
	\begin{align*}
		\frac{\gl\, \psi \, \theta}{4}\E Y^2-\frac{\theta}{2}Q=-\frac{\gl\theta}{2}\left(1-\frac{\psi }{2}\right)\E Y^2.
	\end{align*}
	Thus we can conclude that
	\begin{align*}
		M_{1, \pm}\left(W^0, Y\right) & =-\gl \biggl[\frac12\psi W^0+\frac{\gl\theta}{2}\left(1-\frac{\psi }{2}\right)\left(Y^2-\E Y^2\right)+\widetilde{\eps}_{0, \pm}+\widetilde{\eps}_{1, \pm}+\widetilde{\eps}_{2, \pm}+\widetilde{\eps}_{3, \pm}\biggr].
	\end{align*}
	Scaling both sides of equality by $\gs_{W^0}$ yields the result.
\end{proof}

As we discussed in Remark~\ref{rem:changes of var} the change of variable~\eqref{eq:multichange} functions very similarly to its univariate analog~\eqref{eq:change of var}. Thus similar computations to the ones in the proof of Lemma~\ref{lem:1order} yield the following lemma in the multivariate case.

\begin{lem}\label{lem:multi1order}
	Under assumptions of Proposition~\ref{prop: multichange}, recall that we defined
	\begin{align*}
		\mvW^0:=\mvX+\gl \Alpha \mvX Y+\frac{\gl\mvgth}{2}\left(Y^2-\E Y^2\right)+\frac{\gl^2(\Alpha+\ga)\mvgth}{3}(Y^3-\E Y^3).
	\end{align*}
	where, $\gA=\frac{\gl_+-\gl_-}{2Q}$, $\ga=\frac{a_+-a_-}{2Q}$, and $\mvgth=\frac{\mvb_+}{Q}$.
	Then the random vector $\mvW:=\left(W^0_i/\gs_{W^0_i}\right)_{1\le i\le d}$ satisfies Assumption~\ref{ass:M1} and
	\begin{align*}
		 & \widetilde{\mvR}_{1, \pm}=\gS^{-\half}_{W^0}\frac{\gl}{2}\left(I-\frac{\Psi}{2}\right)\mvgth\, \overline{Y^2}+\gS^{-\half}_{W^0}(\widetilde{\mveps}_{0, \pm}+\widetilde{\mveps}_{1, \pm}+\widetilde{\mveps}_{2, \pm}+\widetilde{\mveps}_{3, \pm}),
	\end{align*}
	where
	\begin{align*}
		\widetilde{\mveps}_{0, \pm}                & :=\gl a_\pm\Psi\gA \mvX(Y\pm 1)-\frac{\gl\, \Psi\, \gA }{2} \mvX Y+\biggl(1\pm\gl \, \gA+ \gl\, \gA Y\biggr)\mvR_{1, \pm},  \\
		\widetilde{\mveps}_{1, \pm}                & :=\gl\left(\pm\gl\gA \mvb_\pm\mp\frac{a_\pm\mvgth}{2}+(\gA+\ga)\mvgth\left(Q-\frac13\gl a_\pm\right)\right)Y,               \\
		\widetilde{\mveps}_{2, \pm}                & :=\mp\gl^2 a_\pm(\gA +\ga)\mvgth Y^2,                                                                                       \\
		\text{and}\quad\widetilde{\mveps}_{3, \pm} & :=\frac13\gl(\gA +\ga)\mvgth Q-\gl^2 a_\pm(\gA +\ga)\mvgth Y^3-\frac{\gl^2\Psi (\gA +\ga)\mvgth}{6}\left(Y^3-\E Y^3\right).
	\end{align*}
\end{lem}
\begin{proof} For multivariate case proof is essentially the same as above, but now
	\begin{align*}
		\DW^0 & =\gD \mvX+\gl\gA(\mvX\DY +\gD\mvX Y+\gD\mvX \DY )+\gl\mvgth\left(Y\DY +\frac12\DY ^2\right) \\
		      & \quad+\gl^2(\gA+\ga)\mvgth\DY \left(Y^2+Y\DY +\frac{(\DY )^2}{3}\right).
	\end{align*}
	Notice that $\mvX Y$ have the analogous coefficients to the univariate case and hence functions similarly, \ie~cancels out with $\Psi_\pm \mvX$ and creates $\frac12\Psi\mvX$ term. The $Y$ terms cancel out in the exact the same fashion as above. Now we focus our attention on $Y^2$ terms and derive that
	\begin{align*}
		\gl\left(-\frac14\Psi\mvgth+\gA\mvb_\pm\mp(\gA+\ga)\mvgth Q+a_\pm\mvgth\right) & =\gl\left(-\frac14\Psi\mvgth+\gA\left(\mvb_\pm\mp\mvgth Q\right)+\mvgth(a_\pm\pm\ga Q)\right) \\
		                                                                               & \quad=\left(I-\frac{\Psi}{2}\right)\frac{\gl\mvgth}{2}Y^2,
	\end{align*}
	and notice the the constant term matches it
	\begin{align*}
		-\frac{\mvgth}{2}Q+\gl\Psi\frac{\mvgth}{4}\E Y^2=-\left(I-\frac{\Psi}{2}\right)\frac{\gl\mvgth}{2}\E Y^2.
	\end{align*}
	Notice that by Assumption~\ref{ass:M1g} $\mvb_+=\frac{1}{2}\Psi\mvb_+$ and hence
	\[\left(I-\frac{\Psi}{2}\right)\mvgth=0.\]
	Therefore we conclude that
	\begin{align*}
		M_{1, \pm}\left(\mvW^0, Y\right) & =-\gl \biggl[\frac{\Psi}{2}\mvW^0+\widetilde{\mveps}_{0, \pm}+\widetilde{\mveps}_{1, \pm}+\widetilde{\mveps}_{2, \pm}+\widetilde{\mveps}_{3, \pm}, \biggr].
	\end{align*}
	Scaling each coordinate of $\mvW^0$ appropriately yields the result.
\end{proof}

In order to prove Propositions~\ref{prop: change} and~\ref{prop: multichange}, it remains to show that the order of the second-order error terms does not change after the change of variables as given in~\eqref{eq:change of var} and~\eqref{eq:multichange}. We present the proof for the univariate case in the following lemma. The treatment of the multivariate case is completely analogous.

\begin{lem}\label{lem:2order}
	Under the assumptions of Proposition~\ref{prop: change} $\left(W, Y\right)$ still satisfy the Assumption~\ref{ass:M2}. Moreover,
	\begin{align*}
		\norm{\widetilde{R}_{2, \pm}-{R_{2, \pm}}/{\gs^2_{W^0}}}_{2p}
		 & \lesssim \gl\psi \abs{\ga}\left(\frac{\gs_{X}}{\gs_{W^0}}\sqrt{\norm{\gl(1+R_{2, \pm}/\gs^2_{X})}_{2p}\norm{Y}_{2p}}+\frac{\norm{X}_p}{\gs_{W^0}}\right) \\
		 & \qquad+\gl\abs{\theta}\cdot\frac{\norm{Y}_p}{\gs_{W^0}}
		+\gl(\psi +1)\abs{\ga\theta}\cdot\frac{\gl \norm{Y^2}_p}{\gs_{W^0}}.
	\end{align*}
\end{lem}
\begin{proof}[Proof of Lemma~\ref{lem:2order}]
	Rewrite the equality in~\eqref{eq:deltaW} as
	\begin{equation}\label{eq:deltaW2}
		\gD W^0 =\gD X+\gd,
	\end{equation}
	where
	\begin{align*}
		\gd & :=\gl \psi \ga\left(\gD XY+\DY X+\gD X\DY \right)+\gl \, \theta\left(Y\DY +\frac12 \DY ^2\right)\notag \\
		    & \quad+\gl^2(\psi +1)\ga\theta\DY \left(Y^2+Y\DY +\frac{(\DY )^2}{3}\right).
	\end{align*}
	Multiplying both sides of~\eqref{eq:deltaW2} by $\1_{\DY =\pm1}$, taking conditional expectation given $(X, Y)$ on both sides, and using triangle inequality implies that
	\begin{align*}
		 & \abs{\sqrt{\E\left(\abs{\gD W^0}^2\1_{\DY =\pm1}\mid X, Y\right)}-\sqrt{\E\left(\abs{\gD X}^2\1_{\DY =\pm1}\mid X, Y \right)}} \\
		 & \qquad\leq\sqrt{\E\left(\abs{\gd}^2\1_{\DY =\pm1}\mid X, Y \right)}
	\end{align*}
	By Assumption~\ref{ass:M2} we can rewrite it as
	\begin{equation}\label{eq:r2bdd}
		\frac{\gl\gs^2_{W^0}}{2+\frac{R_{2, \pm}}{\gs^2_{W^0}}+\frac{\gs_X^2}{\gs^2_{W^0}}}\cdot\abs{\widetilde{R}_{2, \pm}-\frac{R_{2, \pm}}{\gs^2_{W^0}}}^2\leq
		\E\left(\abs{\gd}^2\1_{\DY =\pm1}\mid X, Y \right)
	\end{equation}
	Assuming $2+\frac{R_{2, \pm}}{\gs^2_{W^0}}+\frac{\gs_X^2}{\gs^2_{W^0}}\in(1, C)$ for some constant $C$, it remains to derive the following bound
	\begin{align*}
		 & \norm{\E\left((\gd/\gs_{W^0})^2\1_{\DY =\pm 1}\mid X, Y\right)}^{\half}_p                                                                                     \\
		 & \quad\lesssim \gl\psi \abs{\ga}\left(\frac{\gs_{X}}{\gs_{W^0}}\sqrt{\norm{\gl(1+R_{2, \pm}/\gs^2_{X})}_{2p}\norm{Y}_{2p}}+\frac{\norm{X}_p}{\gs_{W^0}}\right) \\
		 & \qquad\quad+\gl\abs{\theta}\cdot\frac{\norm{Y}_p}{\gs_{W^0}}
		+\gl(\psi +1)\abs{\ga\theta}\cdot\frac{\gl \norm{Y^2}_p}{\gs_{W^0}}.
	\end{align*}
	From~\eqref{eq:r2bdd} we can see that
	\[
		\norm{\widetilde{R}_{2, \pm}-\frac{R_{2, \pm}}{\gs^2_{W^0}}}_{2p}
		\lesssim \gl^{-\half} \norm{\E\left((\gd/\gs_{W^0})^2\1_{\DY =\pm 1}\mid X, Y\right)}^{\half}_p.
	\]
	This completes the proof.
\end{proof}

\section{Closing remarks and further work}\label{ssec:closing}

Before our work, few results delve into CCLT in general settings. There are many aspects in which we would like to see our approach extended and improved, including generalizing the current approach, extending it to other dependency structures between random variables, and connecting it to concentration inequalities. In this article, we often utilize subtle, sometimes surprising, cancellations caused by exchangeability. Hence we believe that there is much more to understand in this area than we presently know. We discuss our results and possible future directions and state questions of particular interest in the remaining of this section.

\subsection{Change of variable and the assumptions}

For our main result, we require random variables to satisfy symmetric linearity conditions (Assumption~\ref{ass:M1}), and the second-moment condition (Assumption~\ref{ass:M2}). It is important to notice that these assumptions together with Assumption~\ref{ass:Y} yield linearity condition for $(\mvW, Y/\gs_Y)$. Hence if $R_{i, \pm}$ are small for $i=0, 1, 2$ the multivariate Stein's method for exchangeable pair implies joint convergence to a $d+1$-dimensional Gaussian vector.

Given the non-symmetric linearity condition, one can make it symmetric by subtracting the product of random variables with an appropriate coefficient (the $XY$ term in the change of variable). When working with counting random variables, it is often the case that the linearity condition is of the form $\gl(\psi a_\pm W+b_\pm Y)$ (Assumption~\ref{ass:M1g}); as we have seen in Section~\ref{ssec:pattern} on the sub-pattern count example and in Section~\ref{ssec:wedgeedge} on the sub-graph count example. However, in such examples, $M_{0, \pm}$ also has a particular form that we state in the form of the Assumption~\ref{ass:Y}. We utilized this fact to a great extent in the change of variable and the proof of Theorem~\ref{th: imp symcase} and~\ref{th: multi case}, leading to the following question.

\begin{question}
	Is it possible to derive CCLT in the models with $|\DY |=1$ that satisfy Assumption~\ref{ass:M1g} but $M_{0, \pm}\approx Q+ f_{\pm}(Y)$ for non-linear $f$?
\end{question}

We believe that in such cases, our approach is still applicable, and if the conditional mean is not known, its approximation in the form of change of variable would be different, but analogous, from the one presented in~\eqref{eq:change of var} and~\eqref{eq:multichange}.

\subsection{Range of $\DY $}\label{ssec:rangeY}
In this article, we focused on the case where $|\DY | \le1$. When $Y$ can change by more than $\pm1$ but $\pr(|\DY |=1)$ is sufficiently large, conditioning on this event, one can still apply our results similar to the way how we used the classical methods by conditioning on the event $\{\DY =0\}$ in the Section~\ref{ssec:applications of classic}. However, in the models where $\pr(|\DY |=1)$ is negligible, it remains open to extending our approach.

\begin{question}
	Is it possible to derive CCLT with an explicit convergence rate for models where $\DY $ can take infinitely many values?
\end{question}

For example, consider the number of edges given the number of triangles in the Erd\H{o}s--R\'enyi random graph. Another immediate application would be the extension of~\cite[Theorem 4]{Bolthausen80} and~\cite[Theorem 4.2]{Peterson19}, for $\gd=1$, to the dependent settings.

\subsection{Other types of events than $\{Y=k\}$}

In this article, we condition on the simplest type of the event $\{Y=k\}$ where $Y$ is the sum of indicators, as it already required a considerable amount of effort and novel techniques. The next step is to extend our result to CCLT where one conditions on a vector $\{(Y_1, Y_2, \ldots, Y_d)=(y_1, y_2, \ldots, y_d)\}$ sill under the assumption of joint Gaussian convergence. A natural application would be a joint CCLT for a sub-graph count given a value for several other sub-graph counts in the Erd\H{o}s--R\'enyi random graph.

\begin{question}
	How does the rate of convergence depend on the dimension of the vector $\mvY=(Y_1, Y_2, \ldots, Y_d)$?
\end{question}

Another, the more complicated direction, is to relax the assumption of joint Gaussian convergence and condition on more complicated events such as some property of the trajectory of a random walk or random environment.

\subsection{Sufficient statistic}

In most of the applications presented in this paper, we condition on the sufficient statistic. We believe our theorem should reliably work without this condition, as we demonstrated in Lemma~\ref{lem:11dif}. However, while writing this paper, we realized that we do not know of many natural examples satisfying our assumptions that will also not be sufficient statistics for the parameters of the model, especially in the univariate case. If one relaxes our condition on $\DY $ and allows it to range over an infinite set, this will create many natural examples with conditioning on non-sufficient statistics.

Besides the application presented in Section~\ref{ssec:even odd}, where $Y$ was not a sufficient statistic for the parameter $p$, another application could be a ``noisy" version of a sufficient statistic. For example, suppose we would like to condition on the number of edges in an inhomogeneous random graph where each of $ N(1-\gd)$ edges is present independently with probability $p$ and each of the remaining $\gd N$ edges is present with probability $p\pm\eps$ independently from everything else. One can estimate the difference between this model and the homogeneous random graph and work with the latter. However, depending on the $\gd$ and $\eps$, the error might be significant. Our approach provides an alternative that allows working directly with inhomogeneous models.

\subsection{Other approaches}

It is of interest to extend other existing approaches of Stein's method to the conditional setting, particularly the dependency graph approach in relation to the subgraph counting problem (See Section \ref{ssec: general subgraph}). Since, in this work, we assume that $(W, Y)$ jointly converges to a Gaussian vector, a natural place to start would be adopting a multivariate extension of the dependency graph approach such as in~\cite{fang16} or some other variation of the method to the conditional setting.

\begin{question}\label{q: locdep}
	Is it possible to get the rate of converges of order $n^{-1}$ in Theorem~\ref{th: gengraph} using another approach than the exchangeable pair to match the bound in Corollary~\ref{cor:barbour}?
\end{question}

\subsection{Other distances between distributions}
In this article, we use Wasserstein-$1$ distance to quantify the rate of convergence. Suppose one is interested in bounding other metrics, such as Kolmogorov-Smirnov distance. Then one has to work with functions with fewer derivatives. In that case, one can usually approximate those functions by two or three times differentiable functions to apply our techniques. However, that would result in an extra loss in the rate of convergence. Thus it remains open to acquiring optimal bounds on the convergence rates in other distances directly if at all possible.

\begin{appendix}

	\section{From Smooth functions to Lipschitz functions}\label{sec:lip}

	\begin{lem}\label{lem:s2l}
		Let $\mvW$ be a $d$-dimensional random vector with mean zero and $\E\abs{\mvW}^{4}<\infty$, such that for some event $B$ independent of $\mvZ$ with $\pr(B)>0$, we have
		\[
			\abs{ \E(\sS f(\mvW)\mid B) } \le A_1\cdot M_1(f) + A_2\cdot \sup_w \norm{\hess f(w)}_{\hs} +A_4\cdot M_4(f)
		\]
		for all $f$ with $M_{4}(f)<\infty$. Then, for all Lipschitz functions $h$ with $M_{1}(h)<\infty$, we have
		\begin{align}\label{eq:lip4a}
			\begin{split}
				&\abs{\E((h(\mvW)-h(\gS^{\half}\mvZ)\mid B)}\\
				&\qquad \le M_1(h)\left(A_1+\norm{\gS^{-1/2}}_{\op}\cdot A_2 +\norm{\gS^{-1/2}}_{\op}^{3/2}\cdot \sqrt{4 c_3 A_0A_4} \right)
			\end{split}
		\end{align}
		where $c_3=(2+8e^{-3/2})/\sqrt{2\pi}<2$, $\mvZ\sim N_{d}(\mvzero, I_{d})$ is independent of $\mvW$ and
		$$
			A_{0}:= \sqrt{\E (\abs{\mvW}^{2} + \tr(\gS)\mid B)}.
		$$
		Similarly, if $\E\abs{\mvW}^{3}<\infty$ and
		\[
			\abs{ \E(\sS f(\mvW)\mid B) } \le A_1\cdot M_1(f) + A_2\cdot \sup_w \norm{\hess f(w)}_{\hs} +A_3\cdot M_3(f),
		\]
		for all $f$ with $M_{3}(f)<\infty$,
		we have
		\begin{align}\label{eq:lip3a}
			\begin{split}
				&\abs{\E((h(\mvW)-h(\gS^{\half}\mvZ)\mid B)}
				\le M_{1}(h)\cdot \biggl( A_{1} + \norm{\gS^{-1/2}}_{\op}\cdot A_2 \\
				&\qquad\qquad+ c_{2}\cdot \norm{\gS^{-1/2}}_{\op}^{2}\cdot A_{3} \cdot \biggl(1+\abs{\log\bigl({c_{2}\norm{\gS^{-1/2}}_{\op}^{2}\cdot A_{3}}/{A_0}\bigr)}\biggr)\biggr).
			\end{split}
		\end{align}
		where $c_{2}=4/\sqrt{2\pi e}<1.$
	\end{lem}
	\begin{proof}[Proof of Lemma~\ref{lem:s2l}]
		We mainly follow the proof of~\cite[Theorem 1.1]{FangKoike22} along with estimates from~\cite{Raic19}. Consider a Lipschitz function $h:\dR^{d}\to\dR$ and for all $\alpha\in [0, \pi/2]$, define the function
		\[
			h_{\ga}(\mvw):=\E h(\cos\ga\cdot \mvw + \sin\ga\cdot \gS^{\half}\mvZ), \qquad \mvw\in\dR^{d}
		\]
		where $\mvZ$ is a $d$-dimensional standard Gaussian vector. Clearly $h_{0}=h$ and $h_{\pi/2}\equiv \E h(\gS^{\half}\mvZ)$. It is easy to check that $h_{\ga}$ is infinitely differentiable. Using~\cite[Lemma 4.7]{Raic19}, we get that
		\begin{align*}
			M_{r+1}(h_{\ga})
			\le c_{r}\cdot \frac{\cos^{r+1}\ga}{\sin^{r}\ga}\cdot M_{1}(h)\cdot \norm{\gS^{-1/2}}_{\op}^{r}
		\end{align*}
		for any non-negative integer $r$ and $\ga\in (0, \pi/2)$, where $c_{r}$ is as given in~\cite[Eqn.~4.9]{Raic19}; in particular $$c_{0}=1, c_{2}=4/\sqrt{2\pi e}<1, c_{3}=(2+8e^{-3/2})/\sqrt{2\pi}<2.$$
		Moreover, by~\cite[Lemma 2.2(i)]{ChatterjeeMeckes08} we get
		\begin{align*}
			\sup_{\mvw}\norm{\hess h_{\ga}(\mvw)}_{\hs} \le \frac{\cos^2\ga}{\sin\ga}\cdot M_{1}(h)\cdot \norm{\gS^{-1/2}}_{\op}.
		\end{align*}
		Now, using the fact that $\frac{d}{d\ga}h_{\ga}(\mvw)=\sS h_{\ga}(\mvw)\ \tan\ga$, we get that
		\begin{align}\label{eq:spl1}
			\begin{split}
				&\abs{\E(h(\mvW)-h(\gS^{\half}\mvZ)\mid B)} \\
				&=\abs{\E(h_{0}(\mvW)-h_{\pi/2}(\mvW)\mid B)}\\
				&\le \abs{\E(h_{0}(\mvW)-h_{\eps}(\mvW))\mid B)} + \int_{\eps}^{\pi/2} \abs{\E(\sS h_{\ga}(\mvW)\mid B)}\tan\ga\, d\ga.
			\end{split}
		\end{align}
		We have for any $\eps\in(0, \pi/2)$
		\begin{align*}
			\abs{\E(h_{0}(\mvW)-h_{\eps}(\mvW)\mid B)}
			 & \le M_{1}(h) \E(\abs{(1-\cos\eps)\cdot \mvW + \sin\eps\cdot \gS^{\half}\mvZ}\mid B)                               \\
			 & \le M_{1}(h) \sqrt{\E ((1-\cos\eps)^{2}\cdot \abs{\mvW}^{2} + \sin^{2}\eps\cdot \abs{\gS^{\half}\mvZ}^{2}\mid B)} \\
			 & \le M_{1}(h)\cdot A_{0}\cdot \sin\eps
		\end{align*}
		where
		\[
			A_{0}:=\sqrt{\tr(\gS)+ \E (\abs{\mvW}^{2}\mid B)}.
		\]
		Here we used the fact that $(1-\cos\eps)/\sin\eps=\tan(\eps/2)\le 1$ for $\eps\in(0, \pi/2)$.
		Combining with equation~\eqref{eq:spl1} and integrating, when $\E\abs{\mvW}^{4}<\infty$, we get that
		\begin{align*}
			 & \abs{\E(h(\mvW)-h(\gS^{\half}\mvZ)\mid B)}                                                                                                                                     \\
			 & \le M_{1}(h)\cdot \left( A_{0}\cdot \sin\eps + c_{0}\cdot A_{1} + \norm{\gS^{-1/2}}_{\op}\cdot A_2 + c_{3}\cdot \norm{\gS^{-1/2}}_{\op}^{3}\cdot A_{4}\cdot \frac{1}{\sin\eps}
			\right)
		\end{align*}
		for any $\eps\in (0, \pi/2]$. Taking $\sin^{2}\eps = \min(c_{3}A_{4}\norm{\gS^{-1/2}}_{\op}^{3}/A_{0}, 1)$, we get
		\begin{align*}
			 & \abs{\E(h(\mvW)-h(\gS^{\half}\mvZ)\mid B)}                                                                                      \\
			 & \le M_{1}(h)\cdot \left( A_{1} + \norm{\gS^{-1/2}}_{\op}\cdot A_2 + \norm{\gS^{-1/2}}_{\op}^{3/2}\cdot \sqrt{4c_{3}A_{0} A_{4}}
			\right).
		\end{align*}
		Note that, when $c_{3}A_{4}\norm{\gS^{-1/2}}_{\op}^3\ge A_{0}$, we can directly use the upper bound $$M_{1}(h)\cdot A_{0}\sin(\pi/2) \le M_{1}(h)\cdot \norm{\gS^{-1/2}}_{\op}^{3/2}\cdot \sqrt{4c_{3}A_{0} A_{4}}.$$

		When, $\E\abs{\mvW}^{3}<\infty$, we similarly get
		\begin{align*}
			 & \abs{\E(h(\mvW)-h(\gS^{\half}\mvZ)\mid B)}                                                                                                                         \\
			 & \le M_{1}(h)\cdot \left( A_{0}\cdot \sin\eps + A_{1} + \norm{\gS^{-1/2}}_{\op}\cdot A_2 + c_{2}\cdot \norm{\gS^{-1/2}}_{\op}^{2}\cdot A_{3} \cdot \log(1/\sin\eps)
			\right)
		\end{align*}
		for any $\eps\in (0, \pi/2]$. Here we use the fact that
		\[
			\int_{\eps}^{\pi/2} \frac{\cos^{2}\ga}{\sin\ga}d\ga \le \int_{\sin\eps}^{1}\frac{dt}{t}= \log\frac1{\sin\eps}.
		\]
		Choosing $\sin\eps=\min(c_{2}\cdot \norm{\gS^{-1/2}}_{\op}^{2}\cdot A_{3}/A_{0}, 1)$, we get the bound
		\begin{align*}
			 & \abs{\E(h(\mvW)-h(\gS^{\half}\mvZ)\mid B)}
			\le M_{1}(h)\cdot \biggl( A_{1} + \norm{\gS^{-1/2}}_{\op}\cdot A_2                                                                                                                \\
			 & \qquad\qquad\qquad + c_{2}\cdot \norm{\gS^{-1/2}}_{\op}^{2}\cdot A_{3} \cdot \biggl(1+\abs{\log\bigl({c_{2}\norm{\gS^{-1/2}}_{\op}^{2}\cdot A_{3}}/{A_0}\bigr)}\biggr)\biggr).
		\end{align*}
		This completes the proof.
	\end{proof}

	The following lemma shows that under LLT $|p_k-p_{k-1}|\ll p_k$, and in particular, $r_{k}=p_{k-1}/p_{k}$ is bounded away from $\infty$, for $|k|\ll \gs_Y$.

	\begin{lem}\label{lem:lltratio}
		Suppose $Y$ is a random variable that take values in $\gz+\dZ$ for some $\gz\in[0, 1)$. Assume $Y$ has mean $0$ and variance $\gs_Y^2$ and is such that for all $k\in\gz+\dZ$
		\begin{equation*}
			\abs{\gs_Y p_k-\gs_Y\varphi_{\gs_Y^2}(k)}\leq \eps_Y<1,
		\end{equation*}
		where $p_k:=\pr(Y=k)$ and $\eps_Y$ is some function. Then, for $|k|\ll \gs_Y$
		\begin{align*}\label{eq:pkratio}
			\abs{1-p_{k-1}/p_k}\lesssim (1-\eps_{Y})^{-1}\cdot \max\left\{\eps_Y, {k}/{\gs_Y^2}\right\}.
		\end{align*}
	\end{lem}
	\begin{proof}
		By LLT we have that
		\begin{equation*}
			\abs{\gs_Y p_k-\frac{1}{\sqrt{2\pi}}\exp\left(-{k^2}/{2\gs_Y^2}\right)}\leq \eps_{Y},
		\end{equation*}
		where $\eps_{Y}=o(1)$ as $n\to \infty$.
		Thus,
		\begin{align*}
			\abs{1-\frac{p_{k-1}}{p_k}}
			 & \leq\frac{\abs{\exp\left(-{k^2}/{2\gs_Y^2}\right)-\exp\left(-{(k-1)^2}/{2\gs_Y^2}\right)}}{\sqrt{2\pi}\gs_Y p_k}+\frac{2\eps_{Y}}{\gs_Y p_k}\\
			 &
			\lesssim (1-\eps_{Y})^{-1}\cdot \max\left(\eps_{Y}, {k}/{\gs_Y^2}\right)
		\end{align*}
		for $|k|\ll \gs_{Y}$
		and the proof is complete.
	\end{proof}

	\section{Computations for Lemma~\ref{lem:wedgeedge}}\label{sec: wedge comp}
	Let $\cH$ be a graph on $v$ vertices and $m$ edges, and let $H$ be the centered random variable that counts the number of its copies in $G_{n, p}$. Let $\mvs$ denote a subset of edges in the complete graph $\cK_n$ on $n$ vertices that form an isomorphic copy of $\cH$ and $\cS=\cS_n(H)$ denote the collection of all such $\mvs$'s. Define $|\cS|$ to be the size of the collection $\cS$. Thus we can rewrite the sub-graph count as
	$$
		H=\sum_{\mvs\in\cS}\prod_{e\in\mvs}\go_e=\sum_{\mvs\in\cS}\go_{\mvs},
	$$
	It is often more convenient to work with centered edges as it makes as we did in~\eqref{eq:triangle count} and~\eqref{eq:wedge count}. Thus define
	$$\widetilde{H} :=\sum_{\mvs\in\cS}\prod_{e\in\mvs}\bgo_e=\sum_{\mvs\in\cS}\widetilde{\go}_{\mvs}.$$
	\begin{lem}\label{lem:graph var}
		With the notations as above for fixed $p\in (0, 1)$ we have that $$\var{\widetilde{H}}=O\left(n^{v}\right).$$
	\end{lem}
	\begin{proof}
		Given two $\mvs$ and $\mvs'$ from $\cS$. The expectation $\E\widetilde{\go}_{\mvs}\widetilde{\go}_{\mvs'}$ is nonzero if and only if each edge appears twice in the product. Since $\mvs$ and $\mvs'$ induce isomporphic graphs this yields that $\mvs$ has to be equal to $\mvs'$. Thus $V(\mvs\cup\mvs')=v$ and thus the variance of $\widetilde{H}=O(n^v)$.
	\end{proof}

	In this section, we present explicit computations that we used in the derivation of the CCLT for the number of wedges given the number of edges in a random graph in Lemma~\ref{lem:wedgeedge}.
	Recall that $G_{n, p}$ is the Erd\H{o}s--R\'enyi random graph and $q:=1-p$.
	Let $E =\sum_{x<y}\1_{x\sim y}=\sum_{x<y}\go_{xy}$ be the number of edges in $G_{n, p}$, and
	$V =\sum_{x<y, z\neq x, y}\go_{xy}\go_{yz}$ be the number of wedges. Define $X =V -2(n-2)pY -(n-2)\binom{n}{2}p^2$ and $Y =E -\E E $. Finally recall we place bar above the random variable to denote the centered version of it.

	It is straight forward to check that $X$ is uncorrelated with $Y$ and could be rewritten as
	\begin{align*}
		\sum_{i<j, k\neq i, j}\bgo_{ik}\bgo_{kj}
		 & =\sum_{i<j, k\neq i, j}(\go_{ik}-p)(\go_{kj}-p)                   \\
		 & =\sum_{i<j, k\neq i, j}(\go_{ik}\go_{kj}-p\go_{ik}-p\go_{jk}+p^2) \\
		 & =V -2(n-2)pY -(n-2)\binom{n}{2}p^2=X.
	\end{align*}
	In spirit of Lemma~\ref{lem:graph var}, this representation is particularly helpful in computing the variance of $X$
	$$\gs_X^2=\binom{n}{2}(n-2)p^2q^2=\frac12 n(n-1)(n-2)p^2(1-p)^2.$$
	To compute the first order terms we also work with centered edges and derive that
	\begin{align*}
		M_{1, +}(X, Y)
		 & =\frac{p}{ N}\sum_{i<j}(q-\bgo_{ij}) \left(\sum_{k\neq i, j} \bgo_{ik}+\bgo_{jk}\right) \\
		 & =-\frac{p}{ N}\sum_{i\neq j}(\bgo_{ij} -q ) \sum_{k\neq i\neq j} \bgo_{ik}              \\
		 & =-\frac{p}{ N}
		\sum_{i\neq j \neq k}\bgo_{ij}\bgo_{ik}
		+\frac{pq}{ N}\sum_{i\neq k \neq j} \bgo_{ik}
		=-\frac{2}{ N}(pX-pq(n-2)Y).
	\end{align*}
	Similarly
	\begin{align*}
		M_{1, -}(X, Y)
		 & =\frac{q}{ N}\sum_{i<j}(\bgo_{ij}+p) \left(-\sum_{k\neq i, j} \bgo_{ik}+\bgo_{jk}\right)=-\frac{2}{ N}(qX+pq(n-2)Y).
	\end{align*}
	For the second order terms we compute
	\begin{align*}
		M_{2, +}(X, Y)
		 & =\frac{p}{ N}\sum_{i<j}(q-\bgo_{ij}) \left(\sum_{k\neq i, j} \bgo_{ik}+\bgo_{jk}\right)^2                                                         \\
		 & = \frac{p}{ N}\sum_{i<j}(q-\bgo_{ij}) \left(\overline{d}_i + \overline{d}_j - 2\bgo_{ij} \right)^2                                                \\
		 & = \frac{p}{ N}\sum_{i\neq j}(q-\bgo_{ij}) \left(\overline{d}_i^2 +\overline{d}_i\overline{d}_j + 2\bgo_{ij}^2 - 4\bgo_{ij} \overline{d}_i \right)
	\end{align*}
	Letting
	\[
		I_{ij}:=\overline{d}_i^2 +\overline{d}_i\overline{d}_j + 2\bgo_{ij}^2 - 4\bgo_{ij} \overline{d}_i,
	\]
	we get that
	\[
		M_{2, +}(X, Y)=\frac{pq}{ N}\sum_{i\neq j} I_{ij}-\frac{p}{ N}\sum_{i\neq j}\bgo_{ij}I_{ij}.
	\]
	The second term counts centered-edge graphs that have at most four distinct vertices and hence by Lemma~\ref{lem:graph var}
	\[
		\E\abs{\sum_{i\neq j}\bgo_{ij}I_{ij}}\lesssim n^2.
	\]
	Letting $\overline{d}_i:=\sum_{j\neq i}\bgo_{ij}$, the first summand gives the variance terms required in the Assumption~\eqref{ass:M2} in the following way
	\begin{align*}
		\frac{pq}{ N}\sum_{i\neq j} I_{ij}
		 & =\frac{pq}{ N}\sum_{i\neq j} \left(\overline{d}_i^2 +\overline{d}_i\overline{d}_j + 2\bgo_{ij}^2 - 4\bgo_{ij} \overline{d}_i\right) \\
		 & =\frac{pq}{ N} \left((n-6)\sum_i\overline{d}_i^2+\left(\sum_i\overline{d}_i\right)^2 +4\sum_{i<j}\bgo_{ij}^2 \right)                \\
		 & =\frac{2}{ N}\left(\frac{n(n-1)(n-2)}{2}p^2q^2+(n-6)pqX+R_{2, +}\right),
	\end{align*}
	where
	\[
		R_{2, +} := pq\left( 2\left(E^2- Npq\right) +(n-4)\sum_{i<j}(\bgo_{ij}^2-pq)\right).
	\]
	with
	$\E|R_{2, +}|^2\lesssim n^5.$ $M_{2, -}$ is treated similarly as one can rewrite in as
	\[
		M_{2, -}(X, Y)=\frac{pq}{ N}\sum_{i\neq j} I_{ij}+\frac{q}{ N}\sum_{i\neq j}\bgo_{ij}I_{ij}.
	\]
	The third order terms we can bound by
	\begin{align*}
		\gl^{-1}\E(\abs{\gD W}^3\mid W, Y) & \leq\E(\abs{\gD W}\abs{\gD W}^2\mid W, Y)                                                       \\
		                                   & \lesssim \gl^{-1}\frac{n-2}{n^{3/2}}\gl(\psi+\abs{R_{2, +}}+\abs{R_{2, -}})\lesssim n^{-\half}.
	\end{align*}

\end{appendix}



\begin{acks}[Acknowledgments]
	We thank Persi Diaconis and Jonathon Peterson for their insightful comments and for pointing out the existing literature. We also thank two anonymous referees for careful reading, which resulted in the improved presentation of the article, and the suggestion to adopt Rai\v{c}'s and Fang--Koike's results to derive the Wasserstein distance bound in the multivariate case. We further thank Felix Christian Clemen, Gleb Chernov, Kesav Krishnan, Charlie Terlov, and Anush Tserunyan for many enlightening discussions.
\end{acks}

%



\bibliographystyle{imsart-nameyear} 
\bibliography{csm.bib}

\end{document}